\documentclass[10pt]{article}

\usepackage{amssymb,latexsym,amsmath,enumerate,verbatim,amsfonts,amsthm, mathabx}
\usepackage{color} 
\usepackage{algorithm}
\usepackage{mathtools}
\usepackage{caption}
\usepackage{subcaption}

\usepackage{url}
\usepackage[pdftex,
plainpages=false,
bookmarks,
bookmarksnumbered,
colorlinks=true,
linkcolor=blue,
citecolor=blue,
urlcolor=blue,
filecolor=black,
hypertexnames=false
]
{hyperref}
\usepackage[colorinlistoftodos,prependcaption,textsize=footnotesize]{todonotes}
\usepackage[framemethod=tikz]{mdframed}
\mdfsetup{%
  linewidth=1.25pt,
  backgroundcolor=gray!5,
  userdefinedwidth=\textwidth,
  roundcorner=10pt,
}

\newtheorem{lemma}{Lemma}[section]
\newtheorem{definition}[lemma]{Definition}

\newtheorem{assumption}[lemma]{Assumption}
\newtheorem{theorem}[lemma]{Theorem}
\newtheorem{corollary}[lemma]{Corollary}
\newtheorem{proposition}[lemma]{Proposition}
\newtheorem{remark}[lemma]{Remark}
\newtheorem{example}[lemma]{Example}

\def\RV{{\rm RV}}
\def\RVz{{\rm RV}^0}
\def\R{{\rm I\!R}}
\def\dist{{\rm dist}}
\def\inf{{\rm inf}}

\def\Argmax{\mathop{\rm Arg\,max}}

\def\argmin{\mathop{\rm arg\,min}}
\newcommand{\expCone}{\ensuremath{K_{\exp}}}
\newcommand{\stdCone}{ {\mathcal{K}}}

\newcommand{\stdAffine}{ \mathcal{V}}
\newcommand{\stdFace}{ \mathcal{F}}

\newcommand{\spanVec}{\mathrm{span}\,}
\newcommand{\inProd}[2]{\langle #1 , \, #2 \rangle }
\newcommand{\norm}[1]{\|#1\|}
\newcommand{\reInt}{\mathrm{ri}\,}

\newcommand{\dpp}{d_{\text{PPS}}}
\newcommand{\ds}{d_{\text{S}}}
\newcommand{\comp}{\diamondsuit}
\newcommand{\frakg}{\ensuremath{\mathfrak{g}}}
\DeclarePairedDelimiter{\ceil}{\lceil}{\rceil}
\DeclareMathOperator{\matRank}{rank}

\title{\sf  Convergence analysis under consistent error bounds}
\author{Tianxiang Liu\thanks{School of Computing, Tokyo Institute of Technology,  Japan. (\href{liu.t.af@m.titech.ac.jp}{liu.t.af@m.titech.ac.jp})} \and Bruno F. Louren\c{c}o\thanks{Department of Statistical Inference and Mathematics, Institute of Statistical Mathematics, Japan. (\href{bruno@ism.ac.jp}{bruno@ism.ac.jp})} }
\numberwithin{equation}{section}
\allowdisplaybreaks
\begin{document}
\maketitle

\begin{abstract}
We introduce the notion of \emph{consistent error bound functions} which provides a unifying framework for error bounds for multiple convex sets. This framework  goes beyond the classical Lipschitzian and H\"olderian error bounds and includes logarithmic and entropic error bounds  found in the exponential cone. It also includes the error bounds obtainable under the theory of amenable cones. 
Our main result is that the convergence rate of several projection algorithms for feasibility problems can be expressed explicitly in terms of the underlying consistent error bound function. 
Another feature is the usage of {Karamata theory} and functions of regular variations which allows us to reason about convergence rates while bypassing certain complicated expressions.
Finally, applications to conic feasibility problems are  given and we show that a number of algorithms have convergence rates depending explicitly on the singularity degree of the problem.
\end{abstract}

\section{Introduction}

In this paper, we consider the following convex feasibility problem (CFP)  
\begin{equation}\label{CFP}
{\rm find}\ x\in C: = \bigcap_{i = 1}^mC_i, \tag{CFP}
\end{equation}
where $C_1, \cdots, C_m$ are closed convex sets contained in a finite dimensional real vector space $\mathcal{E}$ with $C\neq\emptyset$. Convex feasibility problems have been extensively studied in connection to various applications, see \cite{AC89,BB96,BLY14,C97,DLW17,NRP19}. 
Then, given some fixed algorithm for solving \eqref{CFP}, the following two 
questions are of natural interest.
\begin{enumerate}[$(1)$]
	\item Does the algorithm converge to a point in $C$?
	\item If it indeed converges, how fast is the convergence?
\end{enumerate}
For question (1), convexity ensures that many algorithms converge without  further assumptions on the $C_i$, see, for example, section~3 of \cite{BB96} and \cite{BT03}.
On the other hand, the answer to question (2) does not generally follow from convexity alone. 

In order to pin down the convergence rate, in many cases it is necessary to assume 
that some \emph{error bound} is known. 
Informally, an  error bound is some inequality that relates 
the individual distances to the sets $C_i$  to the distance to their intersection $C$.
For more information on error bounds in general settings, see \cite{Pang97,LP98}.

We now present a simple example of error bound.
Given $x \in \mathcal{E}$, let $\dist(x,\, C_i)$ denote the distance from $x$ to $C_i$.
Suppose that, for every bounded set $B\subseteq\mathcal{E}$, there exists some $\theta_B > 0$  such that
\begin{equation}
\dist(x,\, C) \le \theta_B\max_{1\le i\le m}\dist(x,\, C_i) \ \ \ \forall\ x\in B. \label{eq:leb}
\end{equation}
In this case, we say that a (local) \emph{Lipschitzian error bound} holds for \eqref{CFP}.
The property given in \eqref{eq:leb} is also called \emph{bounded linear regularity}, see \cite{BBL99}.
Under \eqref{eq:leb}, many common projection methods are known to converge linearly, see \cite{BB96,BT03}. 

If we replace the $\dist(x, C_i) $ by $\dist(x, C_i)^\gamma$ in  \eqref{eq:leb} for some $\gamma \in (0,1]$, we obtain what is called a \emph{H\"olderian error bound}. 
H\"olderian error bounds typically hold under milder conditions than Lipschitzian bounds, although it might be hard to estimate the exponent $\gamma$. 
A notable exception is the H\"olderian error bound 
by Sturm for semidefinite programs \cite{S00}, where the exponent can be, in principle, computed via  a technique called \emph{facial reduction}.

H\"olderian bounds usually only lead to sublinear convergence rates, with the precise rate often depending on the exponent, e.g.,
Corollary~4.6 in \cite{BLY14}. 
It might be fair to say that results such as this are rarer in comparison to convergence rates obtained under \eqref{eq:leb}. Beyond H\"olderian bounds there are even fewer results.

In this paper, we take a bird's eye view and propose the  notion of \emph{consistent error bound functions} (see Definition~\ref{GEB}) which provides a unifying framework for error bounds. Informally, a consistent error bound function is a two-parameter function $\Phi$ satisfying some reasonable properties and the following error bound condition
\begin{equation}\label{eq:int_geb}
\dist(x,\, C) \le \Phi\left(\max_{1 \le i \le m}\dist(x, \, C_i), \, \|x\|\right) \ \ \ \forall\ x\in\mathcal{E}.
\end{equation}
The first argument to $\Phi$ is ``$\max_{1 \le i \le m}\dist(x, \, C_i)$'' which means that the error bound must take into account the individual distances to the sets $C_i$. The second argument is 
``$\norm{x}$'' which reflects the fact that many error bounds correspond to inequalities that are only valid after a bounded subset is specified. Since we will impose coordinate-wise monotonicity of $\Phi$, under \eqref{eq:int_geb}, we have
\[
\dist(x,\, C) \le \Phi\left(\max_{1 \le i \le m}\dist(x,\, C_i), \, \rho\right) \ \ \ \forall\ x, \norm{x} \leq \rho,
\]
if $\rho > 0$ is some fixed constant.
An important property is that consistent error bound functions always exist  whenever \eqref{CFP} is feasible (see~Proposition~\ref{prop:uni}).

One of the main results of this paper is that a number of methods have convergence rates that can be written in terms of $\Phi$, see Theorem~\ref{thm_conv}. This will allow us to cover several previous results and also prove new ones.
For example, we will give a broad extension of the results of \cite{DLW17} and connect the singularity degree of certain conic feasibility problems to the convergence rates of several methods, see Section~\ref{sec:cone}.
Admittedly, for a general consistent error bound function, the expressions governing the convergence rate can be complicated, so we show in Section~\ref{sec:rv} how to use some tools from \emph{Karamata theory} in order to reason about those rates while avoiding certain complicated expressions.

\subsection{Our contributions}
Our contributions are as follows:

\begin{itemize}

\item We introduce a new notion of (strict) \emph{consistent error bound functions} (Definition~\ref{GEB}), which provides a unifying framework for error bounds for multiple convex sets, and includes error bounds beyond classical Lipschitzian and H\"olderian error bounds (Theorem~\ref{hold_geb}). We also show that a ``best'' consistent error bound function always exists for any finite family of convex sets having non-empty intersection (Proposition~\ref{prop:uni}). 

\item Under a \emph{strict} consistent error bound, we prove convergence rates for a number of algorithms fitting an abstract framework which includes many  projection algorithms, see Theorems~\ref{thm_conv} and \ref{thm_sp_4}. In particular, under H\"olderian error bounds, we will also derive precise sublinear rates for those algorithms, see also Corollaries~\ref{coro_abs_rate} and \ref{proj_hold_ge}.

\item We show how Karamata theory and functions of regular variation can be used to reason about the convergence rates obtained in Theorem~\ref{thm_conv} without the need of evaluating the integrals appearing therein, see Theorems~\ref{thm_comp},  \ref{thm:upper} and \ref{thm:lograte}. This will be used to analyze logarithmic and entropic error bounds appearing in some problems associated to the exponential cone, see Section~\ref{sec:exp}. In particular, we show that the convergence rate associated to the entropic error bound has an ``almost linear'' behavior, see Proposition~\ref{prop:exp}. We also provide a thorough analysis of logarithmic error bounds and  corresponding convergence rates, see Section~\ref{sec:log}.

\item  We also specialize our discussion to conic linear feasibility problems where the underlying cone is \emph{amenable} \cite{L19}. In this case, we prove that the convergence rates of several algorithms depend on the \emph{singularity degree} of the problem (see Section~\ref{sec:cone}), which is a quantity related to the facial reduction algorithm \cite{BW81_2,P13,WM13}.
In particular, when the cone is symmetric, we are able to extend a previous result of Drusvyatskiy, Li and Wolkowicz \cite{DLW17} along several directions, see Theorem~\ref{theo:sym_conv}.
\end{itemize}

The rest of the paper is organized as follows. In Section~\ref{sec:notation}, we introduce the notation appearing in the paper. In Section~\ref{sec:geb}, we  introduce the notions of (strict) consistent error bounds and corresponding (strict) consistent error bound functions, and discuss the relationship to H\"olderian error bounds. 
In Section~\ref{sec:proj}, under a \emph{strict} consistent error bound, we establish the convergence analysis for projection algorithms for convex feasibility problems.
Section~\ref{sec:rv} shows how to use Karamata theory to analyze convergence rates. Finally, applications to conic feasibility problems are discussed in Section~\ref{sec:cone}. In particular, Section~\ref{sec:exp} discusses non-H\"olderian error bounds appearing in the study of the exponential cone. Final remarks and future directions are presented in Section~\ref{sec:conc}.

\section{Notation}
\label{sec:notation}
Let $\R$ and $\R_+$ denote the set of real numbers and nonnegative numbers, respectively. Let $\mathcal{E}$ denote a finite-dimensional real vector space equipped with  norm $\|\cdot\|$ induced by some inner product $\langle\cdot,\, \cdot\rangle$. Given  $x\in\mathcal{E}$ and a closed convex set $C \subseteq\mathcal{E}$, we define 
\begin{equation*}
\dist(x,\, C) := \min_{y\in C}\|x - y\|
\end{equation*}
and  let $P_{C}(x)$ denote the projection of $x$ on the set $C$, i.e., $P_C(x) := \argmin_{y\in C}\|x - y\|$. We will denote by $\reInt C, C^\perp, \spanVec C$ the relative interior, orthogonal complement and linear span of $C$, respectively. If $C$ is a cone, we will write $C^*$ for its dual.


 \section{Consistent error bound functions}\label{sec:geb}
Partly motivated by the error bound for amenable cones in \cite{L19}, we propose the following 
notion. 
%
%

\begin{definition}[Consistent error bound functions]\label{GEB}
Let $C_1,\ldots, C_m\subseteq\mathcal{E}$ be closed convex sets with $C:=\bigcap_{i=1}^mC_i\neq\emptyset$.
A function $\Phi :[0, \infty)\times[0, \infty) \to [0, \infty) $ is 
said to be a \emph{consistent error bound function} for $C_1,\ldots, C_m$ if:
\begin{enumerate}[$(i)$]
	\item  the following error bound condition is satisfied:
	\begin{equation}\label{def_eb}
	\dist(x,\, C) \le \Phi\left(\max_{1 \le i \le m}\dist(x, C_i), \, \|x\|\right) \ \ \ \forall\ x\in\mathcal{E};
	\end{equation}
	\item 	for  any fixed $b\ge 0$,  the function $\Phi(\cdot,\, b)$ is monotone nondecreasing on $[0, \infty)$, right-continuous at $0$ and satisfies $\Phi(0,\, b) = 0$;
	\item for any fixed $a\ge 0$, the function $\Phi(a, \, \cdot)$ is monotone nondecreasing on $[0, \infty)$.
\end{enumerate}
In addition, if for every $b > 0$, $\Phi(\cdot, \, b)$ is monotone increasing on $[0, \infty)$ then $\Phi$ is said to be a  \emph{strict} consistent error bound function. We say that \eqref{def_eb} is the \emph{(strict, if $\Phi$ is strict) consistent error bound} associated to $\Phi$.
\end{definition}

\begin{remark}
Definition~\ref{GEB} admits a number of equivalent variations. For example, the individual distances to the sets $C_i$ are aggregated using the \emph{max} function (i.e., $\infty$-norm), however using the sum (i.e., 1-norm) or the square root of the sums-of-squares (i.e., 2-norm) would also be reasonable choices. Because of the equivalence of norms in real finite-dimensional spaces, these variations do not seem to affect significantly the error bound from an asymptotic point of view.
\end{remark}

Next we show that  every $C_1, \ldots, C_m$ with non-empty intersection admit a consistent error bound function. 
\begin{proposition}[The best consistent error bound function]\label{prop:uni}
Let $C_1,\ldots, C_m\subseteq\mathcal{E}$ be closed convex sets with $C:=\bigcap_{i=1}^mC_i\neq\emptyset$. There exists a consistent error 
bound function $\Phi$ for $C_1,\ldots, C_m$ with the property that 
if $\hat \Phi$ is any other consistent error bound function for $C_1,\ldots, C_m$ we have
\begin{equation}\label{eq:uni_unique}
\Phi(a,\, b) \leq \hat \Phi(a,\, b), \quad \forall \, a,\, b \in [0,\infty).
\end{equation}
In particular, $\Phi$ is unique.
\end{proposition}
\begin{proof}
Let $a$ and $b$ be in $[0,\infty)$ and consider the problem below parametrized by $a$ and $b$.
\begin{align}
 \qquad \underset{y}{\sup} & \quad \dist(y,\, C) \label{eq:uni_aux} \tag{U$(a,\, b)$}\\ 
 \mbox{subject to} & \quad \max _{1 \leq i \leq m}\dist(y,\, C_i) \leq a, \nonumber\\
& \quad \norm{y} \leq b. \nonumber
\end{align}	
We define $\Phi$ as follows
\[
\Phi(a,\, b) \coloneqq \begin{cases}
\text{optimal value of \eqref{eq:uni_aux}} & \text{if \eqref{eq:uni_aux} is feasible}\\
0 & \text{otherwise}.
\end{cases}
\]
Because of the norm constraint  in \eqref{eq:uni_aux}, the feasible region of \eqref{eq:uni_aux} is compact although it can be empty. 
Since $\dist(\cdot,C)$ is a continuous function, $\Phi(a,\, b)$ is finite and nonnegative.
Increasing either $a$ or $b$ potentially enlarges the feasible region of \eqref{eq:uni_aux}, 
so $\Phi(\cdot,\, b)$ and $\Phi(a,\, \cdot)$ are monotone nondecreasing.
Furthermore, if $a = 0$, then the only feasible solutions to \eqref{eq:uni_aux} (if any) must be elements of $C$, so $\Phi(0,\, b) = 0$  for every $b$. 

Next, let $x \in \mathcal{E}$, $a = \max_{1 \le i \le m}\dist(x, \, C_i)$ and $b = \norm{x}$.
Then, $y = x$ is feasible for \eqref{eq:uni_aux} and we have
\[
\dist(x,\, C) \leq \Phi(\max _{1 \leq i \leq m}\dist(x,\, C_i), \, \norm{x}).
\]
Therefore, except for the continuity requirement, $\Phi$ satisfies items $(i)$, $(ii)$, $(iii)$. 
So let $b \in [0,\infty)$ and we will check that $\Phi(\cdot,b)$ is (right-)continuous at $0$.
In order to do that, it suffices to show that for any sequence $\{a_k\} \subseteq [0,\infty)$ with $a_k\to 0$, we have $\Phi(a_k,b) \to 0$.  Let $\{a_k\}$ be any such sequence. First, for the $(a_k,b)$ such that ${\rm U}(a_k,b)$ is infeasible, we have
$\Phi(a_{k},b) = 0$. 

Next, we consider the pairs $(a_k,b)$ such that ${\rm U}(a_k,b)$ is feasible.
If there are only finitely 
many such $( a_k,b)$, we must have $\Phi(a_{k},b) \to 0$.
So, suppose that there are infinitely many such $(a_k,b)$ and, for convenience, denote the sequence of the corresponding $a_k$ by $\{\hat{a}_k\}$. We  have 
$\hat a _k \to 0$, since  $\{\hat{a}_k\}$ is a subsequence of $\{a_k\}$.

For each pair $(\hat a_k,b)$, the feasible region of 
${\rm U}(\hat a_k,b)$  is compact, so there exists an optimal solution $y^k$ satisfying 
\begin{equation}\label{eq:uniq}
\dist(y^k,C) = \Phi(\hat a_k,b),\quad \max _{1 \leq i \leq m}\dist(y^k,\, C_i) \leq \hat a_k, \quad \norm{y^k} \leq b.
\end{equation}
Consequently, to show $\Phi(\hat a_k,b) \to 0$, it suffices to prove $\dist(y^{k},\, C)\to 0$. Suppose that $\dist(y^{k},\, C)\to 0$ does not hold. Then there exist some $\delta > 0$ and a subsequence $\{y^{k_j}\}$ such that $\dist(y^{k_j},\, C) \geq \delta$ for all $j$. Since all the $y^k$ are contained in a ball of radius $b$, by passing to a further subsequence if necessary, we may assume that $y^{k_j}$ has a limit 
$\overline{y}$. By \eqref{eq:uniq} and the continuity of $\dist(\cdot,C_i)$ we have $\dist(\overline{y},C_i) = 0$ for all $i$, which implies that $\overline{y} \in C$.
Furthermore, because $\dist(\cdot, C)$ is continuous, we 
have\[
 \dist(y^{k_j},C)    \to \dist(\overline y,C) = 0,
\]
which contradicts the fact that $\dist(y^{k_j},C) \geq \delta > 0$, for every $j$. This proves $\Phi(\hat a_k,b) \to 0$ for the pairs $(\hat a_k,b)$ such that  ${\rm U}(\hat a_k,b)$ is feasible. Accordingly, we must have $\Phi(a_k,b) \to 0$.
The (right-) continuity of $\Phi(\cdot,b)$ at 0 then follows from the arbitrariness of $\{a_k\}$. 

Finally, in order to show that \eqref{eq:uni_unique} holds, let $\hat \Phi$ be another consistent error bound function for $C_1, \ldots, C_m$. 
For the sake of obtaining a contradiction, suppose that there exist $a,\, b$ such that 
\[
\Phi(a,\, b) >  \hat \Phi(a,\, b),
\]
With that, the corresponding problem \eqref{eq:uni_aux} must be feasible, because otherwise 
we would have $\Phi(a,\, b) = 0$.
Then, since $\Phi(a,\, b)$ is the optimal value of \eqref{eq:uni_aux}, there exists a feasible solution $y$ such that $\Phi(a,\, b) \geq \dist(y,\, C) >  \hat \Phi(a,\, b)$. However,
\[
\dist(y,\, C) \leq \hat \Phi(\max_{1 \le i \le m}\dist(y, \, C_i),\, \norm{y}) \leq \hat \Phi(a,\, b),
\]
where the second inequality follows because $y$  is feasible for \eqref{eq:uni_aux} and $\hat \Phi$ satisfies items $(ii)$ and $(iii)$ of Definition~\ref{GEB}.
Together with  $\dist(y,\, C) >  \hat \Phi(a,\, b)$, we obtain a contradiction.
This shows $\Phi$ satisfies \eqref{eq:uni_unique} and that $\Phi$ must be the unique consistent error bound function for which \eqref{eq:uni_unique} holds.
\end{proof}
We call the function defined in Proposition~\ref{prop:uni} the \emph{best consistent error bound function} for $C_1, \ldots, C_m$ and, in a sense, reflects the tightest possible error bound one can get for the $C_is$. 
We remark that any consistent error bound function $\Phi$ can be made strict as follows. Let $\kappa > 0$ be a constant and let 
\begin{equation*}
\hat\Phi(a,\, b) \coloneqq \Phi(a,\, b) + \kappa a, \quad \forall \, a,\, b \in [0,\infty).
\end{equation*}
Then, $\hat \Phi$ is a consistent error bound function for the same sets that is also strict.
Therefore, Proposition~\ref{prop:uni} also implies the existence of strict consistent error bound functions.

\subsection{H\"olderian and Lipschitzian error bounds}\label{sec:hold}
It turns out that  consistent error bounds include a large variety of existing error bounds. First, we will show that H\"olderian error bounds are covered. Other examples of error bounds will be seen in
Section~\ref{sec:log}, Section~\ref{sec:sym} and Section~\ref{sec:exp}.
We recall the following definition.
\begin{definition}[H\"olderian error bound]\label{def_hold}
	The sets $C_1,\ldots, C_m\subseteq\mathcal{E}$ with $C:=\bigcap_{i=1}^mC_i\neq\emptyset$ are said to satisfy a H\"olderian error bound if for every bounded set $B\subseteq\mathcal{E}$ there exist some $\theta_B > 0$ and an exponent $\gamma _B\in(0, 1]$ such that
	\begin{equation*}
	\dist(x,\, C) \le \theta_B\max_{1\le i\le m}\dist^{\gamma_B}(x, \, C_i) \ \ \ \forall\ x\in B.
	\end{equation*}
	If we can take the same exponent $\gamma _B = \gamma \in (0,1]$ for all $B$, then we say 
	that the bound is \emph{uniform}. Furthermore, 
	if the bound is uniform with $\gamma = 1$, we call it a Lipschitzian error bound.
\end{definition}

\begin{theorem}[Characterization of H\"olderian error bounds]\label{hold_geb}
Let $C_1,\ldots, C_m\subseteq\mathcal{E}$ be convex sets with $C:=\bigcap_{i=1}^mC_i\neq\emptyset$. 
\begin{enumerate}[$(i)$]
	\item $C_1,\ldots, C_m$ satisfy a H\"olderian error bound if and only if there are monotone nonincreasing $\gamma: [0,\infty) \to (0,1]$ and monotone nondecreasing $\rho : [0,\infty) \to (0,\infty)$ such that the following function is  a strict consistent error bound function for $C_1,\ldots,C_m$:
	\begin{equation}\label{phi_hold}
	\Phi(a,\, b) \coloneqq \rho(b)\max(a^{\gamma(b)},\, a).
	\end{equation}

	\item $C_1,\ldots, C_m$ satisfy a uniform H\"olderian error bound with exponent $\gamma\in(0, 1]$ if and only if there exists a monotone nondecreasing $\rho : [0,\infty) \to (0,\infty)$ such that the following function is  a strict consistent error bound function for $C_1,\ldots,C_m$:
	\begin{equation}\label{phi_hold_u}
	\Phi(a,\, b) \coloneqq \rho(b)a^{\gamma}.
	\end{equation}
	
\end{enumerate}
\end{theorem}

\begin{proof}
In what follows, we let $d$ be the function such that
\[
d(x) = \max_{1\le i\le m}\dist(x, \, C_i).
\]	
First we prove item $(i)$. Suppose that $C_1,\ldots, C_m$ satisfy a H\"olderian error bound. Let $B$ be any fixed bounded set.
 From Definition~\ref{def_hold}, there exist $\theta_B > 0$ and an exponent $\gamma_B\in(0, 1]$ such that
\begin{equation}\label{eq:hod_geb}
  \dist(x, \, C) \le \theta_Bd(x)^{\gamma_B} \ \ \ \forall\ x\in B.
\end{equation}
Equivalently, we have
\begin{equation}\label{eq:hod_geb2}
\dist(x, \, C) \le \theta_B\max( d(x)^{\gamma_B},\, d(x)) \ \ \ \forall\ x\in B.
\end{equation}
The equivalence between \eqref{eq:hod_geb} and \eqref{eq:hod_geb2} is as follows.
If $\gamma_B \in (0,1]$ is an exponent such that \eqref{eq:hod_geb} holds for some constant $\theta_B$, then \eqref{eq:hod_geb2} holds. Conversely, suppose that \eqref{eq:hod_geb2} holds for some $\gamma_B$ and some constant $\theta _B$. 
Then \eqref{eq:hod_geb} holds with the same $\gamma_B$ and
constant $\theta _B \max(1, \sup _{x \in B} d(x)^{1-\gamma_B})$.

With that in mind, given  a bounded set $B$, we say that $\gamma$ is an \emph{admissible 	exponent for $B$} if there exists a constant $\theta _B$ such that \eqref{eq:hod_geb} or \eqref{eq:hod_geb2} holds.
Next, we verify the following property: if $\gamma$ is an admissible exponent for $B$, then any $\hat \gamma \in (0,\, \gamma)$ is an admissible exponent for $B$. 
This is because
\[
\max( a^{\gamma},\,a) \leq \max( a^{\hat \gamma},\, a) \quad \forall a\geq0.
\]
For $r > 0$, we let $\gamma _r$ denote the supremum of all admissible exponents 
for $U_r \coloneq \{y : \norm{y} \leq r\}$. Then, $\gamma_r$ has the following property:
\begin{enumerate}[$(a)$]
	\item any $0 <\gamma < \gamma _r$ is an admissible exponent for $U_r$, although $\gamma _r$ itself might not necessarily be admissible.
\end{enumerate}

We will now construct a sequence of admissible exponents $\hat \gamma _k$ for 
the neighbourhoods $U_k$ together with constants 
$\theta _k$, for all positive integer $k$.
First, we let $\hat \gamma _1$ to be any admissible exponent for $U_1$ such that 
$\hat \gamma _1 < \gamma _1$ together with a constant $\theta _1 \geq 1$ such that 
\eqref{eq:hod_geb2} holds with $\gamma = \hat \gamma _1$ and $B = U_1$.

For $k > 1$ we proceed as follows. We let $\hat \gamma _{k}$ be any admissible exponent for $U_k$ satisfying 
\[
\hat \gamma _{k} < \min \{\hat \gamma _{k-1},\, \gamma _{k}\},
\]
which is possible in view of property $(a)$.

Then, we select $\theta _k$ such that \eqref{eq:hod_geb2} holds for $\gamma = \hat \gamma _{k}, B = U_k$ and such that 
\[
\theta _k \geq \theta _{k-1},
\]
which is possible because if \eqref{eq:hod_geb2} is satisfied for some constant $\theta _B$, it is still satisfied for any constant larger than $\theta _B$.

Now, we define functions $\gamma: [0,\infty) \to (0,1]$ and $\rho: [0,\infty) \to (0,\infty)$ that 
interpolate the values of $\hat \gamma _k$ and $\theta _k$.  For that, given a nonnegative real $a$, we define $\ceil{a}$ to be smallest integer satisfying $a \leq \ceil{a}$.
Then, we define
\begin{equation*}
\gamma(a) \coloneq \begin{cases}\hat\gamma _{\ceil{a}} & \text{ if }a > 0 \\ \hat \gamma _1 & \text{ if } a = 0\end{cases}, \qquad \rho (b) \coloneq \begin{cases}
\theta _{\ceil{b}} & \text{ if }b > 0\\
\theta _{1} & \text{ if } b  = 0
\end{cases} .
\end{equation*}
By the construction of $\hat \gamma_k$ and $\theta _k$, both $\gamma$ and $\rho$ are, respectively, monotone nonincreasing and monotone nondecreasing.
Next, we let $\Phi$ be such that 
\[
\Phi(a,\, b) \coloneqq \rho(b)\max(a^{\gamma(b)},\, a).
\]
Let $a,b \in [0,\infty)$ be arbitrary. The monotonicity of $\gamma$ and $\rho$, and $\gamma(\cdot)\in(0,\,1]$ imply that $\Phi(\cdot, \, b)$ and 
$\Phi(a,\, \cdot)$ are monotone increasing and monotone nondecreasing, respectively. For any fixed $b\in[0,\, \infty)$, function $\Phi(\cdot,\, b)$ is right-continuous at $0$. 
We also have $\Phi(0,\, b) = 0$.
Furthermore, if $x \in \mathcal{E}$ arbitrary, then $x \in U _{\ceil{\norm{x}}}$, so 
\[
\dist(x, \, C) \le \rho(\norm{x}) \max( d(x)^{\gamma({\norm{x}})},\, d(x)) = 
\Phi(d(x),\, \norm{x}),
\]
therefore, $\Phi$ is indeed a strict consistent error bound function.

For the converse, we suppose that \eqref{phi_hold} is satisfied and we need to show that 
$C_1, \ldots, C_m$ satisfy a H\"olderian error bound. Let $B$ a bounded set and 
let $r$ be the supremum of the norm of the elements of $B$. Then, $B$ is contained in a ball of radius $r$. Therefore, for $x \in B$ we have
\begin{align*}
\dist(x,\, C)& \le \Phi(d(x),\, \norm{x})  \\
& = \rho(\norm{x}) \max( d(x)^{\gamma({\norm{x}})},\, d(x))\\
& \leq \rho(r) \max( d(x)^{\gamma({\norm{x}})},\, d(x)),
\end{align*}
where the last inequality follows from the monotonicity of $\rho$.
By the equivalence between \eqref{eq:hod_geb} and \eqref{eq:hod_geb2}, we conclude 
that a H\"olderian error bound holds.
This concludes the proof of $(i)$.

We move on to $(ii)$.
First, we suppose that a uniform H\"olderian error bound with exponent $\gamma$ holds 
for $C_1, \ldots, C_m$.
Let $\rho(b)$ be the solution of the following optimization problem:
\begin{equation}\label{cons_prob}
\begin{split}
\rho(b): = & \argmin_{\alpha \ge 1} \ \alpha \\
& {\rm s.t.}\ \  \dist(y,\, C) \le \alpha\,\left(\max_{1\le i\le m}\dist(y,\, C_i)\right)^{\gamma} \ \ \ \forall\ y \text{ satisfying } \norm{y} \leq b.
\end{split}
\end{equation}
From the definition of H\"{o}lderian error bound (Definition~\ref{def_hold}) the feasible set of \eqref{cons_prob} is nonempty for every $b \geq 0$. Furthermore, the feasible set of \eqref{cons_prob} is closed and convex. Therefore, the solution of \eqref{cons_prob} is unique. Consequently, $\rho(b)$ is well-defined and $\rho$ is monotone nondecreasing. 
Finally, we have
\begin{equation*}
\dist(x, \, C) \le \rho(\|x\|)\,\left(\max_{1\le i\le m}\dist(x,\, C_i)\right)^{\gamma},  \quad \forall\ x\in\mathcal{E}.
\end{equation*}
By the monotonicity of $\rho(\cdot)$, we conclude that Definition~\ref{GEB} is satisfied for $\Phi(a,\, b) = \rho(b)\, a^{\gamma}$.

For the converse, suppose that \eqref{phi_hold_u} holds.
 Let $B$ a bounded set and let $r$ be the supremum of the norm of the elements of $B$. Then, $B$ is contained in a ball of radius $r$. Therefore, for $x \in B$ we have
\begin{align*}
\dist(x,\, C)& \le \phi(d(x),\, \norm{x})  = \rho(\norm{x}) d(x)^\gamma   \leq \rho(r) d(x)^\gamma,
\end{align*}
where the last inequality follows from the monotonicity of $\rho$.
\end{proof}

\begin{example}\label{example}
	It is known that certain constraint qualifications imply Lipschitzian error bounds, see \cite[Corollary~3]{BBL99} or \cite[Theorem~3.1]{BT03}.	
	For conditions ensuring the existence of 
	H\"olderian error bounds see \cite[Theorem~3.3]{S00} (linear matrix inequalities), \cite[Theorem~37]{L19} (symmetric cones), \cite[Theorem~3.6]{BLY14} (basic semialgebraic convex sets). These references 
	all include information on how to estimate the 
	exponent of the error bound, which can be quite nontrivial in more general settings. For more on this difficulty, see the comments after Theorems 11 and 13  in \cite{Pang97}.

\end{example}

\section{Convergence analysis under consistent error bounds}\label{sec:proj}

In this section, we show how to connect consistent error bound functions to the convergence rate of a number of algorithms for solving \eqref{CFP}.  
Before proceeding, we introduce a key tool for our analysis - inverse smoothing functions constructed from \emph{strict} consistent error bound functions.

\subsection{Inverse smoothing function from strict consistent error bound function}\label{sub_sec_po}

Let $\Phi$ be a \emph{strict} consistent error bound function as in Definition~\ref{GEB}.  Then, for $\kappa > 0$, we define $\phi_{\kappa, \Phi}$ as follows: 
\begin{equation}\label{def_psi}     
\phi_{\kappa, \Phi}(t)  := \left(\Phi(\sqrt{t},\, \kappa)\right)^2,\ \ \   t\ge 0.
\end{equation}
 The following lemma follows directly from the properties of $\Phi$ in Definition~\ref{GEB}.
\begin{lemma}\label{lemma_varphi}
Let $\phi_{\kappa, \Phi}$ be defined as in \eqref{def_psi}. Then $\phi_{\kappa, \Phi}(0) = 0$, $\phi_{\kappa, \Phi}(\cdot)$ is  monotone increasing on $[0, \infty)$ and 
right-continuous at $0$. Moreover, we have $\phi_{\kappa_1, \Phi}(t) \le \phi_{\kappa_2, \Phi}(t)$ for all $t$ whenever $\kappa_1\le \kappa_2$.
\end{lemma}

 Before proceeding, we define the \emph{generalized inverse function} for any monotone increasing function $f: \R_+\to \R_+$ as:
\begin{equation}\label{inv_fun}
f^{-}(s): = \inf\left\{t\ge 0: f(t) \ge s \right\}, \ \ 0\le s < \sup f,
\end{equation}
see \cite{EH13} for more details on generalized inverses. 
Any monotone increasing function has an inverse $f^{-1}$ in the usual sense, but $f^{-}$ fixes a number of deficiencies that $f^{-1}$ might have when $f$ is not continuous everywhere. 
However, if $f$ is both continuous and monotone increasing, then $f^- = f^{-1}$, see \cite[Remark~1]{EH13}. The proof of the following lemma about the properties of $f^-$ is given in Appendix~\ref{appendix_a}.
 


\begin{lemma}[Properties of the generalized inverse]\label{inv_lemma}
Let $f: \R_+\to \R_+$ be a monotone increasing function with $f(0) = 0$. Define $f^{-}$ as in \eqref{inv_fun}. Then, $f^{-}$ is monotone nondecreasing, $f^{-}(0) = 0$ and the following statements hold:
\begin{enumerate}[$(i)$]

\item \label{inv_lemma:1}   if $f$ is (right-)continuous at $0$, then $f^{-}(s) > 0$ for all $s\in(0,\, \sup f)$;

\item \label{inv_lemma:2} for any $s\ge 0, t\geq 0$
such that $s\le f(t)$ holds, we have $s < \sup f$ and $f^{-}(s)\le t$;
\item \label{inv_lemma:3} 
 for any $s\ge 0, t\geq 0$ such that  $s < \sup f$ and $f(t) < s$ holds, we have $t\le f^{-}(s)$;


\item \label{inv_lemma:4} $f^{-}$ is continuous on $(0,\, \sup f)$.

\end{enumerate}
\end{lemma}

Next, we will introduce the ace of our toolbox: the so-called \emph{inverse smoothing function} associated to $\Phi$.
For $\kappa > 0$ and for $\phi_{\kappa, \Phi}$ as in \eqref{def_psi} we define $\Phi_{\kappa}^\spadesuit$ as
\begin{equation}\label{def_phi}
\Phi_{\kappa}^\spadesuit(t)  := \int_{\delta}^t\frac1{\phi_{\kappa, \Phi}^{-}(s)}ds, \ \ \ t \in\left(0,\, \sup\phi_{\kappa,\Phi}\right),
\end{equation}
where $\delta\in(0,\, \sup\phi_{\kappa,\Phi})$ is some fixed number\footnote{Any $\delta$ in $(0,\, \sup\phi_{\kappa,\Phi})$ is fine, so we will not include $\delta$ in the notation for $\Phi_{\kappa}^\spadesuit(t)$. The only place where we make a specific choice of $\delta$ is in the proof of Corollary~\ref{coro_abs_rate}. See also Remark~\ref{remark_delta}.}.
  We note that  $\Phi_{\kappa}^\spadesuit$ is well-defined thanks to Lemma~\ref{lemma_varphi} and Lemma~\ref{inv_lemma}~$(\ref{inv_lemma:1})$ and $(\ref{inv_lemma:4})$.
 
  The properties of $\Phi_{\kappa}^\spadesuit$ are as follows.

\begin{proposition}[The properties of $\Phi_{\kappa}^\spadesuit$]\label{prop_phi}
Let $\Phi_{\kappa}^\spadesuit$ be defined as in \eqref{def_phi} with $\phi_{\kappa, \Phi}$ defined as in \eqref{def_psi}. Then $\Phi_{\kappa}^\spadesuit$ is concave, monotone increasing and continuously differentiable on $(0, \, \sup\phi_{\kappa,\Phi})$. 
\end{proposition}

\begin{proof}
From Lemma~\ref{lemma_varphi} and Lemma~\ref{inv_lemma}~$(\ref{inv_lemma:1})$, $(\ref{inv_lemma:4})$,  we see that $\phi^{-}_{\kappa, \Phi} $ is continuous on $(0, \, \sup\phi_{\kappa,\Phi})$ and positive. Therefore, $\Phi_{\kappa}^\spadesuit$ is monotone increasing and continuously differentiable with $(\Phi_{\kappa}^\spadesuit)'(t) = \frac1{\phi_{\kappa, \Phi}^{-}(t)}$ for $t\in (0, \, \sup\phi_{\kappa,\Phi})$. This together with the monotonicity of $\phi^{-}_{\kappa, \Phi}$ from Lemma~\ref{inv_lemma} implies that $(\Phi_{\kappa}^\spadesuit)'$ is monotone nonincreasing on $(0, \, \sup\phi_{\kappa,\Phi})$, which shows that $\Phi_{\kappa}^\spadesuit$ is concave. For the sake of self-containment, we show this last assertion. For any fixed $x,y \in(0,\, \sup\phi_{\kappa,\Phi})$, we define $\theta(t) := \Phi_{\kappa}^\spadesuit(x + t(y - x))$. With that, we have $\Phi_{\kappa}^\spadesuit(y) - \Phi_{\kappa}^\spadesuit(x) = \theta(1) - \theta(0)$ and, by integration, we obtain 
\begin{align*}
&\ \  \  \  \Phi_{\kappa}^\spadesuit(y) - \Phi_{\kappa}^\spadesuit(x) \\
&= \int_0^1(\Phi_{\kappa}^\spadesuit)'\left(x + t(y - x)\right)(y - x)\, dt\\
& = \int_0^1 {\left[(\Phi_{\kappa}^\spadesuit)'\left(x + t(y - x)\right) - (\Phi_{\kappa}^\spadesuit)'(x)\right](y - x)}\, dt + \int_0^1(\Phi_{\kappa}^\spadesuit)'(x)(y - x)\, dt \\
&\le (\Phi_{\kappa}^\spadesuit)'(x)(y - x),
\end{align*}
where the last inequality follows from the monotonicity of $(\Phi_{\kappa}^\spadesuit)'$. Therefore, $\Phi_{\kappa}^\spadesuit$ is concave. This completes the proof.
\end{proof}

Next, we take a look at the behavior of $\Phi_{\kappa}^\spadesuit(t)$ as $t \to 0$. 

\begin{proposition}[Asymptotical properties of $\Phi_{\kappa}^\spadesuit$]\label{prop:asym}
Let $\Phi_{\kappa}^\spadesuit$ be defined as in \eqref{def_phi} with $\phi_{\kappa, \Phi}$ defined as in \eqref{def_psi}. Suppose that $C$ is not the whole space.
Let $x^0 \not \in C$ and suppose that 
$\kappa \geq \max\{\dist(0,C), \norm{x^0} \}$.
Then, $\Phi_{\kappa}^\spadesuit(t) \to -\infty$ as 
$t \to 0$. 
\end{proposition}
\begin{proof}
Let $B_{\kappa} \coloneqq \{x \in \mathcal{E} \mid \norm{x} \leq \kappa \}$ and let $d$ be the function such that
\[
d(x) = \max_{1\le i\le m}\dist(x, \, C_i).
\]
From \eqref{def_eb} and the fact that $C \subseteq C_i$ for all $i$, we have
\[
d(x) \leq \dist(x,\, C) \le \Phi\left(d(x), \, \kappa \right) \ \ \ \forall\ x \in B_{\kappa}.
\]
Then, from \eqref{def_psi} we have
\begin{equation}\label{eq:asym_key}
d(x)^2 \le \Phi(d(x),\, \kappa)^2 = \phi_{\kappa,\Phi}(d(x)^2) \ \ \ \forall\ x \in B_\kappa.
\end{equation}
Next, we examine the image of $d(\cdot)^2$ restricted 
to $B_\kappa$.
Since $\kappa \geq \max\{\dist(0,C), \norm{x^0} \}$, we have 
$x^0 \in B_{\kappa}$ and $P_C(0) \in B_\kappa$.
Let $\mu \coloneqq d(x^0)^2$. Since $d(\cdot)^2$ is a continuous function, by the intermediate value theorem, 
the image of $d(\cdot)^2$ restricted to $B_{\kappa}$ contains the interval $[0,\mu]$. We also have $\mu \neq 0$, because 
$x^0 \not \in C$.
In view of \eqref{eq:asym_key}, we have
\[
s \leq \phi_{\kappa,\Phi}(s), \quad \forall \ s \in [0,\mu].
\]
 Let $\tau = \min(\mu, \delta)$, where $\delta$ comes from 
the definition of $\Phi_{\kappa}^{\spadesuit}$ in \eqref{def_phi}. 
From Lemma~\ref{inv_lemma}~$(\ref{inv_lemma:2})$ we obtain
\begin{equation}\label{eq:diff}
\phi_{\kappa,\Phi}^{-}(s) \le s, \ \ s\in(0,\, \tau).
\end{equation}
Therefore, the following inequality holds for $t\in (0,\tau)$
\begin{equation*}
 - \Phi_{\kappa}^{\spadesuit}(t) = \int_{t}^{\delta}\frac1{\phi_{\kappa,\Phi}^{-}(s)}\, ds \ge  \int_{t}^{\tau}\frac1{\phi_{\kappa,\Phi}^{-}(s)}\, ds \ge \int_t^{\tau} \frac1{s}\, ds  = \ln \tau - \ln t.
\end{equation*}
This shows that $\Phi_{\kappa}^{\spadesuit}(t) \to -\infty$ as $t \to 0$ and completes the proof. 
\end{proof}

\subsection{Convergence analysis of sequences}\label{subsec_rate}

In this section, we make use of the inverse smoothing function discussed in Section~\ref{sub_sec_po} to analyze the convergence properties of sequences satisfying the Assumption~\ref{assp} below. Later, in Section~\ref{sec:proj_alg}, we show that several algorithms generate sequences of iterates satisfying Assumption~\ref{assp}.
\begin{assumption}\label{assp}
 Let $\{x^k\}\subseteq \mathcal{E}$ be a sequence such that the following conditions hold.

\begin{enumerate}[$(i)$]
\item \label{assp:h1} \emph{Fej\'er monotonicity condition}. For any fixed $c\in C$, it holds that
\begin{equation}
\|x^{k+1} - c\| \le \|x^k - c\| \ \ \ \forall\ k. \label{h1}
\end{equation}

\item \label{assp:h2}\emph{Sufficient decrease condition}. There exist some positive integer $\ell$ and nonnegative sequence $\{a_k\}$  with $\sum_{k=0}^{\infty}a_k = \infty$ such that 
\begin{equation}\label{h2}
 \dist^2(x^k, C)  \ge \dist^2(x^{k + \ell}, C) + a_k\max_{1\le i \le m}\dist^2(x^k, C_i) \ \ \ \forall\ k.
\end{equation}
\end{enumerate}
\end{assumption}

The  Fej\'er monotonicity assumption appears frequently in the study of convex feasibility problems, see \cite[Theorem~2.16]{BB96}. The sufficient decrease condition is inspired by similar conditions appearing in \cite{lt93,BNPS17}. However, we allow the possibility of having decrease after a fixed number of iterations instead of forcing decrease after every iteration.

\begin{proposition}\label{conv_thm}
Let  Assumption~\ref{assp} hold. Then $\{x^k\}$ converges to some point in $C$. 
\end{proposition}

\begin{proof}
Since $\sum_{k=0}^{\infty}a_k = \infty$ holds, there exists some integer $k_0\in[0, \ell - 1]$ such that 
\begin{equation}\label{eq:ak_inf}
\sum_{i=0}^{\infty}a_{k_0 + i\ell} = \infty.
\end{equation}
For any $N > 0$, summing both sides of \eqref{h2} for $k = k_0 + i\ell$ with $i = 0,\ldots, N-1$, we  obtain
\begin{equation}\label{finite_sum}
\begin{aligned}
 \dist^2(x^{k_0}, C) &\ge \dist^2(x^{k_0}, C) - \dist^2(x^{k_0 + N\ell}, C) \\
& = \sum_{i=0}^{N-1} \dist^2(x^{k_0 + i\ell}, C) - \dist^2(x^{k_0 + (i+1)\ell}, C) \\
&\ge \sum_{i=0}^{N -1} a_{k_0 + i\ell}\max_{1\le j\le m}\dist^2(x^{k_0 + i\ell}, C_j).
\end{aligned}
\end{equation}
Letting $N\rightarrow\infty$ in \eqref{finite_sum}, we then have $\sum_{i=0}^{\infty} a_{k_0 + i\ell}\max_{1\le j\le m}\dist^2(x^{k_0 + i\ell}, C_j) < \infty$. This, together with \eqref{eq:ak_inf}, implies that there exists a subsequence $\{x^{k_i}\}$ such that 
\begin{equation*}
\max_{1\le j\le m}\dist(x^{k_i}, C_j)\rightarrow 0\  \ \ {\rm when}\  i\rightarrow \infty.\footnote{The relevant fact is that if $\{u_k\},\{v_k\}$ are nonnegative sequences with $\sum u_k = \infty$ and $\sum u_kv_k < \infty$, then $\liminf v_k = 0$.}
\end{equation*}
Therefore, $\dist(x^{k_i}, C_j)\rightarrow 0$ for all $j=1,\ldots, m$. On the other hand, we know from the Fej\'er monotonicity of $\{x^k\}$ in \eqref{h1} that  $\{x^k\}$ is bounded. Thus, there exists a subsequence of $\{x^{k_i}\}$ which converges to some point $x^*\in\mathcal{E}$. Without loss of generality, we still let $\{x^{k_i}\}$ denote  this subsequence so that $\lim_{i\rightarrow\infty}\|x^{k_i} - x^*\| = 0$. Then, $\dist(x^{k_i}, C_j)\rightarrow 0$ and the closedness of the $C_j$ imply that $x^*\in\bigcap_{i=1}^m\, C_j = C$. Thus, using again the  Fej\'er monotonicity of $\{x^k\}$, we obtain 
\begin{equation*}
\|x^{k+1} - x^*\| \le \|x^k - x^*\| \ \  \ \forall\ k,
\end{equation*}
which  together with $\lim_{i\rightarrow\infty}\|x^{k_i} - x^*\| = 0$ gives $x^k\rightarrow x^*\in C$. 
\end{proof}

Now we establish our convergence rate under a strict consistent error bound as in Definition~\ref{GEB}.

\begin{theorem}\label{thm_conv}
Suppose that Assumption~\ref{assp} holds. Let $\Phi$ be a strict consistent error bound function for $C_1,\ldots, C_m$ as in Definition~\ref{GEB}. Let $\Phi_{\widehat{\kappa}}^\spadesuit$ be defined as in \eqref{def_phi} with $\widehat{\kappa}$ such that $\widehat{\kappa} \geq  \|x^0\| + 2\,\dist(0,\, C)$. Then, the convergence of $\{x^k\}$ is either finite or 
\begin{equation}\label{rate}
\dist (x^k,\, C) \le \sqrt{(\Phi_{\widehat{\kappa}}^\spadesuit)^{-1}\Big(\Phi_{\widehat{\kappa}}^\spadesuit(\dist^2(x^0,\, C) )- \sum_{i=0}^{b_k -1} a_{k_0 + i\ell}\Big)}\ \ \ \forall\ k\ge 2\ell
\end{equation}
holds for any integer $k_0\in[0,\, \ell-1]$ and $b_k:= \frac{k - \ell - (k \bmod \ell)}{\ell}$.
\end{theorem}

\begin{proof}
First, the convergence of sequence $\{x^k\}$ follows from Proposition~\ref{conv_thm}. Note from \eqref{h1} that if there exists some $\widebar{k}$ such that $\dist(x^{\widebar{k}},\, C) = 0$, we have $x^k = x^{\widebar{k}}$ for all $k\ge \widebar{k}$. Consequently, in this case, $\{x^k\}$ converges finitely and we are done. 

Next, suppose that the convergence is not finite. Then, $\dist(x^k,\, C) > 0$ holds for all $k$. Notice that $\widehat{\kappa} > 0$; otherwise we have $\dist(x^0,\, C) = 0$. Let $c^*:=\argmin_{c\in C}\|c\|$. We then see from the  Fej\'er monotonicity of $\{x^k\}$ (\eqref{h1} in Assumption~\ref{assp}) that
\begin{equation*}
\|x^k - c^*\| \le \|x^0 - c^*\| \ \ \ \forall\ k,
\end{equation*}
which gives $\|x^k\| \le \|c^*\| + \|x^0 - c^*\| \le \widehat{\kappa}$ for all $k$.  This together with Definition~\ref{GEB}~(i), the definition of $\phi_{\kappa, \Phi}$ in \eqref{def_psi} and  Lemma~\ref{lemma_varphi} implies that for all $k$,
\begin{equation*}
\begin{aligned}
 \dist^2(x^k,\, C) &\le \Big(\Phi\big(\max_{1\le i\le m}\dist(x^k,\, C_i),\, \|x^k\|\big)\Big)^2 \\
& = \phi_{\|x^k\|, \Phi}(\max_{1\le i\le m}\dist^2(x^k,\, C_i)) \le \phi_{\widehat{\kappa}, \Phi}(\max_{1\le i\le m}\dist^2(x^k,\, C_i)).
\end{aligned}
\end{equation*}
This combined with Lemma~\ref{lemma_varphi} and Lemma~\ref{inv_lemma}~$(\ref{inv_lemma:2})$ implies that $\dist^2(x^k,\, C) \in (0,\, \sup\phi_{\widehat{\kappa},\Phi})$ and 
\begin{equation}\label{eb_k}
\phi_{\widehat{\kappa}, \Phi}^{-}\left(\dist^2(x^k,\, C)\right) \le \max_{1\le i\le m}\dist^2(x^k,\, C_i)\ \ \ \forall\ k.
\end{equation}
Now we combine  \eqref{def_phi}, \eqref{h2} and \eqref{eb_k}, use the concavity and differentiability of $\Phi_{\widehat{\kappa}}^\spadesuit$ from Proposition~\ref{prop_phi} and obtain
\begin{equation} \label{key}
\begin{split}
\Phi_{\widehat{\kappa}}^\spadesuit\left(\dist^2(x^{k},\, C)\right) - \Phi_{\widehat{\kappa}}^\spadesuit\left(\dist^2(x^{k+\ell},\, C)\right)&    \ge (\Phi_{\widehat{\kappa}}^\spadesuit)'(\dist^2(x^k,\, C))\left(\dist^2(x^k, \, C) - \dist^2(x^{k+\ell},\, C)\right) \\
& = \frac{1}{\phi_{\widehat{\kappa}, \Phi}^{-}(\dist^2(x^k,\, C))}\left(\dist^2(x^k,\, C) - \dist^2(x^{k+\ell}, \, C)\right) \\
& \ge \frac{1}{\max_{1\le i\le m}\dist^2(x^k,\, C_i)}\left(\dist^2(x^{k},\, C) - \dist^2(x^{k+\ell},\, C)\right)\\
&  \ge a_k.
\end{split}
\end{equation}
Moreover, fixing any integer $k_0\in[0,\, \ell-1]$, for any $N > 0$, summing both sides of \eqref{key} for $k = k_0 + i\ell$ with $i = 0,\ldots, N-1$, we further obtain
\begin{equation*}
\begin{split}
& \Phi_{\widehat{\kappa}}^\spadesuit\left(\dist^2(x^{k_0},\, C)\right) - \Phi_{\widehat{\kappa}}^\spadesuit\left(\dist^2(x^{k_0+N\ell},\, C)\right) \\
= &  \sum_{i=0}^{N-1}\Phi_{\widehat{\kappa}}^\spadesuit\left(\dist^2(x^{k_0+i\ell},\, C)\right) - \Phi_{\widehat{\kappa}}^\spadesuit\left(\dist^2(x^{k_0+(i+1)\ell},\, C)\right) \ge \sum_{i=0}^{N -1} a_{k_0 + i\ell}.
\end{split}
\end{equation*}
This together with the strict monotonicity and continuity on $(0, \, \sup\phi_{\widehat{\kappa},\Phi})$ of $\Phi_{\widehat{\kappa}}^\spadesuit$ (thus invertible), $\dist^2(x^k,\, C) \in (0,\, \sup\phi_{\widehat{\kappa},\Phi})$ and the Fej\'er monotonicity of $\{x^k\}$ further gives
\begin{equation}\label{sum_phi}
\dist(x^{k_0+N\ell},\, C)  \le \sqrt{(\Phi_{\widehat{\kappa}}^\spadesuit)^{-1}\Big(\Phi_{\widehat{\kappa}}^\spadesuit(\dist^2(x^0,\, C) )- \sum_{i=0}^{N -1} a_{k_0 + i\ell}\Big)}.
\end{equation}
Now, we note that for any positive integer $k$ we have $(k\bmod \ell) \geq 0 \geq k_0-\ell$ so that 
\begin{equation*}
k = (k\bmod \ell) + \frac{k - (k \bmod \ell) }{\ell}\cdot\ell \ge k_0  + \frac{k - \ell - (k\bmod \ell) }{\ell}\cdot\ell = k_0 +b_k\cdot\ell.
\end{equation*}
Using this, the Fej\'er monotonicity of $\{x^k\}$ and \eqref{sum_phi}, we see that for any $k\ge 2\ell$ (so that $b_k \geq 1$),
\begin{equation*}
\dist (x^k,\, C) \le \dist(x^{k_0 + b_k\cdot\ell},\, C) \le \sqrt{(\Phi_{\widehat{\kappa}}^\spadesuit)^{-1}\Big(\Phi_{\widehat{\kappa}}^\spadesuit(\dist^2(x^0, \, C) )- \sum_{i=0}^{b_k -1} a_{k_0 + i\ell}\Big)}.
\end{equation*}
This completes the proof.
\end{proof}
Next, we remark that the choice of $\delta$ in the definition of 
$\Phi_{\widehat{\kappa}}^\spadesuit$ has no 
impact in Theorem~\ref{thm_conv}.
\begin{remark}[No dependency on $\delta$ in \eqref{rate}]\label{remark_delta}
Let 
$g:(0,\, a) \to (0,\, \infty)$ be a positive continuous function, where $a>0$ or $a = \infty$.
Let $\delta \in(0,\, a)$ and define  
$f_{\delta}(s) \coloneqq \int _{\delta}^s g(t) dt$, for $s \in(0,\, a)$. With that, 
$f_{\delta}$ is monotone increasing and continuous, thus invertible. 

Let $L = f_{\delta}^{-1}(f_{\delta}(s_0) - c)$ be well-defined with some $s_0 > 0, c\ge 0$. 
We have
\[
-c = f_{\delta}(L) - f_{\delta}(s_0) = \int _{s_0}^{L} g(t) dt  = f_{s_0}(L),
\]
so that $L = f_{s_0}^{-1}(-c)$. This shows that $L$ is constant as a function of $\delta$ and only depends on $c,g$ and $s_0$. Therefore the term inside the square root in  $\eqref{rate}$ only depends on $\Phi$, $\widehat{\kappa}$, $\dist^2(x^0,\, C)$ and  $\sum_{i=0}^{b_k -1} a_{k_0 + i\ell}$ but not on $\delta$.
	
\end{remark}

Before we conclude this subsection, we show that sublinear rates 
can be derived from Theorem~\ref{thm_conv} when $\Phi$ is as in Theorem~\ref{hold_geb}.
\begin{corollary}\label{coro_abs_rate}
Suppose that {\rm Assumption~\ref{assp}} holds with $\inf _{k} a_k > 0$. Suppose that a  H\"olderian error bound defined as in {\rm Definition~\ref{def_hold}} holds. Then the sequence $\{x^k\}$ converges to some point in $C$ at least with a sublinear rate $O(k^{-p})$ for some $p > 0$. 
In particular, if the H\"olderian error bound is uniform with exponent $\gamma\in(0,1]$, then there exist some $M > 0$ and $\theta\in(0,1)$ such that  for any $k\ge 2\ell$,
\begin{equation}\label{rate_abstract}
\dist(x^k,\, C) \le \begin{cases}
M\, k^{-\frac1{2(\gamma^{-1} - 1)}} & {\rm if}\  \gamma\in(0,1),\\
M\, \theta^k & {\rm if} \ \gamma = 1.
\end{cases}
\end{equation}

\end{corollary}

\begin{proof}
The convergence of $\{x^k\}$ follows  from {\rm Assumption~\ref{assp}} and Proposition~\ref{conv_thm}. If the sequence $\{x^k\}$ has finite convergence, one can see that \eqref{rate_abstract} holds for some $M > 0$ and $\theta\in(0,1)$. 
In the following, we consider the case where $\{x^k\}$ does not have finite convergence.

First, assume that a non-uniform H\"olderian error bound holds. From Theorem~\ref{hold_geb}~(i)  the following function is a strict consistent error bound function for the sets $C_1,\ldots, C_m$:
\begin{equation*}
\Phi(a,\, b) \coloneqq \rho(b)\max\{a^{\gamma(b)},\,a\},
\end{equation*}
where $\rho(\cdot)$ is monotone nondecreasing and $\gamma(\cdot)$ is monotone nonincreasing. Let $\Phi_{\widehat{\kappa}}^\spadesuit$ be defined as in \eqref{def_phi} with $\widehat{\kappa}:=  \|x^0\| + 2\,\dist(0,\, C)$. Since $\inf _{k} a_k > 0$, there exists $\tau > 0$ such that $a_k \geq \tau$ for every $k$. Then, from  Theorem~\ref{thm_conv} (setting $k_0 = 0$) and the strict monotonicity of $\Phi_{\widehat{\kappa}}^\spadesuit$ we get that for any $k\ge 2\ell$,
\begin{equation}\label{phi_b}
\begin{aligned}
 \dist (x^k,\, C) &\le \sqrt{(\Phi_{\widehat{\kappa}}^\spadesuit)^{-1}\Big(\Phi_{\widehat{\kappa}}^\spadesuit(\dist^2(x^0,\, C) )- \sum_{i=0}^{b_k - 1} a_{i\ell}\Big)} \\
& \le  \sqrt{(\Phi_{\widehat{\kappa}}^\spadesuit)^{-1}\Big(\Phi_{\widehat{\kappa}}^\spadesuit(\dist^2(x^0,\, C) )- (k/\ell-2)\tau\Big)}.
\end{aligned} 
\end{equation}
Now we calculate the formula of $\Phi_{\widehat{\kappa}}^\spadesuit$. First, we see from  \eqref{def_psi}  that
\begin{equation}\label{phi_comp}
\phi_{\widehat{\kappa}, \Phi}(t) = \left(\Phi(\sqrt{t},\, \widehat{\kappa})\right)^2 = \rho(\widehat{\kappa})^2\max\{t^{\gamma(\widehat{\kappa})},\, t\}.
\end{equation}
Next, we consider two cases depending on the value of $\gamma(\widehat{\kappa})$.

\textbf{Case 1.} $\gamma(\widehat{\kappa})\in(0,1)$. 
In this case, the computation of $\phi_{\widehat{\kappa}, \Phi}^{-}$ is as follows.
\begin{equation*}
\phi_{\widehat{\kappa}, \Phi}^{-}(s) = \begin{cases}
\frac{s}{\rho(\widehat{\kappa})^2} & \text{ if } s \geq \rho(\widehat{\kappa})^2, \\
\frac{1}{\rho(\widehat{\kappa})^{2/\gamma(\widehat{\kappa})}} {s^{\frac{1}{\gamma(\hat\kappa)}}} & \text{ if } 0< s < \rho(\widehat{\kappa})^2.
\end{cases}
\end{equation*}
Next, we compute $\Phi_{\widehat{\kappa}}^\spadesuit$ and we let $\delta \coloneqq \rho(\widehat{\kappa})^2$ in \eqref{def_phi} ($0 < \delta < \sup\phi_{\widehat{\kappa},\Phi} = \infty$), so that 
\begin{equation}\label{eq:phi_spade}
\Phi_{\widehat{\kappa}}^\spadesuit(t) =
\begin{cases}
\frac{\gamma(\widehat{\kappa})}{1 - \gamma(\widehat{\kappa}) } \rho(\widehat{\kappa})^{\frac2{\gamma(\widehat{\kappa})}}\bigg( (\rho(\widehat{\kappa})^2)^{1- \gamma(\widehat{\kappa})^{-1}} - t^{1- \gamma(\widehat{\kappa})^{-1}}\bigg) & \text{ if } 0 < t < \delta,\\
\rho(\widehat{\kappa})^2 (\ln t - 2\ln\rho(\widehat{\kappa})) & \text{ if } t \geq \delta.
\end{cases} 
\end{equation}
Letting  $c_0 \coloneqq \frac{\gamma(\widehat{\kappa})}{1 - \gamma(\widehat{\kappa}) } \rho(\widehat{\kappa})^{\frac2{\gamma(\widehat{\kappa})}}$, we have
\begin{equation}\label{phi_inv}
(\Phi_{\widehat{\kappa}}^\spadesuit)^{-1}(s) = \begin{cases}
\left((\rho(\widehat{\kappa})^2)^{1- \gamma(\widehat{\kappa})^{-1}}- \frac{s}{c_0} \right)^{\frac1{1 - \gamma(\widehat{\kappa})^{-1}}} & \text{ if } s <  0, \\
\rho(\widehat{\kappa})^2 e^{s/{\rho(\widehat{\kappa})^2}} & \text{ if } s \geq  0. \\
\end{cases}
\end{equation}
For simplicity, let $c_1 \coloneqq \Phi_{\widehat{\kappa}}^\spadesuit(\dist^2(x^0,\, C) )+ 2\tau$. From 
\eqref{phi_b}, we have
\begin{equation*}
\dist (x^k,\, C) \le   \sqrt{(\Phi_{\widehat{\kappa}}^\spadesuit)^{-1}\Big(c_1 - \frac{k\tau}{\ell} \Big)}.
\end{equation*}
Therefore, if  $k > \frac{\ell c_1}{\tau}$ and $k \geq 2\ell$, we have
\begin{equation}\label{eq:hold_rate}
\begin{aligned}
 \dist (x^k,\, C) &\le
 \left((\rho(\widehat{\kappa})^2)^{1- \gamma(\widehat{\kappa})^{-1}}- \frac{c_1}{c_0} + \frac{k\tau}{\ell c_0}\right)^{-\frac1{2(\gamma(\widehat{\kappa})^{-1} - 1)}}\\ &\le M\, k^{-\frac1{2(\gamma(\widehat{\kappa})^{-1} - 1)}},
\end{aligned}
\end{equation}
holds for some $M > 0$. This proves the sublinear convergence rate of $\{x^k\}$\footnote{We note that for the $x_k$ such that $k \geq 2\ell$ but 
$k \le \frac{\ell c_1}{\tau}$, the rate for those iterates is governed by 
the second expression in \eqref{phi_inv}, so overall, we have a sublinear convergence rate 
for all $k \geq 2\ell$.}.

\textbf{Case 2.} $\gamma(\widehat{\kappa}) = 1$. 
For this case, it will be more convenient to use $\delta \coloneqq 1$ in 
\eqref{def_phi}. Then, from \eqref{def_phi} and \eqref{phi_comp} we have
\begin{equation}
\Phi_{\widehat{\kappa}}^\spadesuit(t) = \int_{1}^t\frac1{\phi_{\widehat{\kappa}, \Phi}^{-}(s)}ds = \rho(\widehat{\kappa})^2\int_{1}^t s^{-1} ds = \rho(\widehat{\kappa})^2 \ln t .\label{eq:phi_spade_lin}
\end{equation}
Let $c_2:= \rho(\widehat{\kappa})^2$. Then, we have $(\Phi_{\widehat{\kappa}}^\spadesuit)^{-1}(t) = e^{t/c_2}$ and 
\begin{equation*}
\dist(x^k,\, C) \le \sqrt{(\Phi_{\widehat{\kappa}}^\spadesuit)^{-1}\Big(\Phi_{\widehat{\kappa}}^\spadesuit(\dist^2(x^0,\, C) )- (k/\ell-2)\tau\Big)}  = e^{\tau/c_2}\dist(x_0,\, C)\cdot e^{-\frac{\tau}{2\ell c_2}k},
\end{equation*}
which proves the linear convergence rate of $\{x^k\}$. 
This concludes the proof for the non-uniform case.

If  the H\"olderian error bound is uniform with exponent $\gamma\in(0,1]$, 
the function $\Phi$ is as in \eqref{phi_hold_u}, so the $\max$ term in \eqref{phi_comp} becomes $t^\gamma$ and there is no need 
to divide the computation of $\Phi_{\widehat{\kappa}}^\spadesuit$ and $(\Phi_{\widehat{\kappa}}^\spadesuit)^{-1}$ in two cases.
In particular,  \eqref{eq:phi_spade} and \eqref{phi_inv} become simpler since the second case in each expression is discarded. Then, \eqref{rate_abstract} follows from a similar line of arguments\footnote{The only subtlety is that  in the proof of Case $1$ in the uniform case, \eqref{eq:hold_rate} holds for all $k \geq 2\ell$ and there is no need to impose $k > \ell c_1/\tau$. } as above, replacing $\gamma(\widehat{\kappa})$ by $\gamma$.
This completes the proof.
\end{proof}

\subsection{Projection algorithms}\label{sec:proj_alg}
 In the following,  we consider  an algorithm scheme contained in the broader framework given in Section~3 of \cite{BB96}. Specifically, given $x^0\in\mathcal{E}$, \emph{relaxation parameter} $\{\alpha_i^k\}\subseteq[0,2)$ and \emph{weight} $\{\lambda_i^k\}$ satisfying $\sum_{i=1}^m\lambda_i^k = 1$ with $\lambda_i^k\ge 0$ for all $k$, we consider the following algorithm scheme:
\begin{equation}\label{proj_alg}
x^{k + 1} = \sum_{i = 1}^m\lambda_i^k\Big[(1 - \alpha_i^k) I + \alpha_i^k P_{C_i}\Big](x^k),
\end{equation}
where $I$ denotes the identity operator and $P_{C_i}$ is the orthogonal projection operator onto $C_i$.

\begin{example}\label{remark_proj}
Here are a few examples of algorithms covered under the algorithm scheme \eqref{proj_alg}. 
\begin{enumerate}[$(a)$]
	\item \emph{Mean projection algorithm} (MPA)(\cite{BB96,BT03,gp72}):  $\alpha_i^k = 1$ for all $i$ and $k$, and the weights $\lambda_i^k$ ($i = 1,\ldots,m$) are positive constants for all $k$. When $\lambda_i^k = \nu_i > 0$ for every $i$ and $k$ with $\sum_{i=1}^m\nu_i = 1$, the iterations are of the format 
       \begin{equation*}
	x^{k+1} = \sum _{i=1}^m  \nu_i P_{C_i}(x^k).
	\end{equation*}

	\item \emph{Projections onto convex sets algorithm} (POCSA)(\cite{censor81,cpl96,HLL78,YW82}):  Let $t(k): = (k \bmod  m) + 1$. For every $k$, set $\lambda_i^k = 1$ and $\epsilon \le \alpha_i^k\le 2-\epsilon$  with $\epsilon\in(0,\,1)$  when $i = t(k)$, and set $\lambda_i^k = 0$ when $i \neq t(k)$ ($\alpha_i^k$ can be arbitrarily defined in this case).
	The iterations are of the format
	  \begin{equation*}
	x^{k+1} = \left(1 - \alpha_{t(k)}^k\right)x^k + \alpha_{t(k)}^k P_{C_{t(k)}}(x^k).
     \end{equation*}
Especially, when $\alpha_{t(k)}^k \equiv 1$ for all $k$, it reduces to $x^{k+1} = P_{C_{t(k)}}(x^k)$, which is the well-known \emph{Cyclic projection algorithm} (CPA),  see \cite{as54,BB96,BT03,BLY14}.
	
	\item \emph{Motzkin’s method} (MM)(\cite{as54,LHN17,ms54}): Fix any $i(k)\in\Argmax_{1\le i\le m}\dist(x^k, C_i)$. For every $k$, let $\lambda_{i}^k= 1$  and $\alpha_i^k = \lambda$ with $\lambda\in(0,\,2)$ for $i = i(k)$, and $\lambda_i^k = 0$ for $i\neq i(k)$ ($\alpha_i^k$ can be arbitrarily defined in this case).
	The iterations are of the format 
	\begin{equation*}
	x^{k+1} = (1 - \lambda)x^k + \lambda P_{C_{i(k)}}(x^k).
	\end{equation*}
Especially, when $\lambda = 1$, it reduces to $x^{k+1} =  P_{C_{i(k)}}(x^k)$, which is known as \emph{Maximum distance projection algorithm} (MDPA), see \cite{BB96,BT03}.

	\item The following \emph{adaptive weighted projection algorithm} (AWPA): $\alpha_i^k = 1$ for all $i$ and $k$, and the weights $\lambda_i^k$ ($i = 1,\ldots,m$) are adaptively chosen. Let $f:[0,+\infty) \to [0,+\infty)$ be a monotone increasing nonnegative function such that $f(0) = 0$. Define $d_i^k := \dist(x^k,\, C_i)$ and let $\lambda_i^k = \frac{f(d_i^k)}{f(d_1^k)+\cdots+f(d_m^k)}$.
	The iterations are of the format 
      \begin{equation*}
	x^{k+1} = 
	\sum_{i=1}^m\frac{f(d_i^k)}{f(d_1^k)+\cdots+f(d_m^k)}P_{C_i}(x^k),
	\end{equation*}
	if at least one of the $d_i^k$ is nonzero.
	This is related  to a generalization of Ansorge's method discussed in Example 6.32 in \cite{BB96}. A particular case is the following iteration
	\begin{equation*}
	x^{k+1} = \sum_{i=1}^m\frac{d_i^k}{d_1^k+\cdots+d_m^k}P_{C_i}(x^k). 	
	\end{equation*}

	For analysis purposes and in order for the iteration to be well-defined for all $k$  we consider that if $d_i^k = 0$ for all $i$ (i.e., $x^k \in C$), then 
	AWPA falls back to the following MPA iteration: $x^{k+1} = 
	\sum_{i=1}^m\frac{1}{m}P_{C_i}(x^k)$.

\end{enumerate}
\end{example}

Now we show that the sequence generated by scheme \eqref{proj_alg} satisfies Assumption~\ref{assp} under some conditions on the parameters.  For that, we introduce the following notation:
\begin{equation}\label{def_notation}
\begin{aligned}
M(k) &\coloneqq \left\{ i \mid  i\in \Argmax_{1\le i\le m}\dist(x^k,\, C_i)\right\}, \\  I_{\sigma}(k) &\coloneqq \left\{ i \mid \lambda_i^k \ge \sigma\right\}.
\end{aligned}
\end{equation}

\begin{lemma}[Checking Assumption~\ref{assp}]\label{lemma_1} 
Let the sequence $\{x^k\}$ be generated by \eqref{proj_alg}. Then  $\{x^k\}$ is Fej\'er monotone with respect to~$C$, i.e., {\rm Assumption}~\ref{assp}~$(\ref{assp:h1})$ holds. Let 
\begin{equation}\label{mu_def}
\mu_i^k: = \alpha_i^k\lambda_i^k(2 - \sum_{j=1}^m\alpha_j^k\lambda_j^k), \ \ \  i = 1, \ldots, m.
\end{equation} 
Then it holds for all $k$ that
\begin{equation}\label{dist_dec}
 \dist^2(x^k,\, C)  \ge \dist^2(x^{k + 1},\, C) + \sum_{i = 1}^m\mu_i^k\,\dist^2(x^k,\, C_i).
 \end{equation}
 Moreover, the following statements hold.
\begin{itemize}

\item[(i)]  If there exists $m(k)\in M(k)$ such that $\sum_{k=0}^{\infty}\mu_{m(k)}^k = \infty$,  then {\rm Assumption~\ref{assp}~$(\ref{assp:h2})$} holds with $\ell = 1$ and $a_k = \mu_{m(k)}^k$ in  inequality \eqref{h2}.

\item[(ii)] If $\alpha_i^k\in[\alpha_1,\,\alpha_2]$ holds for all $i$ and $k$ with some $0 <\alpha_1 \le \alpha_2 < 2$, and there exist some $\sigma \in(0, 1]$ and integer $s\ge 1$ such that for all k,
\begin{equation}\label{cycle}
I_{\sigma}(k) \cup I_{\sigma}(k+1) \cup \cdots \cup I_{\sigma}(k+s - 1) = \{1, 2, \ldots, m\},
\end{equation}
then {\rm Assumption~\ref{assp}~$(\ref{assp:h2})$} holds with $\ell = s$ and $a_k = \min\left(\frac{\sigma\alpha_1(2 - \alpha_2)}{s},\,\frac{\alpha_1(2 - \alpha_2)}{(\alpha_2)^2 s}\right)$ in  inequality \eqref{h2}.
  \end{itemize}
\end{lemma}

\begin{proof}
The scheme \eqref{proj_alg} is a particular case of the the scheme 
described in Section~3 of \cite{BB96} (with $T_i^k = P_{C_i}$). Consequently, the Fej\'er monotonicity of $\{x^k\}$ follows directly from \cite[Lemma~3.2~(iv)]{BB96}. Moreover, by \cite[Lemma~3.2~$(i)$]{BB96}, we have for any $x\in C$ that
\begin{equation}\label{inq_comp}
\begin{split}
& \left\|x^k - x\right\|^2 - \left\|x^{k+1} - x\right\|^2 - \sum_{i=1}^m\alpha_i^k\lambda_i^k\Big(2 - \sum_{j=1}^m\alpha_j^k\lambda_j^k\Big)\Big\|x^k - P_{C_i}(x^k) \Big\|^2 \\
= &  \sum_{i<j}\alpha_i^k\alpha_j^k\lambda_i^k\lambda_j^k\Big\| P_{C_i}(x^k) - P_{C_j}(x^k)\Big\|^2 + 2\sum_{i=1}^m\alpha_i^k\lambda_i^k\left\langle x^k - P_{C_i}(x^k),\, P_{C_i}(x^k) - x \right\rangle \ge 0,
\end{split}
\end{equation}
where the last inequality follows from the non-negativity of $\{\alpha_i^k\}$ and $\{\lambda_i^k\}$ and the convexity of each $C_i$. We then have \eqref{dist_dec} by rearranging \eqref{inq_comp} and taking the infimum on both sides for $x\in C$. Furthermore, by the definition of $M(k)$ in \eqref{def_notation}, we have for all $m(k)\in M(k)$ that
\begin{align*}
 \dist^2(x^k,\, C)  &\ge \dist^2(x^{k + 1},\, C) + \sum_{i = 1}^m\mu_i^k\,\dist^2(x^k, \, C_i) \\ 
 &\ge \dist^2(x^{k + 1},\, C) + \mu_{m(k)}^k\max_{1\le i\le m}\dist^2(x^k,\, C_i).
\end{align*}
The conclusion $(i)$ then follows from this and assumption $\sum_{k=0}^{\infty}\mu_{m(k)}^k = \infty$ directly. 

Now we prove $(ii)$. Since $\alpha_i^k\in[\alpha_1,\,\alpha_2]$ for all $i$ and $k$, we have $ \mu_i^k\ge\alpha_1(2 - \alpha_2)\lambda_i^k$. Consequently, by \eqref{dist_dec}, the convexity of $\|\cdot\|^2$ and $\sum_{i=1}^m\lambda_i^k = 1$ with $\lambda_i^k\ge 0$, we have for all $k$ that
\begin{equation}\label{succ_change}
\begin{split}
 \left\| x^k - x^{k+1}\right\|^2 & = \Big\| x^k - \sum_{i=1}^m \lambda_i^k\left((1 - \alpha_i^k)x^k +\alpha_i^k P_{C_i}(x^k)\right)\Big\|^2 \\
 & = \Big\| \sum_{i=1}^m\lambda_i^k\alpha_i^k\big( x^k -P_{C_i}(x^k)\big)\Big\|^2  \le \sum_{i=1}^m\lambda_i^k(\alpha_i^k)^2\big\| x^k -P_{C_i}(x^k)\big\|^2 \\
 & \le  (\alpha_2)^2\sum_{i=1}^m\lambda_i^k \dist^2(x^k,\, C_i)  \le  \frac{(\alpha_2)^2}{\alpha_1(2 - \alpha_2)}\left(\dist^2(x^k,\, C)  - \dist^2(x^{k + 1},\, C)\right).
\end{split}
\end{equation}
On the other hand, we fix any $k$ and $j\in\{1, 2, \ldots, m\}$, and then know from assumption \eqref{cycle} that there exists $k_j\in\{k, k+1, \ldots, k+s-1\}$ such that $j\in I_{\sigma}(k_j)$, i.e., $\lambda_j^{k_j}\ge \sigma$ (by definition of $I_{\sigma}(k)$ in \eqref{def_notation}).  This together with \eqref{dist_dec} and $\mu_i^k \ge \alpha_1(2 - \alpha_2)\lambda_i^k$ gives
\begin{equation}\label{ineq_tau}
 \dist^2(x^{k_j},\, C)  - \dist^2(x^{k_j + 1},\, C) \ge \sum_{i = 1}^m\mu_i^{k_j}\,\dist^2(x^{k_j},\, C_i) \ge \sigma\alpha_1(2 - \alpha_2)\, \dist^2(x^{k_j},\, C_j).
\end{equation}
Furthermore, combining \eqref{succ_change} and \eqref{ineq_tau} yields
\begin{equation}\label{case_2_dec}
\begin{aligned}
& \ \ \ \   \dist^2(x^k,\, C_j) \\
&\le \norm{x^k - P_{C_j}(x^{k_j})}^2\\
&\overset{\rm (a)}\le \big(\dist(x^{k_j},\, C_j) + \left\|x^k - x^{k_j}\right\| \big)^2 \\
&\overset{\rm (b)}\le \Big(\dist(x^{k_j},\, C_j) + \sum_{p=k}^{k_j-1}\left\|x^p - x^{p+1}\right\| \Big)^2\\
& \overset{\rm (c)}\le (k_j - k + 1)\Big(\dist^2(x^{k_j},\, C_j) + \sum_{p=k}^{k_j-1}\left\|x^p - x^{p+1}\right\|^2 \Big)\\
& \overset{\rm (d)}\le s\,\Big(\frac1{\sigma\alpha_1(2 - \alpha_2)}\left(\dist^2(x^{k_j},\, C)  - \dist^2(x^{k_j + 1},\, C)\right) + \frac{(\alpha_2)^2}{\alpha_1(2 - \alpha_2)}\sum_{p=k}^{k_j-1}\left(\dist^2(x^p,\, C)  - \dist^2(x^{p + 1},\, C)\right)\Big)\\
& =  s\,\Big(\frac1{\sigma\alpha_1(2 - \alpha_2)}\left(\dist^2(x^{k_j},\, C)  - \dist^2(x^{k_j + 1},\, C)\right) +  \frac{(\alpha_2)^2}{\alpha_1(2 - \alpha_2)}\left(\dist^2(x^{k},\, C)  - \dist^2(x^{k_j},\, C)\right)\Big)\\
& \le   s\,\max\Big(\frac1{\sigma\alpha_1(2 - \alpha_2)},\,\frac{(\alpha_2)^2}{\alpha_1(2 - \alpha_2)}\Big)\left(\dist^2(x^{k},\, C) - \dist^2(x^{k_j + 1},\, C)\right)\\
& \overset{\rm (e)}\le  s\,\max\Big(\frac1{\sigma\alpha_1(2 - \alpha_2)},\,\frac{(\alpha_2)^2}{\alpha_1(2 - \alpha_2)}\Big)\left(\dist^2(x^{k},\, C)  - \dist^2(x^{k + s},\, C)\right),
\end{aligned}
\end{equation}
where (a) and (b) follow from the triangle inequality,
(c) follows from the Cauchy-Schwarz inequality, 
(d) holds because of \eqref{succ_change},
\eqref{ineq_tau} and $k_j\in\{k, k+1, \ldots, k+s-1\}$,
finally, (e) follows from the Fej\'er monotonicity of $\{x^k\}$ and the fact that $k \leq k_j \leq k+s-1$. By the arbitrariness of $j$, we take the supreme on both sides of \eqref{case_2_dec} for $j\in\{1,2,\ldots,m\}$ and rearrange it to obtain
\begin{equation*}
\dist^2(x^{k},\, C)  - \dist^2(x^{k + s},\, C) \ge \min\left(\frac{\sigma\alpha_1(2 - \alpha_2)}{s},\,\frac{\alpha_1(2 - \alpha_2)}{(\alpha_2)^2 s}\right)\max_{1\le j\le m}\dist^2(x^k,\, C_j).
\end{equation*}
Therefore, {\rm Assumption~\ref{assp}~$(\ref{assp:h2})$} holds with $\ell = s$ and $a_k = \min\left(\frac{\sigma\alpha_1(2 - \alpha_2)}{s},\,\frac{\alpha_1(2 - \alpha_2)}{(\alpha_2)^2 s}\right)$.
\end{proof}
The gist of Lemma~\ref{lemma_1} is that any iteration generated by \eqref{proj_alg} is automatically Fej\'er monotone, which is a known result, see  \cite[Lemma~3.2]{BB96}. However, not all choices of 
parameters will lead to sufficient decrease as required in Assumption~\ref{assp}~$(\ref{assp:h2})$ (e.g., if $\alpha_{i}^k = 0$ for all $i$ and $k$).
There are many conditions one can impose on the choice of parameters to get sufficient decrease and items~$(i)$ and $(ii)$ of Lemma~\ref{lemma_1} are but two simple examples that are enough to cover a number of algorithms, as we shall see. In particular, $(ii)$ in case of $\alpha_1 = \alpha_2 = 1$ is a simplified version of the assumption underlying the so-called quasi-cyclic algorithms, see \cite{blt17}.

The next step is to apply Theorem~\ref{thm_conv} to the algorithms 
covered by Lemma~\ref{lemma_1}.
We conclude that the convergence of $\{x^k\}$ is either finite or,
if item $(i)$ of Lemma~\ref{lemma_1} holds, we have
\begin{equation}\label{case_pg1}
\dist(x^k,\, C)\leq \sqrt{(\Phi_{\widehat{\kappa}}^\spadesuit)^{-1}\Big(\Phi_{\widehat{\kappa}}^\spadesuit(\dist^2(x^0,\, C) )- \sum_{i=0}^{k -2} \mu_{m(i)}^i\Big)}\  \forall\ k\ge 2.
\end{equation} 
Alternatively, if item $(ii)$ of Lemma~\ref{lemma_1} holds, we have
\begin{equation}\label{case_pg2}
\dist(x^k,\, C) \le 
\sqrt{(\Phi_{\widehat{\kappa}}^\spadesuit)^{-1}\Big(\Phi_{\widehat{\kappa}}^\spadesuit(\dist^2(x^0,\, C) )- c(k-s-(k \bmod s))/s\Big)}\  \forall\ k\ge 2s,
\end{equation}
where $c = \min\left(\frac{\sigma\alpha_1(2 - \alpha_2)}{s},\,\frac{\alpha_1(2 - \alpha_2)}{(\alpha_2)^2 s}\right)$. 

Next, we will see that  more specific choices of parameters will lead to sublinear convergence rates under H\"olderian error bounds as in Corollary~\ref{coro_abs_rate}. 

\begin{corollary}[H\"olderian error bounds and sublinear rates for projection algorithms]\label{proj_hold_ge}
	Let $\{x^k\}$ be generated by the algorithm scheme \eqref{proj_alg}. Suppose that one of the following statements holds:
	\begin{itemize}
		\item[(i)] there exist some $\tau > 0$ and $m(k)\in M(k)$ such that $\mu_{m(k)}^k \ge \tau$ for all $k$, where $\mu_i^k$ is defined as in \eqref{mu_def};
		
		\item[(ii)] $\alpha_i^k\in[\alpha_1,\,\alpha_2]$ holds for all $i$ and $k$ with some $0 <\alpha_1 \le \alpha_2 < 2$, and there exist some $\sigma \in(0, 1]$ and integer $s\ge 1$ such that \eqref{cycle} holds for all k.
		
	\end{itemize}
	If a  H\"olderian error bound holds for \eqref{CFP}, then  $\{x^k\}$ converges to some point in $C$ at least with a sublinear rate $O(k^{-p})$ for some $p > 0$. 
	In particular, if the H\"olderian error bound is uniform with exponent $\gamma\in(0,1]$, then there exist some $M > 0$ and $\theta\in(0,1)$ such that  for any $k\ge 2s$ ($k\ge 2$ if (i) holds),
	\begin{equation*}
	\dist(x^k,\, C) \le \begin{cases}
	M\, k^{-\frac1{2(\gamma^{-1} - 1)}} & {\rm if}\  \gamma\in(0,1),\\
	M\, \theta^k & {\rm if} \ \gamma = 1.
	\end{cases}
	\end{equation*}
\end{corollary}
\begin{proof}
Item $(i)$ and $(ii)$ imply items $(i)$ and $(ii)$ of Lemma~\ref{lemma_1}, respectively.	
In both cases, there exists $\nu > 0$ such that the sufficient decrease inequality \eqref{h2} holds with $a_k \geq \nu$ for every $k$. Therefore, the conditions of Corollary~\ref{coro_abs_rate} are met and the conclusion follows.
\qed\end{proof}

With the aid of the results so far, we can check that Assumption~\ref{assp} holds for the algorithms listed in Example~\ref{remark_proj} and compute their convergence rates.
\begin{theorem}[Convergence of a few common methods]\label{thm_sp_4}
Let $\{x^k\}$ be a sequence generated by one of the four algorithms MPA, POCSA (in particular, CPA), MM (in particular, MDPA) and AWPA given in Example~\ref{remark_proj}.	
The following items holds.
\begin{enumerate}[$(i)$]
	\item Assumption~\ref{assp} is satisfied. In particular, if $\Phi$ is a strict consistent error bound function for $C_1,\ldots, C_m$ and $\Phi_{\widehat{\kappa}}^\spadesuit$ is as in \eqref{def_phi} with $\hat \kappa = \norm{x^0} +2 \, \dist(0,C)$, the convergence rates of MPA, MM (in particular, MDPA), AWPA are governed by \eqref{case_pg1}. The convergence rate of POCSA (in particular, CPA) is governed by 
	\eqref{case_pg2}.
	\item Suppose that a H\"olderian error bound holds.
	 Then $\{x^k\}$ converges to some point in $C$ at least with a sublinear rate $O(k^{-p})$ for some $p > 0$. In particular, if the H\"olderian error bound is uniform with exponent $\gamma\in(0,1]$, then there exist some $M > 0$ and $\theta\in(0,1)$ such that  for any $k\ge 2m$ ($k\ge 2$ for MPA, MDPA and AWPA),
	\begin{equation*}
	\dist(x^k,\, C) \le \begin{cases}
	M\, k^{-\frac1{2(\gamma^{-1} - 1)}} & {\rm if}\  \gamma\in(0,1),\\
	M\, \theta^k & {\rm if} \ \gamma = 1.
	\end{cases}
	\end{equation*}
	
\end{enumerate}
\end{theorem}
\begin{proof}
First, we check item $(i)$.
By Lemma~\ref{lemma_1}, it suffices to check Assumption~\ref{assp}~$(\ref{assp:h2})$ for the four algorithms. 
For $i = 1, \ldots, m$, let
\[\mu_i^k: = \alpha_i^k\lambda_i^k(2 - \sum_{j=1}^m\alpha_j^k\lambda_j^k).
\]  
Then there exists some $m(k)\in M(k)$ such that for MPA, MM (in particular, MDPA) and AWPA we have 
\[\mu_{m(k)}^k = \nu_{m(k)} \ge \min_{1\le i\le m}\nu_i > 0, \qquad \mu_{m(k)}^k = \lambda(2 - \lambda), \qquad \mu_{m(k)}^k \ge \frac1{m}, \]  respectively. Consequently, we have $\sum_{k=0}^{\infty}\mu_{m(k)}^k = \infty$. Therefore,  from Lemma~\ref{lemma_1}~(i) we
see that Assumption~\ref{assp}~$(\ref{assp:h2})$ holds with $\ell =1$ and $a_k = \mu_{m(k)}^k \ge \tau$ for some $\tau > 0$ for MPA, MM and AWPA.

For POCSA, the assumptions in Lemma~\ref{lemma_1}~(ii) are satisfied with $\sigma = 1$, $s = m$, $\alpha_1 = \epsilon$ and $\alpha_2 = 2 - \epsilon$. Thus, POCSA (in particular,  CPA) satisfies \eqref{h2} with $\ell = m$ and $a_k  = \min\big(\frac{\epsilon^2}{m},\,\frac{\epsilon^2}{(2 - \epsilon)^2 m}\big)$. With that, we have $\sum_{k=0}^{\infty}a_k = \infty$. This completes the proof of item $(i)$.

Next, we move on to item $(ii)$. In all cases, 
the conditions in Corollary~\ref{proj_hold_ge} are met. 
Therefore, we can deduce the corresponding sublinear rates.	
\end{proof}



\begin{remark}
{\rm (Connection to existing convergence rates)} Theorem~\ref{thm_sp_4} recovers several existing convergence results. For example, it recovers the linear convergence result for MPA,  POCSA and MM under a Lipschitzian error bound established in \cite[Theorem~2.2]{BT03}, \cite[Theorem~3]{YW82} and \cite[Section~4]{as54}, respectively.
In particular, it recovers the sublinear convergence result for CPA under a H\"olderian error bound established in \cite[Proposition~4.2]{BLY14}.  
It also recovers the sublinear convergence rate for MPA and MDPA under a H\"olderian error bound, which could be obtained by \cite[Theorem~3.3]{blt17} and \cite[Corollary~3.8]{blt17}.
To the best of our knowledge, however, the sublinear rate for AWPA is new since it is not clear if the operator associated to it satisfies the conditions necessary to invoke the results in \cite{blt17}. 
\end{remark}

\section{Regular variation and comparison of convergence rates}\label{sec:rv}
Given a strict consistent error bound function $\Phi$ and some algorithm as in Section~\ref{sec:proj},  the convergence rate is governed by a fairly complicated expression depending on the inverse of the function $\Phi_{\kappa}^\spadesuit$ defined in \eqref{def_phi}, see Theorem~\ref{thm_conv}.  In this section, we provide a number of results that help  to reason about $\Phi_{\kappa}^\spadesuit$ and its inverse without actually having to compute them. 
The main tool we use is the notion of \emph{regular variation} \cite{Se76,BGT87}.

Regular variation will be helpful  because it provides  tools to analyze the asymptotic properties of functions once the so-called \emph{index of regular variation} is known, e.g., Potter's bounds (see \eqref{eq:potter_rv}). Furthermore,  it is well-understood how regular variation behaves under taking integrals, inverses, applying powers and so on, which are exactly the transformations used to obtain $(\Phi_{\kappa}^\spadesuit)^{-1}$  from the original consistent error bound function $\Phi$. With that, it is possible  to obtain bounds to $(\Phi_{\kappa}^\spadesuit)^{-1}$  without having to actually compute a closed-form expression for $(\Phi_{\kappa}^\spadesuit)^{-1}$.
We will showcase this in Theorems~\ref{thm_comp}, \ref{thm:upper} and also with a general analysis of logarithmic error bounds in Section~\ref{sec:log} and error bounds for the exponential cone in Section~\ref{sec:exp}. 


Let $\Phi$ be a function that satisfies 
items $(ii)$ and $(iii)$ of Definition~\ref{GEB} but not necessarily item $(i)$. That is, $\Phi$ is not necessarily 
related to any collection of convex sets $C_1, \ldots, C_m$.
In this case, we shall drop the adjective ``consistent'' and merely 
say that $\Phi$ is an \emph{error bound function}.
If $\Phi(\cdot, b)$ is monotone increasing for every $b > 0$, we say that $\Phi$ is a \emph{strict error bound function}.

In spite of the fact that $\Phi$ might not be attached to any particular intersection of convex sets, we can still define $\phi_{\kappa,\Phi}$ and  $\Phi_{\kappa}^{\spadesuit}$ as in \eqref{def_psi} and \eqref{def_phi}, respectively.
Let $\Phi$ and $\widehat \Phi$ be strict error bound functions. First, we will show how to draw 
conclusions about the order relationship between $(\Phi_{\kappa}^{\spadesuit})^{-1}$ and $(\widehat{\Phi}_{\kappa}^{\spadesuit})^{-1}$ using 
the order relationship between $\Phi$ and $\widehat \Phi$.
The motivation is that, given a particular $\Phi$ we would like to know whether the convergence rate afforded by $\Phi$ is faster or slower than, say,  a linear or a sublinear rate without having to compute $(\Phi_{\kappa}^{\spadesuit})^{-1}$.

We start with some basic aspects of 
the theory  of regular variation in the sense of Karamata~\cite{Se76,BGT87}.
\begin{definition}[Regularly varying functions]\label{def:rv}
	A function $f: [a,\,\infty)\rightarrow (0,\,\infty)\, (a > 0)$
	is said to be \emph{regularly varying at infinity} if it is measurable and there exists a real number $\rho$ such that 
	\begin{equation}\label{eq:rv}
	\lim_{x\rightarrow\infty}\frac{f(\lambda x)}{f(x)} = \lambda^{\rho}, \ \ \ \forall\ \lambda > 0.
	\end{equation}
	In this case, we write  $f\in \RV$. 
	Similarly, a measurable function $f:(0,a]\rightarrow(0,\,\infty)$ is said to be regularly varying at $0$ if 
	\begin{equation}\label{eq:rvz}
	\lim_{x\rightarrow 0_+}\frac{f(\lambda x)}{f(x)} = \lambda^{\rho}, \ \ \ \forall\ \lambda > 0,
	\end{equation}
	in which case we write $f \in \RVz$. The $\rho$ in \eqref{eq:rv} and \eqref{eq:rvz} is called the \emph{index} of regular variation. 
	
If the limit on the left hand side of \eqref{eq:rv} is $0$, $1$ and $+\infty$ for $\lambda$ in $(0,1)$, $\{1\}$ and $(1,\infty)$, respectively, then $f$ is said to be a function of \emph{rapid variation of index $\infty$} and 
we write $f \in \RV_{\infty}$. If $1/f \in \RV_{\infty}$, we say that $f$ is a function of \emph{rapid variation of index $-\infty$} and write $f \in \RV_{-\infty}$.
\end{definition}
The $a$ in Definition~\ref{def:rv} only plays a minor role, since we are interested in what happens when $f$ approaches the opposite side the interval. By an abuse of notation, we sometimes write ``$f \in \RV$'' meaning that $f$ restricted to some interval $[a,\infty)$ (with $a> 0$) satisfies Definition~\ref{def:rv}. We will do the same for $\RVz, \RV_{-\infty}$ and $\RV_{\infty}$.

Next, we need to discuss the behavior of 
the index of regular variation under taking inverses. 
For a monotone nondecreasing function $f: [a,\,\infty)\rightarrow (0,\,\infty)$, we define the following generalized inverse $f^{\leftarrow}(x) \coloneqq \inf\{y  \geq a \mid f(y) > x\}$.  In particular the following result holds
\begin{equation}\label{eq:rv_inv}
\begin{split}
f \in \RV \text{ with index } \rho > 0 & \Rightarrow f^{\leftarrow} \in \RV \text{ with index } 1/\rho,\\
f \in\RV\text{ with index } 0 \text{ and } f\text{ is unbounded}   & \Rightarrow f^{\leftarrow} \in \RV_{\infty},
\end{split}
\end{equation}	
see \cite[Theorem~1.5.12]{BGT87} and \cite[Proposition  2.4.4~ item(iv) and Theorem~2.4.7]{BGT87}, respectively.
Note that if $f$ is continuous and monotone increasing, then 
$f^{\leftarrow} = f^{-1}$. 

In this section, in order to avoid dealing with the differences between $f^{\leftarrow}$, $f^{-1}$
and  $f^{-}$, we  assume that the functions are all monotone increasing and continuous so that all the three inverses coincide at the points at which they are defined. This will be mentioned as needed.

Now, suppose that $f \in \RVz$ with index $\rho > 0$ is continuous monotone increasing and 
define $\hat f$ by $\hat f(x) = 1/f(1/x)$.
For $\lambda > 0$,
\begin{equation}\label{eq:rv_eq}
\lim _{x\to\infty}\frac{\hat f(\lambda x)}{\hat f(x)} = \lim _{x\to\infty}\frac{ f(1/x)}{f(1/(\lambda x))} = \lim _{t\to 0_{+}} \frac{f(\lambda t)}{f(t)} = \lambda^\rho.
\end{equation}
Therefore, $\hat f \in \RV$ with index $\rho$ and \eqref{eq:rv_inv} implies that $\hat f^{-1}$ has index $1/\rho$. Since $\hat f^{-1}(x) =  1/f^{-1}(1/x)$,  we conclude that 
\begin{equation}\label{eq:rvz_inv}
f \in \RVz \text{ with index } \rho > 0 \Rightarrow f^{-1}  \in \RVz \text{ with index } 1/\rho, 
\end{equation}	
when $f$ is monotone increasing and continuous.

We start with the following lemma, which is a particular case of  \cite[Theorem~1]{dt07}.
In what follows, if $f$ and $g$ are functions such that  $\lim _{t \to c} f(t)/g(t) = 0$ we will write that ``$f(t) = o(g(t))$ as $t \to c$''. We will consider three cases: 
$c \in \{-\infty$, $+\infty\}$ or that $t$ approaches $0$ from the right, which we will denote by writing $c = 0_+$. 

\begin{lemma}\label{key_lemma}
Assume that $f,\,g: [a,\,\infty)\rightarrow (0,\,\infty)\, (a > 0)$ are continuous monotone increasing unbounded functions, and $f\in \RV$ or $g\in \RV$. If $f(x) = o(g(x))$ as $x\rightarrow\infty$, then $g^{-1}(x) = o(f^{-1}(x))$ as $x\rightarrow\infty$.
\end{lemma}
\begin{proof}
Theorem~1 of  \cite{dt07} states that if $f,g:[a,\,\infty)\rightarrow (0,\,\infty)\, (a > 0)$ are monotone increasing unbounded functions such that $f(x) = o(g(x))$ as $x\rightarrow\infty$ and at least one among $f,g$ belongs to $\RV$ then \[
g^{\leftarrow}(x) = o(f^{\leftarrow}(x)).
\]
 Under the hypothesis that $f,g$ are continuous and monotone increasing we have $f^{\leftarrow} = f^{-1}$ and 
$g^{\leftarrow} = g^{-1}$, so the result follows.
\end{proof}

Using Lemma~\ref{key_lemma}, we establish the following comparison theorem.
\begin{theorem}\label{thm_comp}
Let $\kappa > 0$ and $\Phi$ and $\widehat{\Phi}$ be two strict error bound functions  satisfying:
\begin{enumerate}[$(i)$]
	\item\label{a1}  $\Phi(\cdot,\,\kappa)$ and $\widehat{\Phi}(\cdot,\,\kappa) $ are continuous,
	\item\label{a3} $\Phi_{\kappa}^\spadesuit(t)\rightarrow -\infty$ and 
	$\widehat{\Phi}_{\kappa}^\spadesuit(t)\rightarrow -\infty$ 	
	 as  $t\rightarrow 0_+$.
\end{enumerate}
Then, the following statements hold.
\begin{enumerate}[$(a)$]
	\item\label{c1} If 	$\Phi(\cdot,\,\kappa)$ belongs to $\RVz$ with index $\rho > 0$, then $\Psi$ such that 
	$\Psi(t) \coloneqq - \Phi_{\kappa}^\spadesuit(1/t)$
	belongs to $\RV$ with index $(1/\rho)-1$.
	\item\label{c2} If at least one among 
	$\Phi(\cdot,\,\kappa), \widehat{\Phi}(\cdot,\,\kappa)$ belongs to $\RVz$ with index $\rho > 0$ and  $\Phi(a,\,\kappa) = o(\widehat{\Phi}(a,\,\kappa))$  as $a \to 0_+$, then 
\begin{equation*}
    (\Phi_{\kappa}^{\spadesuit})^{-1}(s) = o\left((\widehat{\Phi}_{\kappa}^{\spadesuit})^{-1}(s)\right)\ \ {\rm as}\ \ s\rightarrow-\infty.
\end{equation*}
\end{enumerate}
\end{theorem}

\begin{proof}

First we prove items~\eqref{c1} and \eqref{c2} simultaneously by considering the case where $\Phi(\cdot,\,\kappa)\in \RVz$  with index $\rho > 0$. 
By assumption $\Phi(\cdot,\,\kappa)$ is monotone increasing and continuous, so \eqref{eq:rvz_inv} implies that 
$\Phi(\cdot,\,\kappa)^{-1} \in \RVz$ has index $1/\rho$.
From the definition of $\phi_{\kappa,\Phi}$ in \eqref{def_psi}, we have for any $\lambda > 0$,
\begin{equation}\label{rho_p}
\lim_{t\rightarrow 0_+}\frac{\phi_{\kappa,\Phi}(\lambda t)}{\phi_{\kappa,\Phi}(t)} = \lim_{t\rightarrow 0_+}\frac{[\Phi(\sqrt{\lambda t},\,\kappa)]^2}{[\Phi(\sqrt{t},\,\kappa)]^2} = \lim_{t\rightarrow 0_+}\Big(\frac{\Phi(\sqrt{\lambda}t,\,\kappa)}{\Phi(t,\,\kappa)}\Big)^2 = \lambda^{\rho}.
\end{equation}
Because $\Phi(\cdot,\kappa)$ is monotone increasing and continuous, the same is true of $\phi_{\kappa,\Phi}$ and 
$\phi_{\kappa,\Phi}^{-}$ coincides with the usual inverse $\phi_{\kappa,\Phi}^{-1}$.
Therefore, we have from \eqref{eq:rvz_inv} and \eqref{rho_p} that $\phi_{\kappa,\Phi}^{-1}\in\RVz$ with index $1/\rho$, namely,
\begin{equation}\label{mu_p}
\lim_{t\rightarrow 0_+}\frac{\phi^{-1}_{\kappa,\Phi}(\lambda t)}{\phi^{-1}_{\kappa,\Phi}(t)} = \lambda^{1/\rho}.
\end{equation}
Moreover, we see from assumptions in \eqref{c2} that 
\begin{equation}\label{phi_ky}
\lim_{s\rightarrow 0_+}\frac{\phi_{\kappa,\Phi}(s)}{\phi_{\kappa,\widehat{\Phi}}(s)}
= \lim_{s\rightarrow 0_+}\frac{[\Phi(\sqrt{s},\,\kappa)]^2}{[\widehat{\Phi}(\sqrt{s},\,\kappa)]^2} = \lim_{s\rightarrow 0_+}\left(\frac{\Phi(s,\,\kappa)}{\widehat{\Phi}(s,\,\kappa)}\right)^2 = 0.
\end{equation}
Therefore, $\phi_{\kappa,\Phi}$ and $\phi_{\kappa,\hat \Phi}$ are monotone increasing continuous functions with
\begin{equation}\label{orig_phi}
\phi_{\kappa,\Phi},\phi_{\kappa,\Phi}^{-1}\in \RVz \ \ \ {\rm and}\ \ \  \phi_{\kappa,\Phi}(s) = o\big(\phi_{\kappa,\widehat{\Phi}}(s)\big)\ \ {\rm as}\ \ s\rightarrow 0_+.
\end{equation}
Next, we define
\begin{equation*}
    w(x): = \frac{1}{\phi_{\kappa,\Phi}(1/x)},\ \ \ \     \widehat{w}(x): = \frac{1}{\phi_{\kappa,\widehat{\Phi}}(1/x)}, \ \ \ x > 0.
\end{equation*}
With that, $w$ and $\widehat{w}$ are unbounded continuous monotone increasing functions. Analogous to the computations in \eqref{eq:rv_eq}, we have  $w \in \RV$ with index $\rho$. 
Furthermore, from \eqref{phi_ky}  we obtain
\begin{equation}\label{small_2}
   0 = \lim_{s\rightarrow 0_+}\frac{\phi_{\kappa,\Phi}(s)}{\phi_{\kappa,\widehat{\Phi}}(s)} = \lim_{x\rightarrow\infty}\frac{\frac{1}{\phi_{\kappa,\widehat{\Phi}}(1/x)}}{\frac{1}{\phi_{\kappa,\Phi}(1/x)}} = \lim_{x\rightarrow\infty}\frac{\widehat{w}(x)}{w(x)},
\end{equation}
i.e., $\widehat{w}(x) = o(w(x))$ as $x \to \infty$.
In view of \eqref{small_2}, we 
can invoke Lemma~\ref{key_lemma} (by restricting $w$ and $\widehat{w}$ to some interval $[a,\,\infty)$), which leads to
\begin{equation}\label{1_lem}
0 = \lim_{x\rightarrow\infty}\frac{w^{-1}(x)}{\widehat{w}^{-1}(x)} =  \lim_{x\rightarrow\infty}\frac{\frac{1}{\phi_{\kappa,\Phi}^{-1}(1/x)}}{\frac{1}{\phi_{\kappa,\widehat{\Phi}}^{-1}(1/x)}} = \lim_{t\rightarrow 0_+}\frac{\phi_{\kappa,\widehat{\Phi}}^{-1}(t)}{\phi_{\kappa,\Phi}^{-1}(t)}.
\end{equation}
From Proposition~\ref{prop_phi} we have that $\Phi_{\kappa}^\spadesuit $ and $\widehat{\Phi}_{\kappa}^\spadesuit$ are monotone increasing continuously differentiable functions.
Using L'Hospital's rule in combination with assumption ($\ref{a3}$), we have from \eqref{1_lem} that
\begin{equation}\label{big_phi}
    \lim_{t\rightarrow 0_+}\frac{\Phi_{\kappa}^\spadesuit(t)}{\widehat{\Phi}_{\kappa}^\spadesuit(t)} = \lim_{t\rightarrow 0_+}\frac{\left(\Phi_{\kappa}^\spadesuit\right)'(t)}{\left(\widehat{\Phi}_{\kappa}^\spadesuit\right)'(t)} = \lim_{t\rightarrow 0_+}\frac{\frac{1}{\phi_{\kappa,\Phi}^{-1}(t)}}{\frac{1}{\phi_{\kappa,\widehat{\Phi}}^{-1}(t)}} = \lim_{t\rightarrow 0_+}\frac{\phi_{\kappa,\widehat{\Phi}}^{-1}(t)}{\phi_{\kappa,\Phi}^{-1}(t)} = 0.
\end{equation}
Now, we define
\begin{equation*}
    \Psi(t) \coloneqq -\Phi_{\kappa}^{\spadesuit}(1/t),\qquad
     \widehat{\Psi}(t) \coloneqq -\widehat{\Phi}_{\kappa}^{\spadesuit}(1/t), \ \ \ t > 0.
\end{equation*}
Since $\Phi_{\kappa}^\spadesuit(t),\,\widehat{\Phi}_{\kappa}^\spadesuit(t)$ both go to $-\infty$ as $t\rightarrow 0_+$ and  are monotone increasing (Proposition~\ref{prop_phi}), we have  that $ \Psi$ and  $\widehat{\Psi}$ are  monotone increasing and go to $+\infty$ as $t \to \infty$. Moreover, we have
\begin{equation}\label{big_1}
\begin{split}
     \lim_{x\rightarrow\infty}\frac{\Psi(\lambda x)}{\Psi(x)} = \lim_{t\rightarrow 0_+}\frac{\Phi_{\kappa}^{\spadesuit}(t)}{\Phi_{\kappa}^{\spadesuit}(\lambda t)} \overset{(a)}{=} \lim_{t\rightarrow 0_+}\frac{\left(\Phi_{\kappa}^\spadesuit\right)'(t)}{\lambda\left(\Phi_{\kappa}^\spadesuit\right)'(\lambda t)} = \lim_{t\rightarrow 0_+}\frac{\phi_{\kappa,\Phi}^{-1}(\lambda t)}{\lambda\phi_{\kappa,\Phi}^{-1}(t)} \overset{(b)}{=} \lambda^{(1/\rho) - 1},
\end{split}
\end{equation}
where (a) follows from L'Hospital's rule and (b) follows from \eqref{mu_p}.
That is, $\Psi \in \RV$ with index $(1/\rho) -1$, which proves that item~\eqref{c1} holds.
 On the other hand, we see from \eqref{big_phi} that
\begin{equation}\label{big_2}
    0 = \lim_{t\rightarrow 0_+}\frac{\Phi_{\kappa}^{\spadesuit}(t)}{\widehat{\Phi}_{\kappa}^{\spadesuit}(t)} = \lim_{x\rightarrow\infty}\frac{-\Phi_{\kappa}^{\spadesuit}(1/x)}{-\widehat{\Phi}_{\kappa}^{\spadesuit}(1/x)} = \lim_{x\rightarrow\infty}\frac{\Psi(x)}{\widehat{\Psi}(x)}.
\end{equation}
Combining  \eqref{big_1} and \eqref{big_2}, we may use Lemma~\ref{key_lemma} again (by restricting $\Psi$ and $\widehat{\Psi}$ to some interval $[a,\,\infty)$) to obtain
\begin{equation}\label{eq:final}
  0 = \lim_{x\rightarrow\infty}\frac{\widehat{\Psi}^{-1}(x)}{\Psi^{-1}(x)}   = \lim_{x\rightarrow \infty}\frac{\frac{1}{\left(\widehat{\Phi}_{\kappa}^{\spadesuit}\right)^{-1}(-x)}}{\frac{1}{\left(\Phi_{\kappa}^{\spadesuit}\right)^{-1}(-x)}} =   \lim_{s\rightarrow -\infty}\frac{(\Phi_{\kappa}^{\spadesuit})^{-1}(s)}{(\widehat{\Phi}_{\kappa}^{\spadesuit})^{-1}(s)}.
\end{equation}
This completes the proof of item~\eqref{c2} when $\Phi(\cdot,\,\kappa)\in \RVz$ has index $\rho > 0$. 

If $\widehat\Phi(\cdot,\,\kappa)\in \RVz$ has index $\rho > 0$, the proof is of item~\eqref{c2} is analogous since Lemma~\ref{key_lemma} only requires a regular variation assumption for one of the functions.
The difference is that at \eqref{rho_p}, \eqref{mu_p}, \eqref{orig_phi}, \eqref{big_1} we would 
draw conclusions about 
functions derived from $\widehat\Phi$ but all the other equations would remain the same. For example, in \eqref{big_1} we would conclude that $ \lim_{x\rightarrow\infty}\frac{\widehat \Psi(\lambda x)}{\widehat \Psi(x)}  = \lambda^{(1/\rho)-1}$, which would lead to the exact same  \eqref{eq:final}.
\qed\end{proof}

\begin{remark}[On  assumption ($\ref{a3}$) of Theorem~\ref{thm_comp}]\label{rem:as}
Because of Proposition~\ref{prop:asym}, in many cases it is not necessary to check assumption ($\ref{a3}$) of Theorem~\ref{thm_comp} explicitly.
%
\end{remark}

Following Theorem~\ref{thm_comp}, we will prove bounds for the $(\Phi_{\kappa}^{\spadesuit})^{-1}$ function.
This will require the so-called \emph{Potter bounds}.

\begin{lemma}[Potter bounds]
\label{lem:pbs}
If $f \in \RV$ with index $\rho$, then for every $A > 1, \delta > 0$, there exists $M > 0$ such that 
$x \geq M, y \geq M$ implies
\begin{equation}\label{eq:potter_rv}
\frac{f(x)}{f(y)} \leq A\max\left \{\left(\frac{x}{y}\right)^{\rho -\delta}, \left(\frac{x}{y}\right)^{\rho +\delta}  \right\}.
\end{equation}
If $f \in \RVz$ with index $\rho$, 
then for any $A > 1, \delta > 0$, there exists $M > 0$ such that $t \leq M, s\leq M$ implies
\begin{equation}\label{eq:potter}
\frac{f(t)}{f(s)} \leq A\max\left \{\left(\frac{t}{s}\right)^{\rho -\delta}, \left(\frac{t}{s}\right)^{\rho +\delta}  \right\}.
\end{equation}
\end{lemma}
The first half of Lemma~\ref{lem:pbs} is proved in \cite[Theorem~1.5.6]{BGT87}, while the latter half follows from applying the first half to $\hat f$ such that $\hat f(x) = 1/f(1/x)$.

Finally, we also need a similar bound for rapidly varying functions. The following lemma is a consequence of \cite[Lemma~2.2]{BGO06}.

\begin{lemma}
If $f \in \RV_{-\infty}$, then for every $r > 0$ there exists a constant $M$ such that $t \geq M$ implies
\begin{equation}\label{eq:potter_rapid}
f(t) \leq t^{-r}.
\end{equation} 
In particular,  for every $r > 0$ we have
\begin{equation}\label{eq:potter_rapid2}
f(t) = o (t^{-r}) \quad \text{ as } \quad t \to +\infty.
\end{equation} 
\end{lemma}

\begin{theorem}[Bounds on $(\Phi_{\kappa}^{\spadesuit})^{-1}$]\label{thm:upper}
Let $\Phi$ be a strict consistent error bound function associated to $C_1,\ldots, C_m$ and let 
$C \coloneqq \cap _{i=1}^m C_i$.  Suppose that $C$ is not the whole space and suppose that $\kappa \geq \max\{\dist(0,C), \norm{x^0} \}$ holds for some 
 $x^0 \not \in C$.

Suppose also that $\Phi(\cdot,\kappa)$ is continuous and belongs to $\RVz$ with index $\rho$.
Let $\Psi$ be given by $\Psi(t)\coloneqq-\Phi_{\kappa}^{\spadesuit}(1/t)$.
 Then, the following items hold.
\begin{enumerate}[$(i)$]
	\item $\rho \in [0,1]$.
	\item If $\rho \in (0,1)$, then 
	$\Psi$  belongs to 
	$\RV$ with index $(1/\rho) - 1$. In particular, 
	$\Psi^{-1} \in \RV$, has index $\frac{\rho}{1-\rho}$ and for every $\delta > 0$ such that $\gamma\coloneqq\rho/(1-\rho) - \delta$ is 
	positive, there are constants $M$ and $A$ such that
	\[
	\sqrt{(\Phi_{\kappa}^{\spadesuit})^{-1}(-s)} \leq A \left(\frac{1}{s}\right)^{\gamma/2},\qquad \forall s \geq M.
	\]
	\item If $\rho = 1$,  then the function 
	$\Psi$  belongs to 
	$\RV$ with index $0$.
	In particular, $\Psi^{-1} $ belongs to $\RV_{\infty}$ and for every $r > 0$, we have
	\[
	\sqrt{(\Phi_{\kappa}^{\spadesuit})^{-1}(-s)} = o (s^{-r}) \quad \text{ as } \quad s \to +\infty.
	\]
	
	\item If $\rho = 0$, then  
	$\Psi$ belongs to 
	$\RV_{\infty}$. In particular, $\Psi^{-1}$ belongs to ${\rm RV}$ with index $0$ and 
	for any $r > 0$ we have
	$	s^{-r} = 
	o\left({(\Phi_{\kappa}^{\spadesuit})^{-1}(-s)}\right)$ as $s\to\infty$.
\end{enumerate}
\end{theorem}
\begin{proof}
First, we prove item $(i)$.
For $\lambda > 1$, because $\Phi(\cdot, \kappa)$ is monotone, we have $\Phi(\lambda t,\kappa) \geq \Phi(t,\kappa)$. Therefore, $\lambda^\rho = \lim _{t\to 0_+} \Phi(\lambda t,\kappa)/\Phi(t,\kappa)\geq 1$, which shows that $\rho \geq 0$.

Next, let $d(x) \coloneqq \max _{1\leq i \leq m} \dist(x,C_i)$.
Since $C \subseteq C_i$ for all $i$, we have
\begin{equation}\label{eq:phi_aux}
d(x) \leq \dist(x, C) \leq \Phi(d(x),\kappa),
\end{equation}
whenever $\norm{x} \leq \kappa$. By assumption, 
$d(x^0) > 0$ and the projection $P_C(0)$ of $0$ onto 
to $C$ satisfies $\norm{P_C(0)} \leq \kappa$.
By continuity, $d(\cdot)$, 
 assumes every value between $0$ and 
$d(x^0)$ over the ball $\{x \mid \norm{x}\leq \kappa \}$. In view of \eqref{eq:phi_aux}, 
we conclude that for sufficiently small $t$ we have
\begin{equation}\label{eq:phi}
t \leq \Phi(t,\kappa).
\end{equation}
For the sake of obtaining a contradiction, suppose that $\rho > 1$ and let $\delta > 0$ be such that $\rho -\delta > 1$.   By using Potter bound \eqref{eq:potter} for $A = 2$, we conclude that  for sufficiently small $t,s$, we have 
\[
\Phi(t,\kappa) \leq 2\Phi(s,\kappa)\max\left \{\left(\frac{t}{s}\right)^{\rho -\delta}, \left(\frac{t}{s}\right)^{\rho +\delta}  \right\}.
\]
Combining with \eqref{eq:phi}, we obtain
\[
1 \leq \frac{\Phi(t,\kappa)}{t} \leq 2\Phi(s,\kappa)\max \{{t}^{\rho -\delta-1} {(1/s)}^{\rho -\delta},\quad {t}^{\rho + \delta-1} {(1/s)}^{\rho +\delta} \}.
\]
If we fix $s$ and let $t$ go to $0$, the right-hand side converges to $0$ (because $\rho-\delta -1 > 0$), which leads to a contradiction. So, indeed it must be the case that $\rho \in [0,1]$.

Next, we move on to item $(ii)$. 
From item~\eqref{c1} of Theorem~\ref{thm_comp},  $\Psi$ belongs to $\RV$ with index $(1/\rho) - 1$. By \eqref{eq:rv_inv}, 
$\Psi^{-1}$ has index $\rho/(1-\rho)$.
We also have $\Psi^{-1}(s) = 1/(\Phi_{\kappa}^{\spadesuit})^{-1}(-s)$.
Then, we apply Potter bound \eqref{eq:potter_rv} to $\Psi^{-1}$ with $x$, $y$ replaced by $t$ and $s$, respectively. Fixing $t$, taking square roots and recalling that $(t/s)^{b} \leq (t/s)^{a}$ if $0\leq a \leq b$ and $s \geq t$, leads to the final conclusion of item $(ii)$.

Now, we check item $(iii)$.
Again,  from item~\eqref{c1} of Theorem~\ref{thm_comp},  $\Psi$ belongs to $\RV$ with index $0$.
By Proposition~\ref{prop:asym}, $\Psi(t) \to +\infty$ as 
$t \to \infty$. Under these conditions, it is known that 
$\Psi^{-1}$ belongs to $\RV_{\infty}$, see \eqref{eq:rv_inv}.
Therefore, $1/\Psi^{-1}(s) = (\Phi_{\kappa}^{\spadesuit})^{-1}(-s)$ belongs to 
$\RV_{-\infty}$. Applying \eqref{eq:potter_rapid2} to $(\Phi_{\kappa}^{\spadesuit})^{-1}(-s)$ and taking square roots leads to the final conclusion of item~$(iii)$.

Finally, we prove item $(iv)$. First, we see from $\rho = 0$ that for any $\lambda >0$,
\begin{equation*}
\lim_{x\to 0+}\frac{\phi_{\kappa,\Phi}(\lambda x)}{\phi_{\kappa,\Phi}(x)} = \lim_{x\to 0_+}\Big(\frac{\Phi(\sqrt{\lambda x},\,\kappa)}{\Phi(\sqrt{x},\,\kappa)}\Big)^2 = \lim_{x\to 0_+}\Big(\frac{\Phi(\sqrt{\lambda}x,\,\kappa)}{\Phi(x,\,\kappa)} \Big)^2= 1.
\end{equation*}
Let $w(x):=\frac{1}{\phi_{\kappa,\Phi}(1/x)}$. We then have
\begin{equation*}
\lim_{x\to\infty}\frac{w(\lambda x)}{w(x)} = \lim_{x\to\infty}\frac{1/\phi_{\kappa,\Phi}(1/(\lambda x))}{1/\phi_{\kappa,\Phi}(1/x)} = \lim_{s\to 0_+}\frac{\phi_{\kappa,\Phi}(\lambda s)}{\phi_{\kappa,\Phi}(s)} = 1,
\end{equation*}
which implies that $w\in {\rm RV}$ with index $0$. Since $w(x) \to + \infty$ as $x \to +\infty$, 
again by \eqref{eq:rv_inv}, we see that $w^{-1}(x) = \frac{1}{\phi^{-1}_{\kappa,\Phi}(1/x)} \in {\rm RV}_{\infty}$. Note that $\Psi(t):= -\Phi_{\kappa}^{\spadesuit}(1/t)$. We use L'Hospital's rule and further have
\begin{equation*}
\begin{split}
\lim_{t\to\infty}\frac{\Psi(\lambda t)}{\Psi(t)} & = \lim_{t\to\infty}\frac{\Phi_{\kappa}^{\spadesuit}(1/(\lambda t))}{\Phi_{\kappa}^{\spadesuit}(1/t)} =\lim_{s\to 0_+}\frac{\Phi_{\kappa}^{\spadesuit}(s/\lambda)}{\Phi_{\kappa}^{\spadesuit}(s)} = \lim_{s\to 0_+}\frac{\frac1{\lambda}/\phi_{\kappa,\Phi}^{-1}(s/\lambda)}{1/\phi_{\kappa,\Phi}^{-1}(s)}\\ 
& = \lim_{s\to 0_+}\frac{\phi^{-1}_{\kappa,\Phi}(\lambda s)}{\lambda\phi^{-1}_{\kappa,\Phi}(s)} = \lim_{x\to\infty}\frac{1/w^{-1}(x/\lambda)}{\lambda/w^{-1}(x)} = \lim_{x\to\infty}\frac{w^{-1}(\lambda x)}{\lambda w^{-1}(x)},
\end{split}
\end{equation*}
which implies that $\Psi\in{\rm RV}_{\infty}$ and thus $1/\Psi(t) = -1/\Phi_{\kappa}^{\spadesuit}(1/t)\in{\rm RV}_{-\infty}$. Now, from \eqref{eq:potter_rapid}, it follows that  $1/\Psi(t)$ goes to $0$ as $t \to \infty$. 
From \eqref{eq:potter_rapid2},
$1/\Psi(t) = o(t^{-r})$ as $t \to +\infty$. Therefore, $t^r = o(\Psi(t))$ as $t\to\infty$. 
Finally, since $g(t) := t^{r} \in \RV$, from  Lemma~\ref{key_lemma} we obtain 
\begin{equation*}
\frac{s^{-1/r}}{(\Phi_{\kappa}^{\spadesuit})^{-1}(-s)} = \frac{1/(\Phi_{\kappa}^{\spadesuit})^{-1}(-s)}{s^{1/r}} \to 0\ \ \ {\rm as} \ \ \  s\to\infty.
\end{equation*}
This completes the proof.
\qed\end{proof}
Theorem~\ref{thm:upper} has the following informal consequence: any consistent error bound function that corresponds to an $\RVz$ function of index $\rho \in (0,1]$ behaves almost the same as a H\"olderian error bound with exponent $\rho$. In particular, in view of our convergence results (see, for example, \eqref{case_pg1}), items $(ii)$ and $(iii)$ imply that the corresponding convergence rate would be at least as fast as the convergence rate afforded by any H\"olderian error bound with exponent $\rho' < \rho$.

\subsection{Logarithmic error bounds}\label{sec:log}
In Theorem~\ref{thm:upper}, if $\rho  = 0$, only a 
\emph{lower bound} to $(\Phi_{\kappa}^{\spadesuit})^{-1}$ is obtained. 
 Because $(\Phi_{\kappa}^{\spadesuit})^{-1}$ can be used to \emph{upper bound} the convergence rate (see Theorem~\ref{thm_conv}), a lower bound to $(\Phi_{\kappa}^{\spadesuit})^{-1}$ can not be used in general to draw conclusions about the convergence rates of the algorithms discussed in Section~\ref{sec:proj}.
In view of this limitation, it would be useful to get reasonable upper bounds to $(\Phi_{\kappa}^{\spadesuit})^{-1}$ as well when $\rho = 0$.

A challenge in this task is that the class of $\RV$ functions  with index $\rho = 0$  contains functions with 
very slow growth. Indeed, these are called 
\emph{slowly varying functions} in the regular variation literature.
For example, $(\ln(x))^\alpha$ (for any nonzero $\alpha$) and arbitrary compositions of logarithms $\ln(\ln(\cdots \ln(x)))$ belong to $\RV$ with index $0$ (see \cite[Section~1.3.3]{BGT87}).
Because of that, asymptotic upper bounds that are valid for any slowly varying function are doomed to not be very informative.

In order to get meaningful bounds in the case $\rho = 0$ we need to further restrict the class of functions under consideration as follows. 
\begin{definition}[Logarithmic error bound]\label{def:log}
An error bound function $\Phi$ is said to be \emph{logarithmic with exponent $\gamma$} if for every $ b > 0$, there exist
$\kappa_b > 0$ and $a_b > 0$ such that $\Phi(a,b) = \kappa_b\left(-\frac{1}{\ln(a)}\right)^{\gamma} $ holds 
for $a \in (0,a_b)$.
\end{definition}


Next, we show an example of logarithmic error bound. Another instance will be discussed in Section~\ref{sec:exp} in the context of the analysis of 
the exponential cone.

\begin{example}[Example of logarithmic error bound in arbitrary dimension]\label{ex:log2}
We start with the analysis of some functions that will be helpful to build our example.
For every $\gamma \geq 2$, we define $\tilde f_{\gamma} :\R \to \R_+$ such that ${\tilde f}_{\gamma}(0) = 0$ and
\[
{\tilde f}_{\gamma}(t) \coloneqq e^{- \frac{1}{|t|^\gamma} }, \qquad  \forall t \neq 0.
\]
The case $\gamma = 2$ corresponds to a function described in, e.g., \cite[page~453]{ABRS10}.	
We note that  ${\tilde f}_{\gamma}''$ is nonnegative in a neighbourhood of $0$. Then, because a convex function is locally Lipschitz on the relative interior of its domain, we  can select $t_{\gamma} > 0$ such that ${\tilde f}_{\gamma}$ restricted to $[-t_{\gamma},t_\gamma]$ is convex and Lipschitz continuous with constant $L_{\gamma}$. 
 Finally, let $f_{\gamma}$ be the infimal convolution between  ${\tilde f}_{\gamma}$ restricted to $[-t_{\gamma},t_\gamma]$ and $L_{\gamma}|\cdot|$:
\begin{equation}\label{eq:inf_conv}
f_{\gamma}(t) \coloneqq \inf_{u \in [-t_{\gamma},t_\gamma] } {\tilde f}_{\gamma}(u) + L_{\gamma}|t-u|.
\end{equation}
With that $f_{\gamma}$ is  a convex function which is finite over $\R$ and satisfies 
$f_{\gamma}(t) = {\tilde f}_{\gamma}(t)$ for $t \in  [-t_{\gamma},t_\gamma]$. Since $f_{\gamma}$ has an unique minimum at $t = 0$ and is convex, $f_{\gamma}$ is monotone increasing when 
restricted to $[0,\infty)$. Taking $u = 0$ in \eqref{eq:inf_conv} we obtain
\begin{equation}\label{eq:f_gamma1}
f_{\gamma}(t) \leq L_{\gamma} |t|, \qquad \forall t \in \R.
\end{equation}

Let $\varphi_\gamma$ be the inverse of the restriction of $f_{\gamma}$ to 
$[0,\infty)$. 
Since $f_{\gamma}(t) \rightarrow \infty$ as 
$t \to \infty$,  $\varphi_\gamma$ is well-defined over $[0,\infty)$. Because $f_{\gamma}(t) = f_{\gamma}(-t)$, we also have
\begin{equation}\label{eq:log_inv}
\varphi_\gamma(f_\gamma(t)) = |t|, \qquad \forall t \in \R.
\end{equation}
Furthermore, $\varphi_\gamma$ is monotone increasing and for $t \in (0,f_{\gamma}(t_\gamma)]$, 
$\varphi_\gamma$ coincides with the inverse of $\tilde f_{\gamma}$, so we have
\begin{equation}\label{eq:log_inv2}
\varphi_\gamma(t) = \left(-\frac{1}{\ln(t)}\right)^{1/\gamma}.
\end{equation}
Also, \eqref{eq:f_gamma1} implies that $f_{\gamma}(t/L_{\gamma}) \leq t$ for $t \geq 0$, therefore,
\begin{equation}\label{eq:log_inv3}
t \leq L_{\gamma}\varphi_{\gamma}(t), \qquad \forall t \geq 0.
\end{equation}
Because $f_{\gamma}$ is convex and $\varphi_{\gamma}$ is monotone increasing, $\varphi_{\gamma}$ must be concave. 
Combined with the fact that $\varphi_{\gamma}(0) = 0$, we have that 
\begin{equation}\label{eq:log_inv4}
\varphi_{\gamma}((1+\lambda)t) \leq (1+\lambda) \varphi_{\gamma}(t),\qquad \forall \lambda,t \geq 0.  
\end{equation}

Next, we define
\[
C_{1} \coloneqq \{(x,\mu) \in \R^n\times \R \mid \mu \geq f_\gamma(\norm{x}) \}, \qquad C_2 \coloneqq \{(x,0) \in \R^n\times \R \}.
\]
We have $C \coloneq C_1 \cap C_2 = \{(0,0)\}$ and we shall check several things about this example.
For the sake of obtaining a contradiction, suppose that  a H\"olderian error bound holds in a neighbourhood of $(0,0)$. Then, by considering points of the form $(x_t,0)\coloneqq (t,0,\ldots, 0)$ with $t \in \R_{++}$, there exist $k > 0$ and an exponent $\alpha \in (0,1]$ such that
\begin{equation*}\label{eq:hold}
t = \dist((x_t,0), C ) \leq k \dist((x_t,0),C_1)^\alpha
\leq k \norm{(t,0,\cdots,0) - (t,0,\cdots, f_{\gamma}(t)) }^\alpha = k f_{\gamma}(t)^\alpha
\end{equation*}
holds for all sufficiently small $t$. However, this is 
impossible because $t/ f_{\gamma}(t)^\alpha$ goes to $\infty$ as $t \to 0_+$.
The conclusion is that no H\"olderian error bound holds.

Next, we check that $C_1$ and $C_2$  admit a logarithmic error bound with exponent $1/\gamma$.
We recall the following properties of orthogonal projections: if $U,V \subseteq \R^n$ are closed convex sets and $z \in \R^n$, then 
\begin{align}
\dist(z, U) &\leq \dist(z, V) + \dist(P_V(z), U) \label{eq:proj1},\\
\dist(P_V(z),U) &\leq \dist(z,V) + \dist(z, U) \label{eq:proj2}.
\end{align}
Let $b > 0$ and let $(x,\mu)$ be such that $\norm{(x,\mu)} \leq b$. From \eqref{eq:proj1} we have:
\begin{equation}\label{eq:ex2_log3}
\norm{(x,\mu)}=\dist((x,\mu), C_1 \cap C_2) \leq \dist((x,\mu),C_2) + \dist((x,0), C_1 \cap C_2).
\end{equation}
Let $(\bar{x},f_{\gamma}({\norm{\bar{x}}}))$ be the orthogonal projection of $(x,0)$ to $C_1$. 
Since $f_{\gamma}$ is convex and finite everywhere, its restriction to any bounded interval of $\R$ is Lipschitz continuous, e.g., \cite[Theorem~10.4]{rockafellar}. Let $L$ be the Lipschitz constant of $f_{\gamma}$ restricted to the inverval $[-b,b]$.
As projections are nonexpansive and $(0,0) \in C_1$, we have $\norm{(\bar{x},f_{\gamma}(\bar{x}))}\leq \norm{x}$ which implies that $\norm{\bar{x}} \leq \norm{x}\leq b$.
Then
\begin{equation}\label{eq:ex2_log}
f_\gamma(\norm{x}) - f_\gamma(\norm{\bar{x}}) \leq |f_\gamma(\norm{x}) - f_\gamma(\norm{\bar{x}})| \leq L|\norm{x}-\norm{\bar{x}}| \leq L\norm{x-\bar{x}}.
\end{equation}
Letting $\hat L \coloneqq \max\{L,1\}$, from \eqref{eq:ex2_log} we obtain
\begin{equation}\label{eq:ex2_log5}
f_\gamma(\norm{x}) \leq \hat L \left(\norm{x-\bar{x}}+ f_\gamma(\norm{\bar{x}})\right) \leq 
\hat L\sqrt{2}\sqrt{f_\gamma(\norm{\bar{x}})^2 + \norm{x-\bar{x}}^2}.
\end{equation}
Since $\dist((x,0),C_1) = \sqrt{f_\gamma(\norm{\bar{x}})^2 + \norm{x-\bar{x}}^2}$,
from \eqref{eq:ex2_log5} we see that there exists a constant $\tilde L > 0$ such that 
\begin{equation}\label{eq:ex2_log2}
f_\gamma(\norm{x}) \leq \tilde{L}\dist((x,0),C_1).
\end{equation}
Because $\varphi_\gamma$ is monotone increasing, we 
can apply $\varphi_\gamma$ at both sides of \eqref{eq:ex2_log2} and, recalling \eqref{eq:log_inv}, we obtain $\norm{x} \leq \varphi_\gamma(\tilde{L}\dist((x,0),C_1))$.
Since $\norm{x} = \dist((x,0),C_1\cap C_2)$, from 
\eqref{eq:ex2_log3} we obtain
\begin{equation}\label{eq:ex2_log4}
\dist((x,\mu), C_1 \cap C_2) \leq \dist((x,\mu),C_2) + \varphi_\gamma(\tilde{L}\dist((x,0),C_1)).
\end{equation}
Now, let $d(x,\mu)$ be the maximum between 
$\dist((x,\mu),C_2)$ and $\dist((x,\mu),C_1)$.
From \eqref{eq:proj2}, we
obtain $\dist((x,0),C_1) \leq \dist((x,\mu),C_1) + \dist((x,\mu),C_2)$. 
We can use this together with \eqref{eq:log_inv3} and \eqref{eq:log_inv4} to obtain an upper bound to the right-hand-side of \eqref{eq:ex2_log4} thus concluding that there exists $\rho(b) > 0$ such that
\begin{equation}\label{eq:log2_ex}
\dist((x,\mu), C_1 \cap C_2) \leq \rho(b)\varphi_\gamma(d(x,\mu))
\end{equation}
holds for all $(x,\mu)$ with $\norm{(x,\mu)} \leq b$. 
Since increasing $\rho(b)$ still leads to a valid upper bound in \eqref{eq:log2_ex} we may select $\rho(b)$ in such a way that $\rho(\cdot)$ is a monotone nondecreasing function of $b$.  So, $\Phi$ given by 
$\Phi(a,b) \coloneqq \rho(b)\varphi_\gamma(a) $ is a strict consistent error bound function. It is also logarithmic with exponent $1/\gamma$ because of 
\eqref{eq:log_inv2}.
\end{example}

If $\Phi$ is as in Definition~\ref{def:log}, then $\Phi(\cdot,b)$ is an $\RVz$ function of index $0$ for every $b > 0$. Then, the function $\Psi$ in  Theorem~\ref{thm:upper} is  rapidly varying and  $(\Phi_{\kappa}^{\spadesuit})^{-1}$ is again an $\RVz$ function of index $0$. The fact that the index is $0$  precludes the usage of Potter bounds to obtain an asymptotic  upper bound to $(\Phi_{\kappa}^{\spadesuit})^{-1}$.
In addition, neither $\Psi$ nor $(\Phi_{\kappa}^{\spadesuit})^{-1}$ seem to have simple closed form expressions, so evaluating them directly is non-trivial.
However,  we can show that applying a logarithm is enough to ``de-accelerate'' $\Psi$ down to a regular varying function with positive index $\rho$. Better still, we will argue that $\ln \Psi$ is \emph{asymptotically equivalent} to a function for which we can directly compute the inverse. Here, we say that $f$ and $g$ are \emph{asymptotically equivalent at $\infty$} if 
\[
\lim _{t\to \infty} \frac{f(t)}{g(t)} = 1.
\]
In this case, we write \emph{$f(t)\sim g(t)$, as $t \to \infty$}. The following lemma is the first step towards implementing the strategy just outlined.

\begin{lemma}\label{lem_approx}
Let $f: [a,\,\infty)\to(0,\,\infty)\in\RV$ ($a > 0$) with index $\rho > 0$. Then we have 
\begin{equation*}
	g(t):=\ln\int_{a}^te^{f(x)}dx \quad \sim \quad  f(t),\ \ {\rm as}\ \ t \to\infty.
	\end{equation*}
\end{lemma}

\begin{proof}
This result	is a direct consequence of one of the many  \emph{Abelian theorems} discussed in \cite[Chapter~4]{BGT87}. In this context, an Abelian theorem is a result that relates the asympotic properties of a function $f$ to some transform of $f$.

First, we extend the domain of $f$ to $[0,\,\infty)$ by setting $f(x) = f(a)$ for all $x\in[0,\,a)$. Invoking \cite[Theorem~4.12.10~(ii)]{BGT87}, we then have 
\begin{equation}\label{h_equi}
	h(t):=\ln\int_{0}^te^{f(x)}d\,x\sim f(t) \ \ {\rm as}\ \ t\to\infty.
\end{equation}
The proof is now essentially complete because changing the starting point of the integral in \eqref{h_equi} does not influence the 
asymptotic equivalence.  Nevertheless, we will provide a formal justification for this.

To simplify the notation, we let $F(t):=\int_{0}^te^{f(x)}d\,x$ and $b:=  \int_{0}^ae^{f(x)}d\,x$. Therefore, we can rewrite $g$ as
\begin{equation}\label{g_expre}
g(t) = \ln\int_{a}^te^{f(x)}d\,x = \ln \left(\int_{0}^te^{f(x)}d\,x - \int_{0}^ae^{f(x)}d\,x\right) = \ln \left(F(t) - b\right).
\end{equation}
Because $f$ has positive index of regular variation, $f(t) \to \infty$ as $t \to \infty$, which is a consequence of Potter bounds by selecting $\delta = \rho/2$, fixing $x$ and letting $y$ go to infinity in \eqref{eq:potter_rv}, see also \cite[Proposition~1.5.1]{BGT87}.
This implies that $F(t)\to\infty$ as $t\to\infty$ as well.  Using this, \eqref{h_equi} and \eqref{g_expre}, we  obtain
\begin{equation*}
\begin{split}
\lim_{t\to\infty}\frac{g(t)}{f(t)}  & = \lim_{t\to\infty}\frac{g(t)}{h(t)}\frac{h(t)}{f(t)} = \lim_{t\to\infty}\frac{\ln \left(F(t) - b\right)}{\ln(F(t))}\frac{f(t) + o(f(t))}{f(t)}\\
& = 1 + \lim_{t\to\infty}\frac{\ln \left(F(t) - b\right) - \ln(F(t))}{\ln\left(F(t)\right)} 
= 1 + \lim_{t\to\infty}\frac{\ln \left(1 - b/F(t)\right) }{\ln\left(F(t)\right)} = 1.
\end{split}
\end{equation*}
This completes the proof.
\qed
\end{proof}
Next, we need a counterpart of Lemma~\ref{key_lemma} for 
asymptotic equivalence.
\begin{lemma}\label{lemma:inv}
Assume that $f,\,g: [a,\,\infty)\rightarrow (0,\,\infty)\, (a > 0)$ are continuous monotone increasing unbounded functions, and $f\in \RV$ or $g\in \RV$ with positive index. If $f(x) \sim g(x)$ as $x\rightarrow\infty$, then $f^{-1}(x) \sim g^{-1}(x)$ as $x\rightarrow\infty$.
\end{lemma}
\begin{proof}
Under the hypothesis that $f$ and $g$ are continuous and monotone increasing, we have $f^{\leftarrow} = f^{-1}$ and $g^{\leftarrow} = g^{-1}$.	
So the lemma follows from \cite[p190, Exercise~14, items $(ii)$ and $(iii)$]{BGT87}, see 
also \cite[Theorem~A]{dt07} and the surrounding discussion. \qed 
\end{proof}

We are ready to present our main result in this subsection. In the following theorem, we provide a tight estimate for 
the $(\Phi_{\kappa}^{\spadesuit})^{-1}$ function in the case of a logarithmic error bound. In view of Theorem~\ref{thm_conv} this gives a worst-case convergence rate for several algorithms when the underlying error bound is logarithmic. 

\begin{theorem}[Tight bounds to $(\Phi_{\kappa}^{\spadesuit})^{-1}$]\label{thm:lograte} 
	Let $\kappa > 0$ and  error bound function $\Phi$ be logarithmic with exponent $\gamma > 0$ as  in Definition~\ref{def:log}.  Then, there
	exists a constant $\eta > 0$ such 
	that
	\begin{equation}\label{rate_log}
\sqrt{(\Phi_{\kappa}^{\spadesuit})^{-1}(-s)}  \quad \sim \quad \eta\left(\frac{1}{\ln(s)} \right)^\gamma,\ \ {\rm as}\ \ s\to\infty.
	\end{equation}
In particular, there are constants $\eta_1 > 0$, $\eta_2 > 0$ and $N > 0$  such that

	\begin{equation}\label{rate_log2}
	\eta_1\left(\frac{1}{\ln(s)}\right)^{\gamma}\leq \sqrt{(\Phi_{\kappa}^{\spadesuit})^{-1}(-s)} \le \eta_2\left(\frac{1}{\ln(s)}\right)^{\gamma}, \ \ \forall\ s\ge N.
	\end{equation}
\end{theorem}

\begin{proof}
By assumption, there exist $c > 0$ and  $0 <\epsilon < 1$ such that for $a\in(0,\,\epsilon]$, 
\begin{equation*}
\Phi(a,\,\kappa) = c\left(-\frac{1}{\ln(a)}\right)^{\gamma}.
\end{equation*}
By the definition of $\phi_{\kappa,\Phi}$, we have
	\begin{equation*}
	\phi_{\kappa,\Phi}(t) = \Phi^2(\sqrt{t},\,\kappa) = c^22^{2\gamma}\frac{1}{(\ln (t))^{2\gamma}}, \ \ \ t\in(0,\, \epsilon^2].
	\end{equation*}
	Let $c_1:= 2c^{1/\gamma}$ and $c_2:= c_1^{2\gamma}/(2\ln(\epsilon))^{2\gamma}$. We then obtain
\begin{equation*}
	\phi^{-1}_{\kappa,\Phi}(s) = e^{-\frac{c_1}{s^{1/(2\gamma)}}}, \ \ \ s\in(0,\,c_2].
\end{equation*}
Now, we fix $\delta = c_2$ in the definition of $\Phi_{\kappa}^{\spadesuit}$, see \eqref{def_phi}.  Let $\Psi(t):= -\Phi_{\kappa}^{\spadesuit}(1/t)$. Next, we consider the behavior of $\Psi$ on $[1/c_2,\,\infty)$. For $t \ge 1/c_2$, we compute
\begin{equation*}
	\Psi(t) = -\int_{\delta}^{1/t}\frac{1}{\phi^{-1}_{\kappa,\Phi}(s)}d\,s = \int_{1/\delta}^t\frac{e^{c_1x^{1/(2\gamma)}}}{x^2}d\,x = \int_{1/c_2}^te^{c_1x^{1/(2\gamma)} - 2\ln (x)}d\,x.
\end{equation*}
	Let $f(x):= c_1x^{1/(2\gamma)} - 2\ln (x)$. Then, a direct limit computation shows that $f|_{[1/c_2,\,\infty)}\in\RV$ with index $1/(2\gamma)$. By Lemma~\ref{lem_approx}, we have
	\begin{equation*}
	\ln\Psi(t) = \ln\int_{1/c_2}^te^{f(x)} \quad \sim \quad f(t) \quad \sim\quad c_1t^{1/(2\gamma)},
	\end{equation*}
as $t \to \infty$.
Let $g(t) \coloneqq c_1t^{1/(2\gamma)}$.
Since $g$ belongs to $\RV$ with positive index $1/(2\gamma)$ and both $\ln \Psi$ and $g$ are continuous monotone increasing unbounded functions we can invoke Lemma~\ref{lemma:inv} which tells us that
\[
\Psi^{-1}(e^t) = (\ln \Psi)^{-1}(t) \quad \sim\quad g^{-1}(t) \ \ {\rm as}\ \ t\to\infty.
\]
We note that if $f_1(t) \sim f_2(t)$ as $t \to \infty$ holds then
$1/f_1(t) \sim 1/f_2(t)$ as $t \to \infty$ holds as well. With that in mind, we let $s = e^t$ and recalling 
that $\Psi(s)= -\Phi_{\kappa}^{\spadesuit}(1/s)$, we 
obtain
\begin{equation}\label{eq:final_log}
\sqrt{(\Phi_{\kappa}^{\spadesuit})^{-1}(-s)} = \frac{1}{\sqrt{\Psi^{-1}(s)}} \quad \sim \quad \frac{1}{\sqrt{g^{-1}(\ln s)}} = c_1^{
\gamma}\left(\frac{1}{\ln(s)} \right)^\gamma,
\end{equation}
as $s \to \infty$, which proves \eqref{rate_log}.
Finally, \eqref{rate_log2} is a consequence of \eqref{eq:final_log} and the definition of asymptotic equivalence which implies that for sufficiently large $s$ we have
\[
\frac{\sqrt{(\Phi_{\kappa}^{\spadesuit})^{-1}(-s)}}{c_1^{\gamma}\ln(s)^{-\gamma}} \quad \in\quad [0.5,2].
\]
This completes the proof.\qed

\end{proof}


%

\section{Convergence rate results for conic feasibility problems}\label{sec:cone}
In this section, we analyze the following  problem.\begin{equation}
{\rm find}\ x \in \stdCone \cap \stdAffine, \tag{Cone} \label{eq:clp}
\end{equation}
where $\stdCone$ is a closed convex cone, 
$\stdAffine$ is an affine space satisfying $\stdCone \cap \stdAffine \neq \emptyset$. 
First, we present some motivation for \eqref{eq:clp}. A conic linear program (CLP) is the problem of minimizing/maximizing a linear function subject to a constraint of the form $x \in \stdCone \cap \stdAffine$.
In this context, the methods discussed in Sections~\ref{sec:proj} can be useful to find feasible solutions to a CLP or to refine slightly infeasible solutions. See, for example, \cite{HM11}.

As discussed in Section~\ref{sec:proj}, the convergence rate of the methods is governed by the type of error bound that exists between $\stdCone$ and $\stdAffine$. 
Here we take a closer look at the error bound proved in 
\cite{L19} for the case where $\stdCone$ is a so-called 
\emph{amenable cone}.
$\stdCone$ is said to be \emph{amenable} if for every face $\stdFace$ of $\stdCone$ there exists a constant $\kappa$ such that $\dist(x, \stdFace) \leq \kappa \dist(x, \stdCone)$ holds for every $x \in \spanVec \stdFace$. The error bound for amenable cones described in 
\cite{L19} requires the following notion.

\begin{definition}[Facial residual functions]\label{def:frf}
	Let $\stdFace$ be a face of $\stdCone$ and $z \in \stdFace^*$.
	We say that $\psi_{\stdFace,z} : \R _+ \times \R _+ \to \R_+$  is a \emph{facial residual function} for 
	$z$ and $\stdFace$ if the following properties are satisfied:
	\begin{enumerate}[$(i)$]
		\item $\psi_{\stdFace,z}$ is nonnegative, monotone nondecreasing in each argument and $\psi(0,\, \alpha) = 0$ for every $\alpha \in \R_+$.
		\item whenever $x \in \spanVec \stdCone$ satisfies  the inequalities
		\[
		\dist(x,\, \stdCone) \leq \epsilon, \quad \inProd{x}{z} \leq \epsilon, \quad \dist(x,\, \spanVec \stdFace ) \leq \epsilon
		\]
		we have:
		\[
		\dist(x, \,  \stdFace\cap\{z\}^\perp)  \leq 
		\psi_{\stdFace, z} (\epsilon,\, \norm{x}).
		\]	
	\end{enumerate}
\end{definition}
We say that 
a function $\tilde \psi_{\stdFace,z}$ is a \emph{positive rescaling of $\psi_{\stdFace,z}$} if there are positive constants $M_1,M_2,M_3$ such that 
$\tilde \psi _{\stdFace,z}(\epsilon,\, \norm{x}) = M_3\psi_{\stdFace,z} (M_1\epsilon,\, M_2\norm{x}).$ We will also need to compose facial residual 
functions in a special way. We define $\psi_2\comp \psi_1$ to be the function satisfying 
\begin{equation}
(\psi_2\comp \psi_1)(a,\, b) = \psi_2(a+\psi_1(a,\,b),\, b),
\qquad \forall \, a,b \in \R. \label{eq:comp}
\end{equation}
In order to give the precise statement of the error bound in \cite{L19}, the final component we need is \emph{facial reduction} \cite{BW81_2,WM13,P13}.
The basic facial reduction algorithm as described in \cite{WM13,P13} shows that it is always possible to obtain a chain of faces of $\stdCone$ 
\begin{equation}\label{eq:chain}
\stdFace _{\ell}  \subsetneq \cdots \subsetneq \stdFace _1 = \stdCone,
\end{equation}
where the following properties are satisfied.
\begin{enumerate}[$(i)$]
	\item For $1 \leq i < \ell$, there exists $z_i \in \stdFace _i^* \cap \stdAffine^\perp$ such that $\stdFace_{i+1} = \stdFace _i\cap \{z_i\}^\perp.$
	\item $\stdFace_{\ell}\cap \stdAffine$ satisfies some desirable constraint qualification.
\end{enumerate}
Here, $\ell$ is called the \emph{length} of the chain.
Classical facial reduction approaches usually find chain of faces such that 
$\stdFace _{\ell}\cap \stdAffine$ satisfies Slater's condition, i.e., $
(\reInt \stdFace _{\ell}) \cap \stdAffine \neq \emptyset.
$
However, the FRA-Poly algorithm \cite{LMT15} finds 
a face $\stdFace _{\ell}$ satisfying a weaker constraint qualification called \emph{partial polyhedral Slater's condition} (PPS condition), which we will now describe. Suppose that $\stdFace_{\ell}$ can be written as a direct product $P\times \tilde \stdFace_{\ell}$, where $P$ is a polyhedral cone and $\tilde \stdFace_{\ell}$ is an arbitrary cone. If 
\[
(P\times (\reInt \tilde \stdFace _{\ell})) \cap \stdAffine \neq \emptyset,
\]
then we say that the PPS condition holds,  see Definition~1 in \cite{LMT15}. $P$ is allowed to be trivial, so if Slater's condition is satisfied the PPS condition is also satisfied.
With that in mind, we define two key quantities.
\begin{itemize}
	\item The \emph{singularity degree} $\ds(\stdCone,\, \stdAffine)$ of the pair $\stdCone$, $\stdAffine$ is the length of the smallest chain of faces (as in \eqref{eq:chain}) where $\stdFace _{\ell}$ and $\stdAffine$ satisfy Slater's condition.
	\item The \emph{distance to the partial Polyhedral Slater's condition} $\dpp(\stdCone,\, \stdAffine)$ is the length \emph{minus one} of the smallest chain of faces (as in \eqref{eq:chain}) where $\stdFace _{\ell}$ and $\stdAffine$ satisfy the PPS condition. Since Slater's condition is a  stronger requirement than the PPS condition, we have $
	\dpp(\stdCone,\, \stdAffine) \leq \ds(\stdCone,\, \stdAffine).$
\end{itemize}
We are now positioned to state the error bound in \cite{L19}.
\begin{theorem}[Error bound for amenable cones, Theorem~23 in \cite{L19}]\label{theo:err}
	Let $\stdCone$ be a closed convex pointed \emph{amenable cone}, $\stdAffine$ be an affine space such that $\stdCone \cap \stdAffine\neq \emptyset$. Let 
	$
	\stdFace _{\ell}  \subsetneq \cdots \subsetneq \stdFace_1 = \stdCone 
	$
	be a chain of faces of $\stdCone$  as in \eqref{eq:chain} 
	together with $z_i \in \stdFace _i^*\cap \stdAffine^\perp$  as in item $(i)$.
	Furthermore,  assume that $\stdFace_{\ell}, \stdAffine$ satisfy the 
	PPS condition.
	For $i = 1,\ldots, \ell - 1$, let $\psi _{i}$ be a facial residual function for $\stdFace_{i}$, $z_i$. 
	Then, after positive rescaling the $\psi _{i}$, there is a  positive constant $\kappa$ such that if  $x \in \spanVec \stdCone$ satisfies the inequalities
	\[
	\quad \dist(x,\, \stdCone) \leq \epsilon, \quad \dist(x,\, \stdAffine) \leq \epsilon,
	\]
	we have 
	\[
	\dist\left(x,\, \stdCone \cap \stdAffine \right) \leq (\kappa \norm{x} + \kappa )(\epsilon+\varphi(\epsilon,\, \norm{x})),
	\]
	where $
	\varphi = \psi _{{\ell-1}}\comp \cdots \comp \psi_{{1}}$, if $\ell \geq 2$. If $\ell = 1$, we let $\varphi$ be the function satisfying $\varphi(\epsilon, \, \norm{x}) = \epsilon$. 
\end{theorem}

Next, we will show that, under a mild condition, the error bound for amenable cones in Theorem~\ref{theo:err} naturally leads to a strict consistent error bound function. 

\begin{proposition}\label{prop:err_am_geb}
	Suppose that $\stdCone$ is a full-dimensional amenable cone, $\stdAffine$ is an affine space such that $\stdCone \cap \stdAffine \neq \emptyset$. Let $\varphi$ be defined as in Theorem~\ref{theo:err}. If $\varphi(\cdot,b)$ is right-continuous at $0$ for every 
	$b \geq 0$ then
	\begin{equation*}
	\Phi(a,\, b) \coloneq (\kappa b+ \kappa)(a+\varphi(a,\, b)).
	\end{equation*}
	is a strict consistent error bound function for $\stdCone$ and  $\stdAffine$.
\end{proposition}

\begin{proof}
	The function $\varphi$ in Theorem~\ref{theo:err} is constructed from facial residual functions using the diamond composition defined in \eqref{eq:comp}. Since facial residual functions are,
	by definition, increasing in each coordinate, the same is true of $\varphi$.
	When we fix $b$, the function $\Phi(\cdot, b)$ is monotone increasing because all its terms are monotone nondecreasing and the term $\kappa a$ is monotone increasing. Now it remains to prove 
	\begin{equation}\label{format}
	\dist\left(x,\, \stdCone \cap \stdAffine \right) \leq \Phi\left(\max(\dist(x,\, \stdCone),\, \dist(x,\, \stdAffine)), \|x\|\right) \quad \forall x \in \mathcal{E}.
	\end{equation}
	The error bound in Theorem~\ref{theo:err}  holds for $x \in \spanVec \stdCone$. However, $\stdCone$ is full-dimensional, so $\spanVec \stdCone = \mathcal{E}$. Therefore, for $x \in \mathcal{E}$ if we let $\epsilon = \max(\dist(x,\, \stdCone),\, \dist(x,\, \stdAffine))$ in Theorem~\ref{theo:err}, we obtain \eqref{format}. Since $\varphi(\cdot,b)$ is right-continuous at $0$ for every $b$, $\Phi$ is indeed a consistent error bound function for $\stdCone$ and $\mathcal{V}$.
\qed\end{proof}
The only gap between Proposition~\ref{prop:err_am_geb} and Theorem~\ref{theo:err} is that the function $\varphi$ in the latter might not satisfy right-continuity at $0$. We address this issue next.

\begin{proposition}[Existence of facial residual functions satisfying right-con\-tinuity at $0$]\label{prop:frf_ex}
	Let $\stdCone$ be a closed convex cone, $\stdFace \subseteq \stdCone$ be a face and $z \in \stdFace^*$. 	
	There exists a facial residual function $\psi_{\stdFace,z}$ for $z$ and $\stdFace$ such that   $\psi_{\stdFace,z}(\cdot,b)$ is right-continuous at $0$ for every $b \geq 0$.
	In particular, under the setting of Theorem~\ref{theo:err}, there exists $\varphi:\R_+\times \R_+ \to \R_+$ such that $\varphi(\cdot,b)$ satisfies right-continuity at $0$ for every $b \geq 0$.  
\end{proposition}
\begin{proof}
	Because we have $\stdFace = \stdCone \cap \spanVec \stdFace$ whenever $\stdFace \subseteq \stdCone$ is a face, the following equality holds:
	\[
	\stdFace \cap \{z\}^\perp = \stdCone \cap \spanVec \stdFace \cap \{z\}^\perp.
	\]
	To construct a facial residual function, we follow an 
	approach similar to the proof of Proposition~\ref{prop:uni} and Section~3.2 in \cite{L19}.  Let $\psi _{\stdFace,z}(\epsilon,\norm{x})$ be the optimal value of the following problem.
	\begin{align}
	\underset{v \in \spanVec \stdCone}{\sup} & \quad \dist(v,  \stdFace \cap \{z\}^\perp ) \label{eq:can}\tag{P}\\ 
	\mbox{subject to} & \quad \dist(v,\stdCone) \leq \epsilon \nonumber \\ 
	&\quad \dist(v, \spanVec \stdFace ) \leq \epsilon \nonumber\\
	&\quad \inProd{v}{z} \leq \epsilon \nonumber\\
	&\quad \norm{v} \leq \norm{x}\nonumber
	\end{align}
	Because $0 \in \stdFace \cap \{z\}^\perp$, \eqref{eq:can} is always feasible and the last constraint ensures compactness.
	With that, $\psi _{\stdFace,z}$ satisfy all the requirements in Definition~\ref{def:frf}. For every $b \geq 0$, 
	it can be shown that $\psi _{\stdFace,z}(\cdot,b)$ is right-continuous at $0$ by following the same argument used for showing the right-continuity of the best error bound function in the proof of Proposition~\ref{prop:uni}.
	
	Next, we observe that if $\psi_1$ and $\psi_2$ are two facial residual functions satisfying right-continuity at $0$, than their diamond composition \eqref{eq:comp} is also right-continuous at $0$, whenever the second argument is fixed. Therefore, under the setting of Theorem~\ref{theo:err}, the functions $\psi_i$ appearing therein can  all be selected in such a way that they satisfy right-continuity at $0$. So the same is true for the function $\varphi$ which is a diamond composition of facial residual functions.
\end{proof}


In view of Propositions~\ref{prop:err_am_geb} and \ref{prop:frf_ex}, 
when applying the methods of Section \ref{sec:proj_alg} to \eqref{eq:clp}, the convergence rate is governed by 
$\Phi$. Although it might not be clear at first, 
their convergence rates depend on the singularity degree  of the problem. This is because the singularity degree influences $\Phi$, which controls the error bound between $\stdCone$ and $\stdAffine$.
In the next subsection, we take a look at the special case of symmetric cones, where the error bounds and the rates are more concrete.

\subsection{The case of symmetric cones}\label{sec:sym}
A convex cone $\stdCone \subseteq \mathcal{E}$ is \emph{symmetric} if $\stdCone = \stdCone^*$  and 
for every $x,y \in \reInt \stdCone$ there exists a bijective linear map $A$ satisfying $Ax = y$, $A\stdCone = \stdCone$.
Symmetric cones are intrinsically connected to the theory of Euclidean Jordan Algebras, see \cite{K99,FK94,FB08}. We now recall some basic facts about them. 
 Examples of symmetric cones include the second-order cone, the symmetric positive semidefinite matrices over the reals, the nonnegative orthant and direct products of those cones. 
There is a notion of \emph{rank} for symmetric cones and the longest chain of faces of a symmetric cone is given by  $\ell_{\stdCone} = \matRank	\stdCone +1$, see \cite[Theorem 14]{IL17}. Finally, 
symmetric cones are amenable and their 
facial residual functions were computed in 
\cite[Theorem~35]{L19}.
%
With that, the following error bound holds.
\begin{theorem}[Theorem~37 and Remark~39 of \cite{L19}]\label{theo:sym_err}
	Let $\stdCone \subseteq \mathcal{E}$ be a symmetric cone, $\stdAffine \subseteq \mathcal{E}$ an affine subspace such that $\stdCone \cap \stdAffine \neq \emptyset$.	
	Then, there is a positive constant $\kappa$  such that whenever  $x$ and $\epsilon$ satisfy the inequalities
	\[
	\quad \dist(x,\, \stdCone) \leq \epsilon, \quad \dist(x, \, \stdAffine) \leq \epsilon,
	\]
	we have 
	\[
	\dist\left(x,\, \stdCone\cap \stdAffine\right) \leq 
	(\kappa \norm{x} + \kappa)\left(\sum _{j = 0}^{\dpp(\stdCone,\, \stdAffine)}  \epsilon^{({2^{-j})} }\norm{x}^{1-{2^{-j}} }\right).
	\]	
	If $\stdCone = \stdCone^1 \times \cdots \times \stdCone^s$ is the direct product of $s$ symmetric cones, we have 
	\[
	\dpp(\stdCone,\, \stdAffine) \leq \min \left\{\dim (\stdAffine^\perp), \, \sum _{i=1}^{s} (\matRank \stdCone^i -1),\, \ds(\stdCone,\stdAffine) \right\}.
	\]
\end{theorem}
Next, we verify that the error bound in Theorem~\ref{theo:err} is a \emph{bona fide} H\"olderian error bound.

\begin{proposition}\label{prop:sym_hold}
	Let $\stdCone$ and $\stdAffine$ be as in Theorem~\ref{theo:sym_err}. Then, $\stdCone$ and $\stdAffine$ satisfy a uniform H\"olderian error bound (Definition~\ref{def_hold}) with exponent
	$2^{-\dpp(\stdCone,\, \stdAffine)}$.
\end{proposition}
\begin{proof}
	Let $C_1 = \stdCone$ and $C_2 = \stdAffine$. By Theorem~\ref{theo:sym_err}, we have
	\begin{equation}
	\dist(x,\, \stdCone\cap \stdAffine) \le \left(\kappa\|x\| + \kappa\right)\left(\sum_{j=0}^{\dpp(\stdCone,\, \stdAffine)}\left(\max_{1\le i\le 2}\dist(x,\, C_i)\right)^{2^{-j}}\|x\|^{1-2^{-j}}\right) \ \ \ \forall x\in\mathcal{E}. \label{eq:global_err}
	\end{equation}
	Let $B \subseteq \mathcal{E}$ be an arbitrary bounded set. For simplicity of 
	notation, let $d = \dpp(\stdCone,\, \stdAffine)$ and $\psi$ be the function such that 
	\[ 
	\psi(x) = \max_{1\le i\le 2}\dist(x,\, C_i)  \quad \forall x \in \mathcal{E}.
	\]
	From the continuity of $\psi$, we see that for every $j \in \{0,\ldots, d\}$  there exists a  positive 
	constant $\kappa _j$ such that
	\[
	\psi(x)^{2^{-j}}= \psi(x)^{2^{-j}-2^{-d}}\psi(x)^{2^{-d}} \leq \kappa_j \psi(x)^{2^{-d}} \quad \forall x \in B,
	\]
	where $\kappa_j$ can be taken, for example, to be the supremum of $\psi(\cdot)^{2^{-j}-2^{-d}}$ over $B$. Similarly, there are positive constants 
	$\tilde \kappa_j$ and $\kappa_b$ such that 
	\[
	\norm{x} \le \kappa_b, \ \ \   \norm{x}^{1-2^{-j}} \leq  \tilde \kappa _j \quad \forall x \in B.
	\]
	Let $\kappa_B := \kappa(\kappa_b + 1)(d+1)\sup_{j}\kappa_j\tilde \kappa_j$. It follows that whenever $x$ belongs to $B$ the right-hand side of  \eqref{eq:global_err} is  upper bounded by
	$
	\kappa _B \left(\max_{1\le i\le 2}\dist(x,\, C_i)\right)^{2^{-d}}.
	$
\end{proof}

We now present convergence results for symmetric cones taking into account all we have 
discussed so far.

\begin{theorem}[Convergence rate results for symmetric cones]\label{theo:sym_conv}
	Let $\stdCone \subseteq \mathcal{E}$ be a symmetric cone and $\stdAffine \subseteq \mathcal{E}$ be an affine space such that 
	$\stdCone \cap \stdAffine \neq \emptyset$.

	Let $\{x^k\}$ be such that Assumption~\ref{assp} is satisfied
	with $\inf _{k}a_k\geq 0$.
	Then, there exist $M > 0$ and $\theta \in (0,1)$ such that
	for any $k\ge 2\ell$,
	\begin{equation}\label{rate_am}
	\dist(x^k,\, \stdCone \cap \stdAffine) \le \begin{cases}
	M\, k^{-\frac1{2\left(2^{\dpp(\stdCone,\, \stdAffine)} - 1\right)}} & \text{if the PPS condition is not satisfied},\\
	M\, \theta^k & \text{otherwise},
	\end{cases}
	\end{equation}
	In particular, the following holds.
	\begin{enumerate}[$(i)$]
		\item The rate \eqref{rate_am} holds for any algorithm satisfying the assumptions of Corollary~\ref{proj_hold_ge}. 
		\item The rate \eqref{rate_am} holds MPA, POCSA (in particular, CPA) , MM (in particular, MDPA) and AWPA (see Example~\ref{remark_proj}).
		\item 	If $\stdCone = \stdCone^1 \times \cdots \times \stdCone^s$ is the direct product of $s$ symmetric cones, we have 
		$
		\dpp(\stdCone,\, \stdAffine) \leq \min \left\{\dim (\stdAffine^\perp), \, \sum _{i=1}^{s} (\matRank \stdCone^i -1),\, \ds(\stdCone,\stdAffine) \right\}.
		$
	\end{enumerate}
\end{theorem}

\begin{proof}
	By Proposition~\ref{prop:sym_hold} a uniform H\"olderian error bound holds 
	between $\stdCone$ and $\stdAffine$, with  exponent ${2^{-\dpp(\stdCone,\, \stdAffine)}}$.
	If either Slater's condition or the Partial Polyhedral Slater's condition is satisfied, then 
	the error bound in Proposition~\ref{prop:sym_hold} becomes a Lipschitz error bound.
	Applying Corollary~\ref{coro_abs_rate}, we obtain \eqref{rate_am}. Item $(i)$ and $(ii)$ are consequences of Corollary~\ref{proj_hold_ge}. Item~$(iii)$ follows from 
	Theorem~\ref{theo:sym_err}.
\qed\end{proof}
\begin{remark}
	Theorem~\ref{theo:sym_conv} extends the main result of Drusvyatskiy, Li and Wolkowicz \cite{DLW17} in several directions: from semidefinite cones to 
	symmetric cones and from the alternating projection algorithm to any algorithm  covered by Corollary~\ref{coro_abs_rate}.
\end{remark}

\subsection{The exponential cone and non-H\"olderian error bounds}\label{sec:exp}
In this subsection, we analyze two error bounds associated to the exponential cone  \cite{Ch09,CS17,MC2020}, which is defined as follows
\begin{align*}
\expCone:=&\left \{(x,y,z)\in \R^3\;|\;y>0,z\geq ye^{x/y}\right \} \cup \left \{(x,y,z)\;|\; x \leq 0, z\geq 0, y=0  \right \},
\end{align*}
see Remark~\ref{rem:exp} for a discussion on applications.  

Unfortunately, Theorem~\ref{theo:err} does not apply to the exponential cone, because $\expCone$ is not amenable, see \cite{LiLoPo20}.
However, in \cite{LiLoPo20}, the authors proved a generalization of the results of \cite{L19} and  proved tight error bounds for the exponential cone, which we will discuss using our tools. In what follows, let 
\[
\stdAffine_1 \coloneqq \{(x,0,z) \mid x,z \in \R\}\quad \text{and} \quad \stdAffine_2 \coloneqq  \{(x,y,0) \mid x,y \in \R \}.
\]
We now consider the error bounds associated to the following feasibility problems
\begin{align}
{\rm find}\ p \in \expCone \cap \stdAffine_1, \label{eq:feas1}\\
{\rm find}\ p \in \expCone \cap \stdAffine_2 \label{eq:feas2}.
\end{align}
For $p \in \R^3$, we define 
$d_i(p) \coloneqq \max\{\dist(p,\expCone), \dist(p,\stdAffine_i) \}$, for $i = 1,2$. 
We also need the following functions.
\begin{equation*}
\frakg_{-\infty}(t) := \begin{cases}
0 & \text{if}\;\; t=0,\\
-t \ln(t) & \text{if}\;\; t\in \left(0,1/e^2\right],\\
t+ \frac{1}{e^2} & \text{if}\;\; t>1/e^2.
\end{cases}\qquad
\frakg_\infty(t) :=  \begin{cases}
0 &\text{if}\;t=0,\\
-\frac{1}{\ln(t)} & \text{if}\;0 < t\leq \frac{1}{e^2},\\
\frac{1}{4}+\frac{1}{4}e^2t & \text{if}\;t>\frac{1}{e^2}.
\end{cases}
\end{equation*}
These functions arise in the computation of the facial residual functions for the exponential cone.
From \cite[Theorem~4.13]{LiLoPo20} and items (a) and (c) of \cite[Remark~4.14]{LiLoPo20}, we have that for every ball
$B_b \coloneqq \{p \in \R^3\mid \norm{p}\leq b\}$ with $b > 0$, there are constants $\rho_1(b)$ and $\rho_2(b)$ such that 
\begin{equation}\label{eq:l1}
\dist(p,\expCone \cap \stdAffine_1) \leq \rho_1(b) \frakg_{-\infty}(d_1(p)), \quad \forall p \in B_b
\end{equation}
and
\begin{equation}\label{eq:l2}
\dist(p,\expCone \cap \stdAffine_2) \leq \rho_2(b) \frakg_{\infty}(d_2(p)), \quad \forall p \in B_b.
\end{equation}
Naturally, $\rho_1$ and $\rho_2$ can be chosen so that they are monotone nondecreasing functions of $b$.
Because $\frakg_{-\infty}$ and $\frakg_{\infty}$ are continuous monotone increasing functions, we have  the following  strict consistent error bound functions for the problems in \eqref{eq:feas1} and \eqref{eq:feas2}, respectively:
\begin{equation}\label{eq:phi_exp}
\Phi_{\mathrm{et}}(a,b) \coloneqq \rho_1(b) \frakg_{-\infty}(a), \qquad \Phi_{\ln}(a,b) \coloneqq \rho_2(b) \frakg_{\infty}(a).
\end{equation}
These are examples of \emph{entropic} and \emph{logarithmic} error bounds, respectively. 
We note that it was proved in  \cite[Example~4.20]{LiLoPo20} that no H\"olderian error bound holds for the problem \eqref{eq:feas2}.
Furthermore, the bounds in \eqref{eq:l1} and \eqref{eq:l2} are tight up to a constant, see \cite[Remark~4.14]{LiLoPo20}.

Using $\Phi_{\mathrm{et}}$ and $\Phi_{\ln}$ in \eqref{eq:phi_exp} we can analyse the convergence rate of algorithms for \eqref{eq:feas1} and \eqref{eq:feas2}.  An initial hurdle to our enterprise is that it is challenging to obtain closed-form expressions for ${(\Phi_{\mathrm{et}})}_{\kappa}^{\spadesuit}$, 
${(\Phi_{\ln})}_{\kappa}^{\spadesuit}$ and their inverses.
On the other hand, checking that $\Phi_{\mathrm{et}}$ and $\Phi_{\ln}$ are regularly varying functions is straightforward and we will use the the machinery developed in Section~\ref{sec:rv}. 

\begin{proposition}\label{prop:exp}
	The following items hold for any $\kappa > 0$.
	\begin{enumerate}[$(i)$]
		\item $\Phi_{\mathrm{et}}(\cdot,\kappa)$ belongs to $\RVz$ with index $1$ and $\Phi_{\ln}(\cdot, \kappa)$ belongs to $\RVz$ with index $0$.
		\item ${(\Phi_{\mathrm{et}})}_{\kappa}^{\spadesuit}(t)\to-\infty$ and ${(\Phi_{\ln})}_{\kappa}^{\spadesuit}(t)\to-\infty$ as $t\to 0_+$.
		\item The convergence rate afforded by
		$\Phi_{\mathrm{et}}$ is \textbf{almost linear} in the following sense: for any $r>0$, the following relations hold as $s \to +\infty$
		\begin{align*}
	 \sqrt{({(\Phi_{\mathrm{et}})}_{\kappa}^{\spadesuit})^{-1}(-s)} = o(s^{-r}),\qquad
	 	e^{-rs} = o\left(\sqrt{({(\Phi_{\mathrm{et}})}_{\kappa}^{\spadesuit})^{-1}(-s)} \right).
		\end{align*}
		\item 
		The convergence rate afforded by $\Phi_{\ln}$ is \textbf{logarithmic}
		in the following sense: 
		there exists $\eta_1 > 0$, $\eta_2 > 0$ and $N$ such that for $s \geq N$, we have
		\begin{align*}
	\eta_1\left(\frac{1}{\ln(s)}\right)\leq \sqrt{({(\Phi_{\ln})}_{\kappa}^{\spadesuit})^{-1}(-s)} \le \eta_2\left(\frac{1}{\ln(s)}\right).	
		\end{align*}
	\end{enumerate}	
\end{proposition}
\begin{proof}
	That item~$(i)$ holds can be readily checked
	by computing the limit in \eqref{eq:rvz}.
Next, we will use 	Proposition~\ref{prop:asym} to verify item~$(ii)$. We note that the feasible sets of \eqref{eq:feas1} and \eqref{eq:feas2} both contain the origin, so $\dist(0,\expCone\cap  \stdAffine_1) = \dist(0,\expCone \cap  \stdAffine_2) = 0$. Furthermore, both feasible regions are contained in two-dimensional sets, so  $\expCone\cap  \stdAffine_1$ and $\expCone\cap  \stdAffine_2$ have empty interior. In particular, there are points $p_1,p_2$ with $\norm{p_1} \leq \kappa, \norm{p_2} \leq \kappa$ such that $p_1\not \in \expCone\cap  \stdAffine_1$ and $p_2 \not \in \expCone\cap  \stdAffine_2$.  This shows that $\kappa$ satisfies the inequality in the statement of Proposition~\ref{prop:asym} for both $\Phi_{\mathrm{et}}$ and $\Phi_{\ln}$, which proves the desired limits.
	
	We move on to item~$(iii)$ and 
	let $r > 0$ be arbitrary.
	From item~$(iii)$ of 	Theorem~\ref{thm:upper}, we have $\sqrt{({(\Phi_{\mathrm{et}})}_{\kappa}^{\spadesuit})^{-1}(-s)} = o(s^{-r}) $ as $s \to+\infty$.
	Next, let $\Phi_1(t,\kappa) := rt$, so that $\Phi_1$ is a strict error bound function.
	Following the computations 
	after \eqref{eq:phi_spade_lin}, we have 
	$\sqrt{({(\Phi_1)}_{\kappa}^{\spadesuit})^{-1}(s)} = e^{s/(2r^2)}$. 
	We have  $rt = o(
	\Phi_{\mathrm{et}}(t,\kappa))$ as $t \to 0_+$.
	By Theorem~\ref{thm_comp},
	we have \[
	e^{s/(2r^2)} = o\left(\sqrt{({(\Phi_{\mathrm{et}})}_{\kappa}^{\spadesuit})^{-1}(s)} \right),
	\]
	as $s \to -\infty$.	Since $r$ is arbitrary, this completes item $(iii)$. 
	
	Finally, item $(iv)$ follows from 
	Theorem~\ref{thm:lograte} because 
	$\Phi_{\ln}$ corresponds to a logarithmic error bound with exponent $1$. 
\qed\end{proof}
As an example, suppose that we are interested in the behaviour of the cyclic projection  algorithm (CPA) when applied to \eqref{eq:feas1} and 
	\eqref{eq:feas2}. 
	We will denote the iterates generated by CPA by $p^k$ and the initial iterate by $p^0$.
	In the numerical experiments that follow, we use the code developed by Friberg in order to compute the projection onto the exponential cone, see \cite{Fr21}.
	
	First, we consider \eqref{eq:feas1}.
	From item~$(i)$ of Theorem~\ref{thm_sp_4} and item~$(iii)$ of Proposition~\ref{prop:exp}, $\dist(p^k, \expCone \cap \stdAffine_1 )$ goes to $0$ ``almost linearly'' in the sense that 	the rate is faster than $k^{-r}$ for any $r > 0$. To check this empirically, we let $p^{0} = (1,1,1)$ and plot in Figure~\ref{fig:entropic}  the iteration number $k$ against  $\dist(p^k, \expCone \cap \stdAffine_1 )$ (which can be computed exactly in this example). Both axes are in log scale, so that 
	$k^{-r}$ appears as a straight line for any $r$. Figure~\ref{fig:entropic} shows that, as predicted by theory, $\dist(p^k, \expCone \cap \stdAffine_1 )$	goes to $0$ faster than any sublinear rate. Item~$(iii)$ of Proposition~\ref{prop:exp} also gives a lower bound to 	$\sqrt{({(\Phi_{\mathrm{et}})}_{\kappa}^{\spadesuit})^{-1}(-s)}$ and tells us that this function goes to $0$ slower than $e^{-rs}$ for any $r$. Now, a lower bound to  $\sqrt{({(\Phi_{\mathrm{et}})}_{\kappa}^{\spadesuit})^{-1}(-s)}$ does not necessarily lead to a lower bound to $\dist(p^k, \expCone \cap \stdAffine_1 )$, so we cannot immediately refute the possibility that $\dist(p^k, \expCone \cap \stdAffine_1 )$ goes to $0$ linearly. However using a plot where only $y$-axis is in log-scale, we see indication that the convergence rate of $\dist(p^k, \expCone \cap \stdAffine_1 )$ is indeed not linear, see Figure~\ref{fig:entropic_lin}.  In this example, it seems that $\sqrt{({(\Phi_{\mathrm{et}})}_{\kappa}^{\spadesuit})^{-1}(-s)}$ closely reflects the true convergence rate.
	
	\begin{figure}\centering
		\begin{subfigure}[b]{0.45\textwidth}
			\centering
			\includegraphics[width=\textwidth]{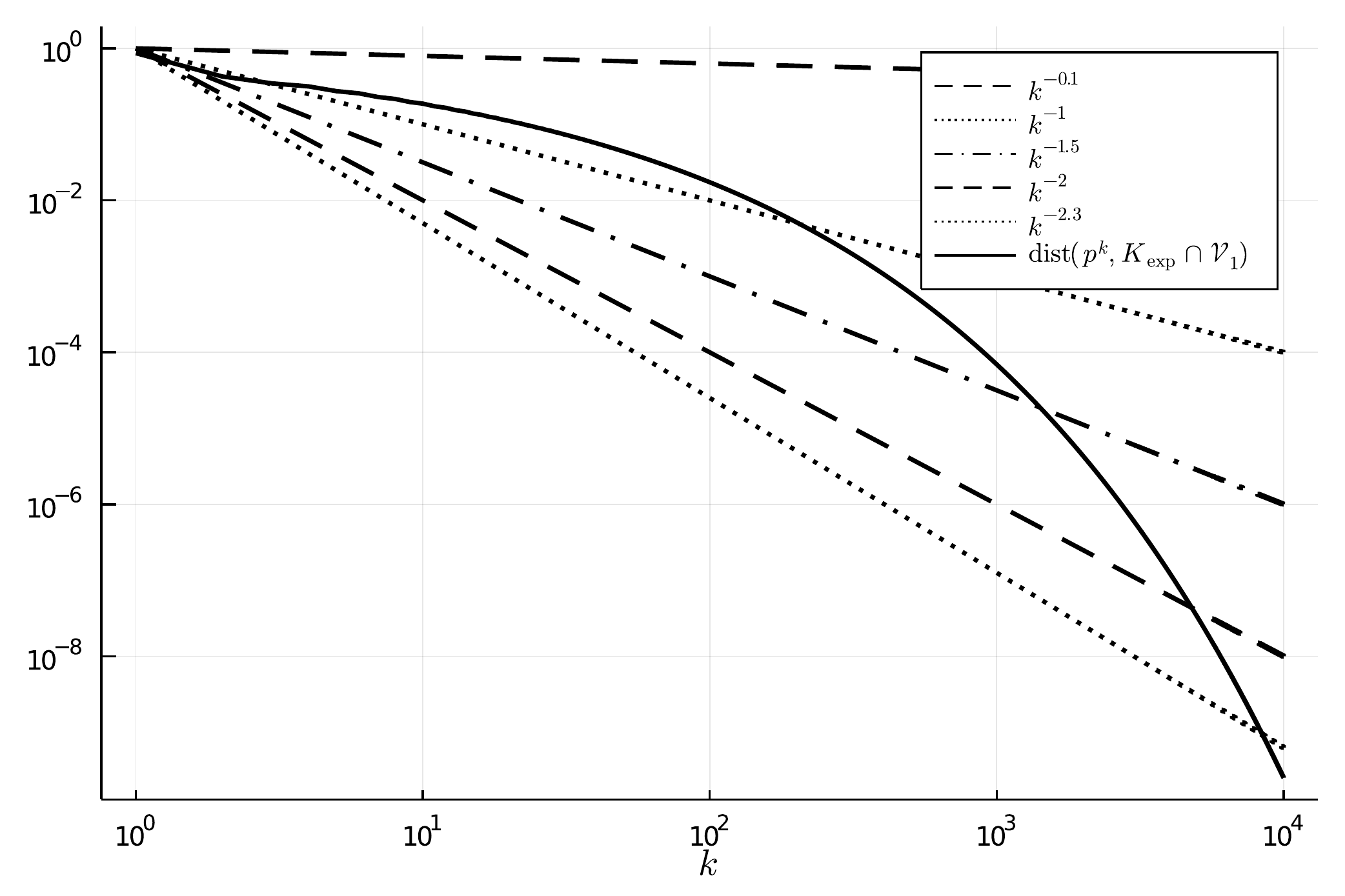}
\caption{Log-log plot of $\dist(p^k, \expCone \cap \stdAffine_1 )$. Dashed and dotted lines correspond to $k^{-r}$ for a few values of $r$.}\label{fig:entropic}
		\end{subfigure}	
		\begin{subfigure}[b]{0.45\textwidth}
		\centering
		\includegraphics[width=\textwidth]{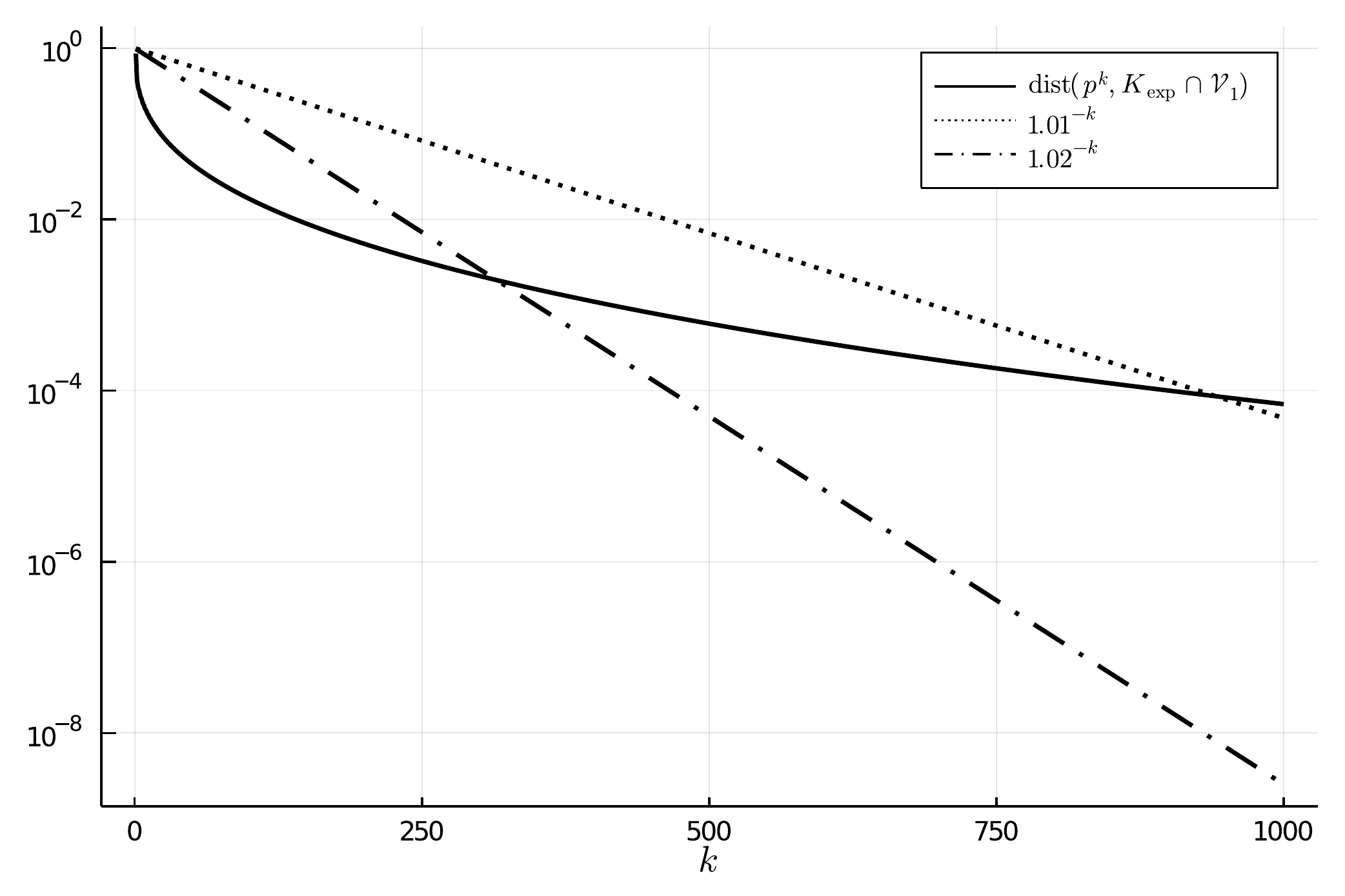}
		\caption{Plot of $\dist(p^k, \expCone \cap \stdAffine_1 )$, where only the $y$-axis is in log scale. Functions of the form $c^{-k}$ appear as straight lines. 
		}\label{fig:entropic_lin}
	\end{subfigure}	
	\caption{Behavior of CPA applied to \eqref{eq:feas1}.  Starting point is $(1,1,1)$.}
	\end{figure}
	
	

	
	Next, we move on to \eqref{eq:feas2}.
By item~$(iv)$ of Proposition~\ref{prop:exp}, we have that 
the convergence rate is at least logarithmic.
In principle, this does not exclude the possibility that the true convergence rate
of $\dist(p^k, \expCone \cap \mathcal{V}_2)$ is faster.   
	However, Figure~\ref{fig:log} suggests that $\dist(p^k, \expCone \cap \stdAffine_2 )$ goes to $0$ slower than $k^{-r}$ for any $r > 0$, which again suggests that $({(\Phi_{\ln})}_{\kappa}^{\spadesuit})^{-1}(-s)$ is reflective of the true convergence rate. 

	\begin{figure}
	\centering
	\includegraphics[width=0.5\columnwidth]{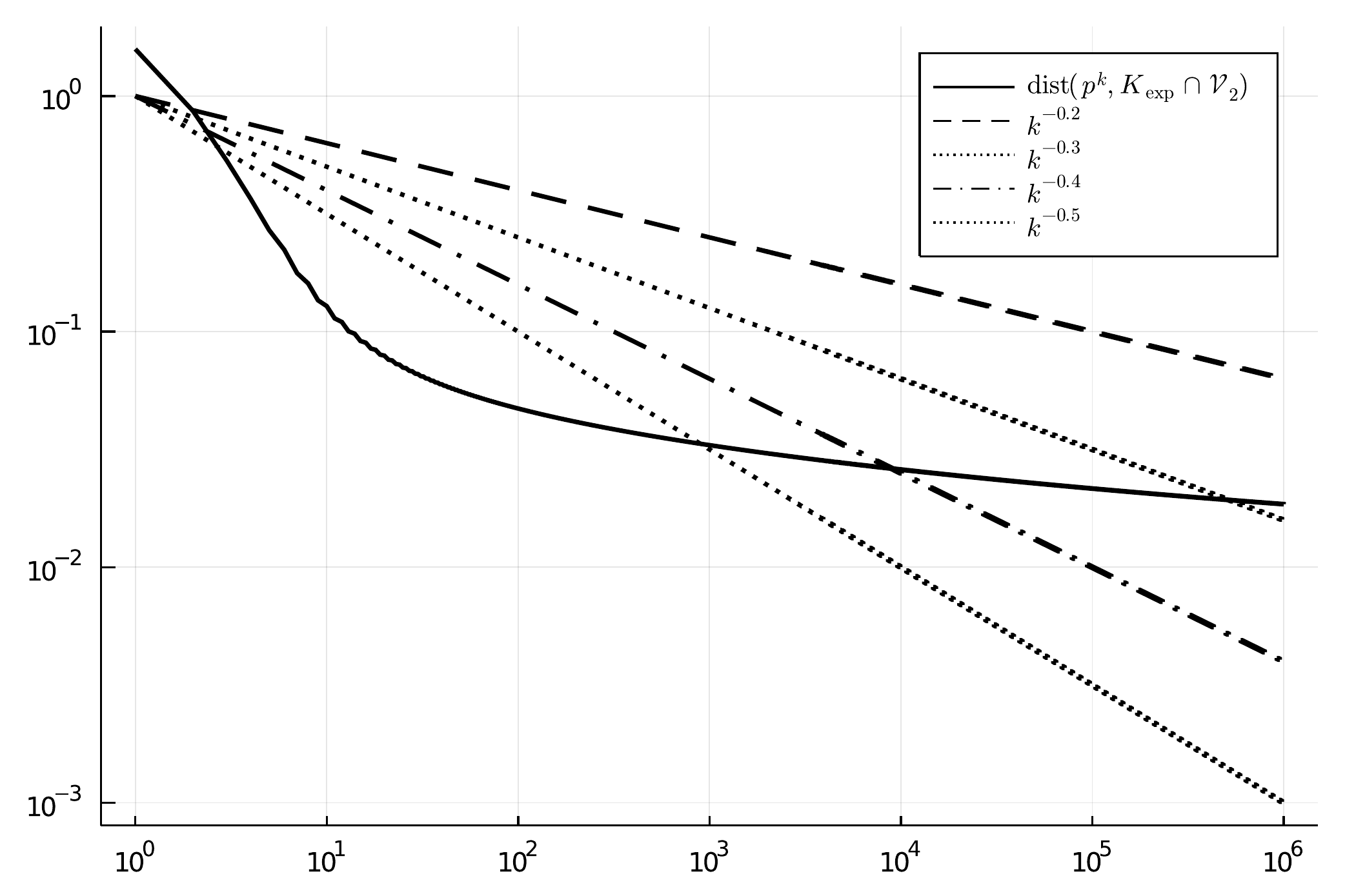}
	\caption{Log-log plot of $\dist(p^k, \expCone \cap \stdAffine_2 )$ for the iterates generated by CPA. Starting point is $(1,1,1)$. Dashed and dotted lines correspond to $k^{-r}$ for a few values of $r$.}\label{fig:log}
\end{figure}

\begin{remark}[On the exponential cone and beyond]\label{rem:exp}
The exponential cone is a building block for modelling many important problems related to entropy optimization, geometric programming and others, see \cite{Ch09,CS17,MC2020}. 
For example, the Kullback-Leibler divergence between two nonnegative vectors $x,y\in \R^n$ is defined as $D(x,y) \coloneqq \sum _{i} x_i \ln(x_i/y_i) $ and its epigraph is often modelled using $n$ exponential cones as follows:
\[
t \geq t_1+\cdots+ t_n, \quad (-t_i,x_i,y_i)  \in \expCone, i \in {1,\ldots, n},
\]
as indicated, for example, in \cite[Section~1.1]{CS17} and \cite[Chapter~5]{MC2020}.
In particular, the problem of minimizing the Kullback-Leibler divergence subject to linear constraints on $x$ and $y$ can be expressed as a conic linear program (CLP) over a product of exponential cones.
Notably, in  \cite{LYBV16}, the authors found that nearly one third of a library of more than $300$ instances of mixed integer continuous optimization problems can be modelled using mixed integer conic formulations with exponential cone constraints, see Table~1 therein. 
Certain relaxations of these problems naturally lead to CLPs over a direct product of exponential cones.  
Although we have discussed only the case of a single exponential cone, our results are representative of what can happen in more general settings.

There is now a larger movement towards algorithms, software and theory for non-symmetric cones with quite a few  solvers  supporting exponential cones, e.g.,  \cite{KT19,DY01,CKV21,MC2020}. 
These references also discuss other convex sets involving logarithms and exponentials, such as the the log-determinant cone in \cite{CKV21}.
On a more speculative note, it seems likely that some intersections involving those sets will have non-H\"olderian error bounds due to the presence of exponentials and logarithms. Therefore, the techniques discussed in this section and in Section~\ref{sec:rv} will  likely be applicable as well.
\end{remark}

\section{Concluding remarks}\label{sec:conc}

In this paper we proposed the notion of (strict) consistent error bounds. Under a strict consistent error bound, we established convergence rates for a family of algorithms for the convex feasibility problem \eqref{CFP}. The key idea is to construct an \emph{inverse smoothing function}  based on the corresponding consistent error bound function. Our analysis recovers several old results and also gives several new ones. We also apply the convergence results to conic feasibility problems in order furnish further links between the singularity degree of the underlying problem and the convergence rate of several algorithms. Another novel aspect is the usage of regularly varying functions, which allows to draw conclusions about convergence rates while avoiding certain complicated computations.
To conclude this paper, we first make some comparisons to approaches based on the KL-property.

\subsection{On the Kurdyka-{\L}ojasiewicz (KL) property and related concepts}
The Kurdyka-{\L}ojasiewicz (KL) property is an 
important and remarkable tool for convergence analysis used successfully in several works \cite{ABRS10,abs13,LP18}, so in this 
subsection we make a few comparisons in order to explain what could or what could (probably) not be done under the KL framework. 

First, there is a close relation between error bounds and the KL property in the presence of convexity.  As shown in  \cite[Theorem~30]{bdlm10} and \cite[Theorem~5]{BNPS17}, under certain conditions on $\varphi$, an error bound of the form ``$\dist(x,\,\argmin\,f) \le \varphi(f(x))$''
implies that $f$ satisfies the KL property with a desingularization function involving $\varphi$.
Under our setting, there are several candidates for $f$ but they will, in all likelihood, be functions involving terms of the form $\max_i\dist(x,\,C_i)$ or positive combinations of the $\dist(x,C_i)^2$, for example.

The choice of $f$ must be typically tailored to the target algorithm. 
Our understanding is that most of the algorithms in Section~\ref{sec:proj_alg} would require different
choices of $f$ in order for the analysis to be carried out 
under the KL framework. Finding the appropriate $f$ can be nontrivial, as illustrated by the merit function for the Douglas-Rachdford algorithm in \cite{LP16}.
It might also be \emph{impossible} in some cases.
For example, based on a result by Baillon, Combettes and Cominetti~\cite{BCC12}, it is claimed in a footnote in \cite{BNPS17} that there is no potential function corresponding to the cyclic projection algorithm (CPA, see Example~\ref{remark_proj}) for more than two sets.

Once the appropriate potential function is identified, 
it is necessary to show that certain conditions hold for the potential function along the sequence, e.g., the sufficient decrease condition and the relative error condition, see \cite{ABRS10,abs13,ochs19}.
These properties and Assumption~\ref{assp} have a similar motivation: ensuring that the sequence generated by the underlying algorithm satisfies some desirable properties.

If a convergence rate is desired, one usually has to show that the potential function satisfies the KL property with some KL \emph{exponent}. The general KL property holds under relatively mild conditions, but identifying the exponent (if one exists) is a more challenging task, see \cite{LP18}. Due to \cite[Theorem~5]{BNPS17}, existence of a KL exponent is equivalent to the validity of a H\"olderian error bound, so  establishing the former or the latter are tasks of comparable difficulty. We note that the logarithmic error bound example in \eqref{eq:feas2} can be used to construct a function which does not have a KL exponent, see \cite[Example~4.22]{LiLoPo20}. Similarly, $f_{\gamma}$ in Example~\ref{ex:log2} has no KL exponent. In particular, the convergence rate results based on the existence of a KL exponent do not seem applicable to \eqref{eq:feas2} nor to Example~\ref{ex:log2}.

%
That said, it is possible to analyze convergence rates without assuming that a KL exponent holds, see \cite[Theorem~24]{bdlm10} and \cite[Theorem~14]{BNPS17} for results which only rely on the desingularizing function $\varphi$ without assumptions on the format of $\varphi$. And, interestingly, the existence of $\varphi$ can, sometimes, be characterized via certain integrals involving subgradient curves, see \cite[Theorem~18]{bdlm10}. However, 
we do not immediately see a connection between the integrals appearing 
in \cite[Theorem~18]{bdlm10} and in \eqref{def_phi}.
We do note, however, that a certain optimal desingularizing function can be characterized via an integral, see \cite[Section~3.2]{ww20}. Similarly, 
if the {best consistent error bound function} in Proposition~\ref{prop:uni} is strict, it can be used 
to construct the inverse smoothing function $\Phi_{\kappa}^{\spadesuit}$ as in \eqref{def_phi}.
So both integrals seem to be able to capture optimal phenomena, under certain conditions.

 Another point is that the upper bounds in \cite[Theorem~24]{bdlm10} and \cite[Theorem~14]{BNPS17} include expressions of the format $\varphi(f(x^k) -\kappa)$ (for some constant $\kappa$), so they are still dependent on the iterate $x^k$ and it might be fair to say they require some work in order to get an explicit convergence rate in terms of $k$. 
In contrast, our upper bound on the convergence rate in \eqref{rate} does not rely on the iterate $x^k$ and only uses the iteration number $k$ itself, which gives a more explicit expression. The drawback is that one must deal with the $(\Phi_{\widehat{\kappa}}^\spadesuit)^{-1}$ term that appears in \eqref{rate}, which is indeed nontrival. Nevertheless, as shown in Section~\ref{sec:rv} and illustrated in Section~\ref{sec:exp}, there are ways of bypassing this difficulty if the consistent error bound function is a function of regular variation.

Finally, we remark that the KL inequality is, of course, heavily connected to semialgebraic geometry \cite{BCR98}, so one might wonder the extent to which our results
could also be obtained by imposing semialgebraic assumptions on $\Phi$ or on the sets $C_i$. Our assessment is that this seems unlikely, because the results in Section~\ref{sec:rv} are also applicable to sets and functions involving exponentials and logarithms (as in Example~\ref{ex:log2} and Section~\ref{sec:exp}), which are not semialgebraic in general.

%

\subsection{Future directions}
At last, we mention some possible future directions.
In the concluding remarks of \cite{blt17}, the authors mention the characterization of convergence rates in the absence of H\"olderian regularity as an area of future research. 
We believe that the tools developed in this paper are a step forward towards this research goal, since Theorem~\ref{thm_conv} is quite general. And, indeed, we were able to reason about convergence rates in non-H\"olderian settings as described in Sections~\ref{sec:log} and \ref{sec:exp}.

In addition, it might be fair to say that regular variation has been rarely explored in the context of optimization algorithms and we believe there is significant room for further exploration. 
For example, we showed that consistent error bound functions  always exist (Proposition~\ref{prop:uni}). It could be interesting to try to prove whether a \emph{regularly varying} consistent error bound function always exists as well.
Since regular variation is connected to upper bounds for the convergence rate (Theorem~\ref{thm:upper}), exploring this kind of question might lead to some insights on whether 
arbitrary slow convergence is possible in finite dimensions, which is another open problem 
mentioned in the conclusion of \cite{blt17}.

Finally, we believe it would be interesting to analyse convergence rates of other algorithms beyond projection methods. A natural candidate would be the \emph{Douglas-Rachford} (DR) algorithm \cite{DR56,LM79}, which was also extensively analyzed in \cite{blt17}. However, the convergence rate results obtained in \cite[Proposition~4.2]{blt17} require not only an error bound condition on the underlying sets, but also a semialgebraic assumption. 
This suggests that it might be hard to obtain convergence rates for the DR algorithm purely based on  consistent error bounds. 
On the other hand,   damped versions of the DR algorithm (see \cite[Section~5]{blt17} or \cite[Equation~(25)]{DY17}) might be more amenable to our techniques. In fact, sublinear rates were proved in \cite[Theorem~5.2]{blt17} when the underlying error bound is H\"olderian without the need of imposing extra assumptions, see also \cite[Remark~5.3]{blt17}. In view of this, we believe it is likely that a result analogous to Theorem~\ref{thm_conv} and suitable for  damped DR algorithms holds.
%

{\small{
		\subsection*{Acknowledgements}
We thank the referees and the associate editor for their  comments, which helped to improve the paper.
The authors would like to thank Masaru Ito and Ting Kei Pong for the feedback and helpful comments during the writing of this paper.
The first author is supported by ACT-X, Japan Science and Technology Agency (Grant No. JPMJAX210Q).
The second author is partially supported by the
JSPS Grant-in-Aid for Young Scientists 19K20217 and the Grant-in-Aid for Scientific Research (B)18H03206 and 21H03398.
}}

\appendix

\section{Proof of Lemma~\ref{inv_lemma}}\label{appendix_a}

\begin{proof}
The fact that $f^{-}(0) = 0$ follows from $f(0) = 0$ and the definition \eqref{inv_fun}. We also note that in
 \eqref{inv_fun}, if we increase $s$, the set after the `$\inf$' potentially shrinks, so $f^{-}$ is monotone nondecreasing. 	Next, we prove each item.
	
\begin{enumerate}[$(i)$]
	\item Fix any $s\in(0,\, \sup f)$. Suppose that $f^{-}(s) = 0$. By the definition \eqref{inv_fun}, given any $\epsilon_k > 0$, there exists $t_k\in[0,\, \epsilon_k]$ such that $f(t_k)\ge s$. Consequently, there exists a sequence $t_k\to 0_+$ with $f(t_k)\ge s > 0$. This together with $f(0) = 0$ contradicts the (right)-continuity of $f$ at $0$, and thus proves $(i)$.
	\item Let $s\ge 0, t \geq 0$ be such that $s\le f(t)$. Since $f$ is monotone increasing, $\sup f$ is never attained, which 
	implies  $0\le s\le f(t) < \sup f$. 
	Furthermore, by the definition \eqref{inv_fun},
	we have $f^{-}(s)\le t$. 
	
	\item Let $s\ge 0, t \geq 0$ be such that $s < \sup f$ and $f(t) < s$. By definition, $f^{-}(f(t)):=\inf\left\{u\ge 0: f(u)\ge f(t)\right\}$,
	therefore $f^{-}(f(t))\le t$. On the other hand, the strict monotonicity of $f$ implies that there is no $u < t$ with $f(u)\ge f(t)$. This implies $f^{-}(f(t))\ge t$ and thus $f^{-}(f(t)) = t$. Together with the monotonicity of $f^{-}$, we obtain $t = f^{-}(f(t))\le f^{-}(s)$.
	
	\item  Suppose that there exists some $\widebar{s}\in(0,\, \sup f)$ such that $f^{-}$ is not continuous at $\widebar{s}$. Since $f^{-}$ is monotone, both the left-sided limit $f^{-}(\widebar{s}-)$ and the right-sided limit $f^{-}(\widebar{s}+)$ exist and $f^{-}(\widebar{s}-) < f^{-}(\widebar{s}+)$. Fix any $t\in(f^{-}(\widebar{s}-),\, f^{-}(\widebar{s}+))$. From the monotonicity of $f^{-}$, there exists $\epsilon > 0$ such that
	whenever $s_1,s_2$ satisfy $0 < s_1 < \widebar{s} < s_2 < \sup f$ we have
	\begin{equation*}
	f^{-}(s_1) < t - \epsilon < t + \epsilon < f^{-}(s_2).
	\end{equation*}
	We now show that $f(t) = \widebar{s}$. Suppose that $f(t)\neq \widebar{s}$. Then either $f(t) < \widebar{s}$ or $f(t) > \widebar{s}$. If $f(t) < \widebar{s}$, let $s_1 = (f(t) + \widebar{s})/2 \in (f(t),\, \widebar{s})$. Thus, we know from item  $(iii)$ that $f^{-}(s_1)\ge t$, which contradicts $f^{-}(s_1) < t - \epsilon$. 
	
	If $f(t) > \widebar{s}$, let $s_2 = (f(t) + \widebar{s})/2 \in (\widebar{s},\, f(t))$. Then, from item~$(ii)$, we have $f^{-}(s_2) \le t$, which contradicts  $t + \epsilon < f^{-}(s_2)$. This proves  $f(t) = \widebar{s}$. The arbitrariness of  $t\in(f^{-}(\widebar{s}-),\, f^{-}(\widebar{s}+))$ contradicts the strict monotonicity of $f$. Consequently, $f^{-}$ is continuous on $(0,\, \sup f)$. 	
\end{enumerate}	
\qed\end{proof}

\bibliographystyle{abbrvurl}
\bibliography{bib_plain}

\begin{thebibliography}{10}

\bibitem{as54}
S.~Agmon.
\newblock The relaxation method for linear inequalities.
\newblock {\em Canadian Journal of Mathematics}, 6:382--392, 1954.

\bibitem{AC89}
R.~Aharoni and Y.~Censor.
\newblock Block-iterative projection methods for parallel computation of
  solutions to convex feasibility problems.
\newblock {\em Linear Algebra and its Applications}, 120:165 -- 175, 1989.

\bibitem{ABRS10}
H.~Attouch, J.~Bolte, P.~Redont, and A.~Soubeyran.
\newblock Proximal alternating minimization and projection methods for
  nonconvex problems: an approach based on the {K}urdyka-{{\L}}ojasiewicz
  inequality.
\newblock {\em Mathematics of Operations Research}, 35(2):438--457, 2010.

\bibitem{abs13}
H.~Attouch, J.~Bolte, and B.~F. Svaiter.
\newblock Convergence of descent methods for semi-algebraic and tame problems:
  proximal algorithms, forward--backward splitting, and regularized
  {G}auss--{S}eidel methods.
\newblock {\em Mathematical Programming}, 137(1-2):91--129, 2013.

\bibitem{BCC12}
J.-B. Baillon, P.~Combettes, and R.~Cominetti.
\newblock There is no variational characterization of the cycles in the method
  of periodic projections.
\newblock {\em Journal of Functional Analysis}, 262(1):400 -- 408, 2012.

\bibitem{BB96}
H.~H. Bauschke and J.~M. Borwein.
\newblock On projection algorithms for solving convex feasibility problems.
\newblock {\em SIAM Review}, 38(3):367--426, 1996.

\bibitem{BBL99}
H.~H. Bauschke, J.~M. Borwein, and W.~Li.
\newblock Strong conical hull intersection property, bounded linear regularity,
  {J}ameson's property ({G}), and error bounds in convex optimization.
\newblock {\em Mathematical Programming}, 86(1):135--160, 1999.

\bibitem{BT03}
A.~Beck and M.~Teboulle.
\newblock Convergence rate analysis and error bounds for projection algorithms
  in convex feasibility problems.
\newblock {\em Optimization Methods and Software}, 18(4):377--394, 2003.

\bibitem{BGO06}
N.~H. {Bingham}, C.~M. {Goldie}, and E.~{Omey}.
\newblock {Regularly varying probability densities}.
\newblock {\em {Publications de l'Institut Math\'ematique. Nouvelle S\'erie}},
  80:47--57, 2006.

\bibitem{BGT87}
N.~H. Bingham, C.~M. Goldie, and J.~L. Teugels.
\newblock {\em Regular Variation}.
\newblock Encyclopedia of Mathematics and its Applications. Cambridge
  University Press, 1987.

\bibitem{BCR98}
J.~Bochnak, M.~Coste, and M.-F. Roy.
\newblock {\em Real Algebraic Geometry}.
\newblock Springer Science, 1998.

\bibitem{bdlm10}
J.~Bolte, A.~Daniilidis, O.~Ley, and L.~Mazet.
\newblock Characterizations of {{{\L}}ojasiewicz} inequalities: subgradient
  flows, talweg, convexity.
\newblock {\em Transactions of the American Mathematical Society},
  362(6):3319--3363, 2010.

\bibitem{BNPS17}
J.~Bolte, T.~P. Nguyen, J.~Peypouquet, and B.~W. Suter.
\newblock From error bounds to the complexity of first-order descent methods
  for convex functions.
\newblock {\em Mathematical Programming}, 165(2):471--507, 2017.

\bibitem{blt17}
J.~M. Borwein, G.~Li, and M.~K. Tam.
\newblock Convergence rate analysis for averaged fixed point iterations in
  common fixed point problems.
\newblock {\em SIAM Journal on Optimization}, 27(1):1--33, 2017.

\bibitem{BLY14}
J.~M. Borwein, G.~Li, and L.~Yao.
\newblock Analysis of the convergence rate for the cyclic projection algorithm
  applied to basic semialgebraic convex sets.
\newblock {\em SIAM Journal on Optimization}, 24(1):498--527, 2014.

\bibitem{BW81_2}
J.~M. Borwein and H.~Wolkowicz.
\newblock Regularizing the abstract convex program.
\newblock {\em Journal of Mathematical Analysis and Applications}, 83(2):495 --
  530, 1981.

\bibitem{censor81}
Y.~Censor.
\newblock Row-action methods for huge and sparse systems and their
  applications.
\newblock {\em SIAM Review}, 23(4):444--466, 1981.

\bibitem{CS17}
V.~Chandrasekaran and P.~Shah.
\newblock Relative entropy optimization and its applications.
\newblock {\em Mathematical Programming}, 161(1):1--32, 2017.

\bibitem{Ch09}
R.~Chares.
\newblock {\em Cones and interior-point algorithms for structured convex
  optimization involving powers and exponentials}.
\newblock Phd thesis, Universit\'{e} catholique de Louvain, 2009.

\bibitem{CKV21}
C.~Coey, L.~Kapelevich, and J.~P. Vielma.
\newblock Solving natural conic formulations with {H}ypatia.jl.
\newblock {\em ArXiv e-prints}, 2021.
\newblock \href {http://arxiv.org/abs/2005.01136} {\path{arXiv:2005.01136}}.

\bibitem{cpl96}
P.~L. Combettes.
\newblock {\em The convex feasibility problem in image recovery}, volume~95,
  pages 155--270.
\newblock Elsevier, 1996.

\bibitem{C97}
P.~L. Combettes.
\newblock Hilbertian convex feasibility problem: Convergence of projection
  methods.
\newblock {\em Applied Mathematics and Optimization}, 35(3):311--330, 1997.

\bibitem{DY17}
D.~Davis and W.~Yin.
\newblock Faster convergence rates of relaxed {Peaceman-Rachford} and {ADMM}
  under regularity assumptions.
\newblock {\em Mathematics of Operations Research}, 42(3):783--805, 2017.

\bibitem{dt07}
D.~Djurčić and A.~Torgašev.
\newblock Some asymptotic relations for the generalized inverse.
\newblock {\em Journal of Mathematical Analysis and Applications},
  335(2):1397--1402, 2007.

\bibitem{DR56}
J.~Douglas and H.~H. Rachford.
\newblock On the numerical solution of heat conduction problems in two and
  three space variables.
\newblock {\em Transactions of the American Mathematical Society},
  82(2):421--439, 1956.

\bibitem{DLW17}
D.~Drusvyatskiy, G.~Li, and H.~Wolkowicz.
\newblock A note on alternating projections for ill-posed semidefinite
  feasibility problems.
\newblock {\em Mathematical Programming}, 162(1):537--548, 2017.

\bibitem{EH13}
P.~Embrechts and M.~Hofert.
\newblock A note on generalized inverses.
\newblock {\em Mathematical Methods of Operations Research}, 77(3):423--432,
  2013.

\bibitem{FK94}
J.~Faraut and A.~Kor\'{a}nyi.
\newblock {\em Analysis on Symmetric Cones}.
\newblock Oxford Mathematical Monographs. Clarendon Press, Oxford, 1994.

\bibitem{FB08}
L.~Faybusovich.
\newblock Several {J}ordan-algebraic aspects of optimization.
\newblock {\em Optimization}, 57(3):379--393, 2008.

\bibitem{Fr21}
H.~A. Friberg.
\newblock Projection onto the exponential cone: a univariate root-finding
  problem.
\newblock {\em Optimization Online}, Jan. 2021.

\bibitem{gp72}
P.~Gilbert.
\newblock Iterative methods for the three-dimensional reconstruction of an
  object from projections.
\newblock {\em Journal of {T}heoretical {B}iology}, 36(1):105--117, 1972.

\bibitem{HM11}
D.~Henrion and J.~Malick.
\newblock Projection methods for conic feasibility problems: applications to
  polynomial sum-of-squares decompositions.
\newblock {\em Optimization Methods and Software}, 26(1):23--46, 2011.

\bibitem{HLL78}
G.~T. Herman, A.~Lent, and P.~H. Lutz.
\newblock Relaxation methods for image reconstruction.
\newblock {\em Communications of the ACM}, 21(2):152--158, 1978.

\bibitem{IL17}
M.~Ito and B.~F. Louren\c{c}o.
\newblock A bound on the {C}arath\'{e}odory number.
\newblock {\em Linear Algebra and its Applications}, 532:347 -- 363, 2017.

\bibitem{KT19}
M.~Karimi and L.~Tunçel.
\newblock {D}omain-{D}riven {S}olver ({DDS}) {V}ersion 2.0: a {MATLAB}-based
  software package for convex optimization problems in domain-driven form.
\newblock {\em ArXiv e-prints}, 2019.
\newblock \href {http://arxiv.org/abs/1908.03075} {\path{arXiv:1908.03075}}.

\bibitem{K99}
M.~Koecher.
\newblock {\em The {M}innesota Notes on {J}ordan Algebras and Their
  Applications}.
\newblock Number 1710 in Lecture Notes in Mathematics. Springer, Berlin, 1999.

\bibitem{LP98}
A.~S. Lewis and J.-S. Pang.
\newblock Error bounds for convex inequality systems.
\newblock In J.-P. Crouzeix, J.-E. Martinez-Legaz, and M.~Volle, editors, {\em
  Generalized Convexity, Generalized Monotonicity: Recent Results}, pages
  75--110. Springer US, 1998.

\bibitem{LP16}
G.~Li and T.~K. Pong.
\newblock {D}ouglas-{R}achford splitting for nonconvex optimization with
  application to nonconvex feasibility problems.
\newblock {\em Mathematical Programming}, 159(1):371--401, 2016.

\bibitem{LP18}
G.~Li and T.~K. Pong.
\newblock Calculus of the exponent of {K}urdyka--{{\L}}ojasiewicz inequality
  and its applications to linear convergence of first-order methods.
\newblock {\em Foundations of Computational Mathematics}, 18(5):1199--1232,
  2018.

\bibitem{LiLoPo20}
S.~B. Lindstrom, B.~F. Louren\c{c}o, and T.~K. Pong.
\newblock {Error bounds, facial residual functions and applications to the
  exponential cone.}
\newblock {\em ArXiv e-prints}, 2020.
\newblock \href {http://arxiv.org/abs/2010.16391} {\path{arXiv:2010.16391}}.

\bibitem{LM79}
P.~L. Lions and B.~Mercier.
\newblock Splitting algorithms for the sum of two nonlinear operators.
\newblock {\em SIAM Journal on Numerical Analysis}, 16(6):964--979, 1979.

\bibitem{LHN17}
J.~D. Loera, J.~Haddock, and D.~Needell.
\newblock A sampling {K}aczmarz-{M}otzkin algorithm for linear feasibility.
\newblock {\em SIAM Journal on Scientific Computing}, 39(5):S66--S87, 2017.

\bibitem{L19}
B.~F. Louren\c{c}o.
\newblock {Amenable cones: error bounds without constraint qualifications}.
\newblock {\em Mathematical Programming}, 186:1--48, 2021.

\bibitem{LMT15}
B.~F. Louren\c{c}o, M.~Muramatsu, and T.~Tsuchiya.
\newblock Facial reduction and partial polyhedrality.
\newblock {\em SIAM Journal on Optimization}, 28(3):2304--2326, 2018.

\bibitem{LYBV16}
M.~Lubin, E.~Yamangil, R.~Bent, and J.~P. Vielma.
\newblock Extended formulations in mixed-integer convex programming.
\newblock In Q.~Louveaux and M.~Skutella, editors, {\em Integer Programming and
  Combinatorial Optimization}, pages 102--113, 2016.

\bibitem{lt93}
Z.~Luo and P.~Tseng.
\newblock Error bounds and convergence analysis of feasible descent methods: a
  general approach.
\newblock {\em Annals of Operations Research}, 46(1):157--178, 1993.

\bibitem{MC2020}
{{MOSEK} {ApS}}.
\newblock {\em {MOSEK} Modeling Cookbook Release 3.2.3}, 2021.
\newblock URL: \url{https://docs.mosek.com/modeling-cookbook/index.html}.

\bibitem{ms54}
T.~S. Motzkin and I.~J. Schoenberg.
\newblock The relaxation method for linear inequalities.
\newblock {\em Canadian Journal of Mathematics}, 6:393--404, 1954.

\bibitem{NRP19}
I.~Necoara, P.~Richt\'{a}rik, and A.~Patrascu.
\newblock Randomized projection methods for convex feasibility: conditioning
  and convergence rates.
\newblock {\em SIAM Journal on Optimization}, 29(4):2814--2852, 2019.

\bibitem{ochs19}
P.~Ochs.
\newblock Unifying abstract inexact convergence theorems and block coordinate
  variable metric ipiano.
\newblock {\em SIAM Journal on Optimization}, 29(1):541--570, 2019.

\bibitem{Pang97}
J.-S. Pang.
\newblock Error bounds in mathematical programming.
\newblock {\em Mathematical Programming}, 79(1):299--332, 1997.

\bibitem{DY01}
D.~Papp and S.~Yıldız.
\newblock alfonso: {M}atlab package for nonsymmetric conic optimization.
\newblock {\em ArXiv e-prints}, 2021.
\newblock \href {http://arxiv.org/abs/2101.04274} {\path{arXiv:2101.04274}}.

\bibitem{P13}
G.~Pataki.
\newblock Strong duality in conic linear programming: facial reduction and
  extended duals.
\newblock In {\em Computational and Analytical Mathematics}, volume~50, pages
  613--634. Springer New York, 2013.

\bibitem{rockafellar}
R.~T. Rockafellar.
\newblock {\em {Convex Analysis}}.
\newblock Princeton University Press, 1997.

\bibitem{Se76}
E.~Seneta.
\newblock {\em Regularly Varying Functions}.
\newblock Lecture Notes in Mathematics. Springer Berlin Heidelberg, 1976.

\bibitem{S00}
J.~F. Sturm.
\newblock Error bounds for linear matrix inequalities.
\newblock {\em {SIAM} Journal on Optimization}, 10(4):1228--1248, 2000.

\bibitem{WM13}
H.~Waki and M.~Muramatsu.
\newblock Facial reduction algorithms for conic optimization problems.
\newblock {\em Journal of Optimization Theory and Applications},
  158(1):188--215, 2013.

\bibitem{ww20}
X.~Wang and Z.~Wang.
\newblock The exact modulus of the generalized concave
  {K}urdyka-{{{\L}}ojasiewicz} property.
\newblock {\em Mathematics of Operations Research}, 2022.

\bibitem{YW82}
D.~C. Youla and H.~Webb.
\newblock Image restoration by the method of convex projections: Part 1-theory.
\newblock {\em IEEE Transactions on Medical Imaging}, 1(2):81--94, 1982.

\end{thebibliography}

\end{document}